\documentclass[a4paper,11pt,oneside]{amsart}
\usepackage {amscd}
\usepackage {amsmath}
\usepackage {amssymb}
\usepackage [backend=biber,style=alphabetic,maxalphanames=4,sorting=nyt]{biblatex}
\usepackage {mathrsfs}
\usepackage {geometry}
\usepackage {setspace}
\usepackage {tikz}
\usepackage {tikz-cd}
\usetikzlibrary{arrows, decorations.pathmorphing}
\usepackage {float}
\usepackage {indentfirst}
\usepackage {enumitem}
\usepackage {amsthm}
\usepackage {hyperref}
\usepackage[capitalise]{cleveref}
\usepackage {mathabx}
\usepackage {diagbox}
\usepackage {graphicx}
\usepackage{longtable}
\usepackage{etoolbox}
\usepackage{multirow}
\usepackage{adjustbox}
\preto\tabular{\setcounter{magicrownumbers}{0}}
\preto\longtable{\setcounter{magicrownumbers}{0}}
\newcounter{magicrownumbers}

\graphicspath{ {./images/} }

\theoremstyle{plain}
\newtheorem{theorem}{Theorem}[section]
\newtheorem*{theorem*}{Theorem}
\newtheorem{proposition}[theorem]{Proposition}
\newtheorem{corollary}[theorem]{Corollary}
\newtheorem{lemma}[theorem]{Lemma}
\newtheorem{conjecture}[theorem]{Conjecture}

\theoremstyle{definition}
\newtheorem{definition}[theorem]{Definition}
\newtheorem{proposition-definition}[theorem]{Proposition-Definition}

\newtheorem{convention}[theorem]{Convention}

\theoremstyle{remark}
\newtheorem{example}[theorem]{Example}
\newtheorem{construction}[theorem]{Construction}
\newtheorem{statement}[theorem]{}
\newtheorem{algorithm}[theorem]{Algorithm}
\newtheorem*{remark}{Remark}

\newcolumntype{C}{>{$}c<{$}} 

\title{Geography of Landau-Ginzburg models and threefold syzygies}

\author{Yang He\textsuperscript{1}}
\email{heyang@bimsa.cn}
\author{Artan Sheshmani \textsuperscript{1,2,3}}
\email{artan@mit.edu}
\address{\textsuperscript{1}Beijing Institute of Mathematical Sciences and Applications, No. 544, Hefangkou Village, Huaibei Town, Huairou District, Beijing 101408}
\address{\textsuperscript{2}Massachusetts Institute of Technology, IAiFi Institute, 77 Massachusetts Ave, 26-555. Cambridge, MA 02139}
\address{\textsuperscript{3}National Research University, Higher School of Economics, Russian Federation, Laboratory of Mirror Symmetry, NRU HSE, 6 Usacheva str., Moscow, Russia, 119048}

\begin{document}
\begin{abstract}
    We study the behavior of toric Landau-Ginzburg models under extremal contraction and minimal model program. We also establish a relation between the moduli space of toric Landau-Ginzburg models and the geography of central models. We conjecture that there is a correspondence between extremal contractions and minimal model program on Fano varieties, and degenerations of their associated toric Landau-Ginzburg models written explicitly. We prove the conjectures for smooth toric varieties, as well as general smooth Fano varieties in dimensions 2 and 3. As an application, we compute the elementary syzygies for smooth Fano threefolds.
\end{abstract}

\maketitle


\tableofcontents

\section{Introduction}


\subsection{Background}

\subsubsection{The Sarkisov Program}

The study of the birational automorphism groups of the projective spaces $\mathbb{P}^{n}$ dates back to the 19th century, with significant contributions from the Italian school of algebraic geometry and notable figures like Max Noether. These groups were known as Cremona groups after Luigi Cremona, who initially introduced them. One of the most famous theorems during this period is the classical theorem of Noether and Castelnuovo:

\begin{theorem*}[Noether-Castelnuovo]
  The birational automorphism group of $\mathbb{P}^2$ is generated by the regular automorphism group and a single quadratic transformation.
\end{theorem*}

In the 1980s, the minimal model program (MMP) was introduced, expanding the study of Cremona groups to include birational maps between Mori fibre spaces. Sarkisov, in his work \cite{Sarkisov89}, introduced the notion of \emph{elementary links} between Mori fibre spaces, and proposed a proof that every birational transformation between threefold Mori fibre spaces can be factorized into these elementary links. More precisely, he invented a procedure to explicitly untwist any birational map between polarized Mori fibre spaces by an elementary link and introduced an invariant, the \emph{Sarkisov degree}, that \emph{strictly decreases} (with some exceptions) in this procedure. He stated that this procedure must terminate after finitely many steps. This procedure was subsequently named after Sarkisov and is known as the \emph{Sarkisov program}.

The first detailed proof of the strong Sarkisov program in dimension 3 was provided by Corti:
\begin{theorem*}[\cite{Cor}]
  The Sarkisov program holds in dimension 3.
\end{theorem*}

    Instead of looking into Sarkisov's construction and tracing the Sarkisov degree, Hacon and McKernan used the theory of the geography of log models to prove the existence of factorization into Sarkisov links. More explicitly:
\begin{theorem*}[\cite{HM}]
  Every birational map $f: X/S \dashrightarrow Y/T$ between Mori fibre spaces is a composition of Sarkisov links.
\end{theorem*}

    However, the Sarkisov degree may not necessarily decrease for such factorization. We refer to this result as \emph{the weak Sarkisov program}.

\subsubsection{The theory of syzygies of Mori fibre spaces}

  On June 21st, 2011, Vyacheslav Shokurov delivered a lecture on mobility and rigidity in birational algebraic geometry at RIMS, at the conference ``MMP and extremal rays'' dedicated to Mori’s 60th birthday. In his lecture, he introduced the notion of \emph{central models} as a gift to Mori on the occasion. Actually, we have
    $$
      \{\text{central models of rank }1 \}  \leftrightarrow \{ \text{Mori fibre spaces}\},
    $$
    which is one of the central objects of MMP and modern birational geometry. Moreover, it was established in \cite{SC}, Corollary 6.12 (also p. 527) that
    $$
      \{\text{central models of rank }2 \}  \leftrightarrow \{ \text{Sarkisov links}\},
    $$
    which means that every central model of rank 2 corresponds to a Sarkisov link, up to a choice of direction. Conversely, every Sarkisov link has a center. In the paper \cite{SC} was also established the weak Sarkisov theory: any two Mori models of a given variety are related by a composition of Sarkisov links (\cite[Theorem 7.2]{SC}). The prove uses the geography approach of \cite[Example 2.11]{IskSh} up to perturbation. The same idea was used in \cite{HM} where many technical details from MMP needed for log geography were added, in particular, flips and termination with scaling.

    As a continuation of \cite{SC}, in the talk of Shokurov it was stated that
    $$
      \{\text{central models of rank }3 \}  \leftrightarrow \{ \text{elementary
    relations of Sarkisov’s links}\}.
    $$ 
    This means that any relation for Sarkisov’s links can be decomposed into elementary relations (and trivial relations such as $aa^{-1} = 1$). This follows from the geography of log models. The paper \cite{Ka13} used the ideas presented in this lecture.

    Finally, Shokurov conjectured the following: 
    $$
      \{\text{central models of higher rank}\}  \leftrightarrow \{ \text{higher relations (syzygies) of Mori fibre spaces}\}.
    $$

This conjectured program was solved in \cite{Myself_Syzygy} where we have the following theorems:

  \begin{theorem*}[Homotopical syzygies of $Y/R$]\label{homotopical syzygy}
    There exists a regular (possibly infinite dimensional) CW complex $CW = CW(Y/R)$ satisfying the following:
  \begin{enumerate}[label=(\arabic*),align=left]
    \item \emph{Correspondence:} There is a one-to-one correspondence:
     $$
    \{\text{cells of }CW \text{ of dimension }d \}  \leftrightarrow \{ \text{central models of rank } d + 1\text{ of }Y/R\}.
     $$
     For a central model $X/Z$, the corresponding cell will be denoted by $C(X/Z)$.
    \item \emph{Order-preserving:} For any two central models $X_{1}/Z_{1}$ and $X_{2}/Z_{2}$ of $Y/R$, we have
    $$
    X_{1}/Z_{1} \preceq X_{2}/Z_{2} \Leftrightarrow C(X_{1}/Z_{1}) \subseteq C(X_{2}/Z_{2}).
    $$
    \item \emph{Contractibility:} If $CW \neq \emptyset$, then $CW$ is contractible.
    \item \emph{Action of birational automorphisms:} There is an action of $G$ on $CW$ such that for any central model $X/Z$ of $Y/R$ and element $\sigma \in G$, we have
        $$
        \sigma(C(X/Z)) = C(\sigma(X/Z)).
        $$
         (cf. \cref{def:birational model} for the action of $G$ on models).
  \end{enumerate}
  \end{theorem*}

\begin{theorem}[Homological syzygies of $Y/R$]\label{homological syzygy}
  Assume that $CW \neq \emptyset$. Then we have the following exact sequence of $\mathbb{Z}[G]$-modules:
  $$
  \cdots \rightarrow B_{2} \xrightarrow{\partial_{2}} B_{1} \xrightarrow{\partial_{1}} B_{0} \rightarrow \mathbb{Z} \rightarrow 0,
  $$
  where $B_{d}$ is freely generated by the central models of rank $r=d+1$ of $Y/R$ as an abelian group.
  \end{theorem}

  We mention the following corollary:

  \begin{corollary}[Spectral sequence of birational automorphism group]\label{spectral sequence from syzygy}
  Notation as above. Then we have the following spectral sequence:
  $$
  E_{i,j}^{2} = H_{i}(H_{j}(G,B_{\bullet})) \Rightarrow H_{i+j}(G,\mathbb{Z}),
  $$
  where $B_{d}$ is freely generated by central models of rank $r = d+1$ of $Y/R$ as an abelian group.
  \end{corollary}

    Roughly speaking, the above corollary provides a practical algorithm to compute the discrete group homology $H_{i}(G^{discrete},\mathbb{Z})$ of the birational automorphism group from the following data:
    \begin{enumerate}[align=left,label=(\arabic*)]
    \item The classification of the central models of $Y/R$.
    \item The order between the central models. 
    \item The discrete group homology of the (orientation and fibre preserving) biregular automorphism groups of the central models.
    \end{enumerate}

\subsubsection{Known results and open problems}

    It remains a widely open problem to compute the above data for syzygies. Here, we present some current known cases and some expected progress in dimension 2 and 3.

\hfill 

\paragraph{Classification of central models}\label{paragraph: classification of central models}

    The classifications of central models are known in dimension 2. In dimension 3, there are 3 types of central models:
    \begin{enumerate}
        \item Terminal Fano threefolds (or $\mathbb{Q}$-Fano threefolds). The classification of the smooth case was completed in \cite{Isk77}, \cite{Isk78} and \cite{MoriMukai}. The semi-stable $\mathbb{Q}$-Fano threefolds of Picard number 1 were completed by Miles Reid and his collaborators, cf. \cite{ABR}, \cite{BS07a}, \cite{BS07b}.
        \item Del-Pezzo fibrations over $\mathbb{P}^1$. \cite{Manetti} and \cite{HaPro} studied the local behavior when the general fibre satisfies $K^{2} \geq 5$.
        \item $\mathbb{P}^1$ fibrations over surfaces. Such central models are called $\mathbb{Q}$-conic bundles, and this case was partially studied in \cite{MoriProkhorov07}, \cite{MoriProkhorov08} and \cite{MoriProkhorov09}. In particular, they showed that the base has only Du Val singularities of type \textbf{A}, and they classified the possible local germs.
    \end{enumerate}

    The incompleteness of the classification of threefold central models appears to be a major obstruction in the study of the Sarkisov program in dimension 3. In fact, terminal Fano threefolds naturally appear even in Sarkisov links with a smooth center, see \cref{case 2.28} for an example. We want to describe these terminal Fano threefolds appearing in elementary syzygies as detailed as possible. A natural way is to consider some invariants of terminal Fano threefolds, for instance:
    \begin{enumerate}
        \item The invariants coming from the singularities of $X$. Terminal Fano threefold singularities were classified in \cite{Mori1985}. For each smooth Fano threefold, the singularities of its extremal contractions are studied in \cite{Matsuki1995}.
        \item The ($\mathbb{Q}$-Gorenstein) deformation invariants of $X$. For smooth Fano threefolds, many invariants are known to be deformation invariant. For instance, the invariants related to anti-pluricanonical systems, quantum periods and toric Landau-Ginzburg models. Some invariants also work in general, like the anti-canonical volume $(-K_{X})^3$, quantum periods and toric Landau-Ginzburg models.
    \end{enumerate}
    
    Many natural invariants in birational geometry are not constant in a $\mathbb{Q}$-Gorenstein family. For instance, the Hodge numbers $h^{p,q}$ may jump. On the other side, the invariants constructed from mirror symmetry, such as quantum periods and toric Landau-Ginzburg models, are naturally deformation invariant. Moreover, it's conjectured in \cite{CKPT} that toric Landau-Ginzburg models are even \emph{complete invariants} for $\mathbb{Q}$-Gorenstein deformation families:

\begin{conjecture}[cf. \cite{CKPT}, Conjecture 5.1]\label{conjecture: Q-Fano variety and LG model}
    Rigid maximally mutable Laurent polynomials in $n$ variables (up to mutation) are in one-to-one correspondence with pairs $(X,D)$, where $X$ is a Fano $n$-fold of class TG (that is, it admits a $\mathbb{Q}$-Gorenstein degeneration with reduced fibres to a normal toric variety) with terminal locally toric $\mathbb{Q}$-Gorenstein-rigid singularities and $D \in |-K_{X}|$ is a general elephant (up to $\mathbb{Q}$-Gorenstein deformation).
\end{conjecture}

    In other words, every ($\mathbb{Q}$-Gorenstein) deformation family of terminal Fano threefolds has different toric Landau-Ginzburg models. The threefold case of \cref{conjecture: Q-Fano variety and LG model} has been partially studied in \cite{coates2022mirrorsymmetrylaurentinversion}.

\hfill

\paragraph{The order between central models}

    There are 2 problems related to the order between central models:

\begin{enumerate}
    \item For a fixed central model $X/Z$, determine the central models $X'/Z'$ under it.
    \item When $X/Z$ moves in a moduli space, determine how $X'/Z'$ moves in the corresponding moduli space.
\end{enumerate}

    For surfaces, the first problem is well-known. However, the second problem is complicated to describe even for surfaces. In addition, we don't have the answer to both problems for threefolds. Some rank 2 cases for the first problem, namely, smooth weak central models of rank 2 associated to Sarkisov links of type II, were partially classified in \cite{CutroneMarshburn}.

    One of the main purposes for this article is to solve the first problem for terminal Fano threefolds. As discussed in \ref{paragraph: classification of central models}, we can study the behavior the invariants of these models under the minimal model program.

\hfill

\subsubsection{Birational geometry in mirror symmetry}

The idea of using the theory of Landau-Ginzburg models to study the Sarkisov Program with deformations comes from \cite{CKP}. They conjectured that the birational geometry of a Fano variety, in particular the Sarkisov program, can be studied from the moduli space of Landau-Ginzburg models. The expectation is that if we associate every Fano variety with a moduli of Landau-Ginzburg models, and consider the total space gluing up all such moduli spaces, then Sarkisov links should correspond to certain wall-crossing phenomena.

\subsection{Main results}

    We present our main results in the following way: 
    \begin{itemize}
        \item Firstly, we propose the general conjectures;
        \item Secondly, we list the cases in which these conjectures are known.
    \end{itemize}
    
\subsubsection{Conjectures}

\hfill

\paragraph{Mirror correspondence}

    First, we examine how Landau-Ginzburg models change under extremal contractions.

\begin{conjecture}\label{conj: mirror of extremal contractions}
    Let $X$ be a Fano variety with terminal singularities. Let $f$ be a toric Landau-Ginzburg model of $X$.
    \begin{enumerate}
        \item Let $g: X \rightarrow X'$ be a divisorial contraction and $f'$ be a Landau-Ginzburg model of $X'$. Then there exists an associated degeneration from $f$ to $f'$.
        \item Let $h: X \rightarrow Z$ be a Mori fibre space, $F$ be a general fibre of $h$, and $f'$ be a toric Landau-Ginzburg model of $F$. Then there exists an associated degeneration from $f$ to $f'$.
    \end{enumerate}
\end{conjecture}

    Next, we examine how Landau-Ginzburg models change under the minimal model program.

\begin{conjecture}\label{conj: mirror of MMP}
    Let $X$ be a Fano variety with terminal singularities and $D$ be an effective divisor on $X$. Let $f$ be the Landau-Ginzburg model of $X$. Let $X'$ be a $D$-minimal model of $X$ and $Z$ an ample model of $D$. Then
    \begin{enumerate}
        \item The minimal degeneration $f_{min,D}$ of $f$ in the direction of $D$ is the Landau-Ginzburg model of $X'$.
        \item The maximal degeneration $f_{max,D}$ of $f$ in the direction of $D$ is the Landau-Ginzburg model of the general fibre of $X' \rightarrow Z$.
    \end{enumerate}
\end{conjecture}

\hfill

\paragraph{Geography of LG models}

    Now we examine the moduli space of Landau-Ginzburg models. First we examine the moduli space of toric Landau-Ginzburg models for a fixed Fano variety $X$, which is constructed from the geography of $X$. We think of it as a natural compactification of the parametrized toric LG models defined in \cite{DHKOP}.

\begin{conjecture}[Local geography]\label{conj: local geography}
    Let $X$ be a terminal Fano variety. Let $\Sigma$ be the fan given by the geography on $\mathrm{Eff}_{\mathbb{R}}(X)$, and $X_{\Sigma}$ be the associated toric variety. Then $X_{\Sigma}$ is a moduli space of toric Landau-Ginzburg models of all central models under $X$, which is compatible with the geography structure on $\mathrm{Eff}_{\mathbb{R}}(X)$. More precisely, we have the following correspondence:
    $$
    \begin{tikzcd}
        \{\text{Faces of } \Sigma \text{ in the boundary of } \mathrm{Eff}_{\mathbb{R}}(X)\} \arrow[r,<->,"\text{Geography}"] \arrow[d,<->,"\text{Toric correspondence}"] & \{ \text{Central models under } X \} \arrow[d,<->, "\text{LG models}"] \\
        \{ \text{Toruses in }X_{\Sigma} \} \arrow[r,<->,"\text{Parametrization}"] & \{ \text{Moduli space of LG models} \}
    \end{tikzcd}
    $$
\end{conjecture}

    Now we construct a global moduli space which parametrize toric LG models of all central models.

\begin{conjecture}[Global geography]\label{conj: main theorem}
    Let $Y$ be a projective variety. Then there exists a moduli space $M$ that parametrizes toric LG models of all central models of $Y$ under some terminal Fano varieties with the following properties:
    \begin{enumerate}
    \item \emph{Geography:} We can write $M = \bigcup\limits_{X/Z} M(X/Z)$, where $X/Z$ runs through the central models of $X/Z$. Moreover, we have $M(X/Z) \cong X_{\Sigma}$, where $\Sigma$ is the fan given by the geography on the cone of effective divisors $\mathrm{Eff}_{\mathbb{R}}(X/Z)$ and $X_{\Sigma}$ is the toric variety associated to $\Sigma$. The moduli space $M(X/Z)$ parametrized toric LG models of all central models under $X/Z$.
    \item \emph{Syzygies theory:} The space $M$ is contractible.
    \item \emph{Action of birational automorphism with Zariski topology:} There is an action of the birational automorphism group $\mathrm{Bir}(Y/R)$ on $M$, which respects the Zariski topology of $\mathrm{Bir}(Y/R)$.
    \end{enumerate}
\end{conjecture}

\subsubsection{List of known cases}

    First, in the smooth toric case we have:

\begin{theorem}
    \cref{conj: mirror of extremal contractions} and \cref{conj: mirror of MMP} hold for smooth toric varieties.
\end{theorem}

    We also show that in the surface case we have:

\begin{theorem}
    \cref{conj: mirror of extremal contractions} and \cref{conj: mirror of MMP} hold for surfaces.
\end{theorem}

    Also, we have a weak version of \cref{conj: main theorem} in the surface case:

\begin{theorem}[cf. \cref{theorem: surface geography}]
    There exists a moduli space $M'_{2}$ parametrizing LG models of all central models under some smooth del Pezzo surfaces.
\end{theorem}

    In dimension 3, we have:

\begin{theorem}[cf. \cref{theorem: threefold divisorial contraction and LG models}, \cref{theorem: threefold Mori fibre spaces and LG models}]
    \cref{conj: mirror of extremal contractions} holds for smooth Fano threefolds which admit a parametrized toric Landau-Ginzburg model.
\end{theorem}

\subsubsection{Applications to computations}

    In \cref{section: applications}, we compute some elementary syzygies in dimension 3 for smooth Fano threefolds. Here we collect some interesting cases, which are difficult to deal with in the past. First, we have varieties of Gorenstein index 2.

\begin{example}[Divisorial contractions of type E5]
\hfill

    \begin{enumerate}
        \item Mori-Mukai No. 3.9 (cf. \cref{case 3.9}): Let $X$ be a smooth Fano variety in the Mori-Mukai family No. 3.9. Then there is a diagram
        $$
        \begin{tikzcd}
            & X \arrow[dl]\arrow[dr] & \\
            Y_{1} \arrow[dr] & & Y'_{1} \arrow[dl]\\
            & Y_{2} &
        \end{tikzcd}
        $$
        where $X \rightarrow Y_{1}$ and $X \rightarrow Y'_{1}$ are two divisorial contractions of type E5, where $Y_{1}$ and $Y'_{1}$ have a quotient singularity of the form $\frac{1}{2}(1,1,1)$. A toric LG model of $Y_{1}$ and $Y'_{1}$ is given by
    $$
    f_{Y_{1}} = f_{Y'_{1}} = x + y + z + xy + xz + \frac{1}{yz} + \frac{2}{xyz} + \frac{1}{x^{2}yz}.
    $$
    The variety $Y_{2}$ has two quotient singularities of the form $\frac{1}{2}(1,1,1)$. A toric LG model of $Y_{2}$ is given by
    $$
    f_{Y_{2}} = y + z + xy + xz + \frac{1}{yz} + \frac{2}{xyz} + \frac{1}{x^{2}yz}.
    $$
        \item Mori-Mukai No. 3.14 (cf. \cref{case 3.14}): Let $X$ be a smooth Fano variety in the Mori-Mukai family No. 3.14. Then we there is a sequence of divisorial contractions $X \rightarrow Y \rightarrow Y'$,     where $Y$ and $Y'$ have a quotient singularity of the form $\frac{1}{2}(1,1,1)$. A toric LG model of $Y$ is given by
    $$
    f_{Y} = x+y+z+xyz+\frac{1}{xy}+\frac{1}{yz}+\frac{1}{xz}
    $$
    and a toric LG model of $Y'$ is given by
    $$
    f_{Y'} = x+y+z+\frac{1}{xy}+\frac{1}{yz}+\frac{1}{xz}.
    $$
        \item Mori-Mukai No. 3.22 (cf. \cref{case 3.22}): Let $X$ be a smooth Fano variety in the Mori-Mukai family No. 3.22. Then there is a divisorial contraction $X \rightarrow Y$, where $Y$ has a singularity of type $\frac{1}{2}(1,1,1)$. A toric LG model of $Y$ is given by
    $$
    f_{Y} = x+y+z+\frac{z}{x}+\frac{1}{yz}+\frac{1}{xyz}.
    $$
    \end{enumerate}
\end{example}

    We can also recalculate some Gorenstein cases, which were firstly computed in \cite{galkin2018smalltoricdegenerationsfano} using principal invariants.

\begin{example}[A Gorenstein Fano 3-fold in elementary syzygies, cf. \cref{case 4.2}]
    Let $X$ be a smooth Fano variety in the Mori-Mukai family No. 4.2. Then there is a diagram
        $$
        
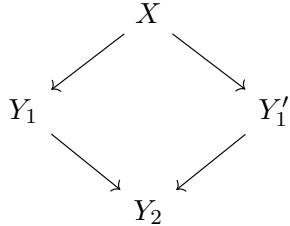
\begin{tikzcd}
            & X \arrow[dl]\arrow[dr] & \\
            Y_{1} \arrow[dr] & & Y'_{1} \arrow[dl]\\
            & Y_{2} &
        \end{tikzcd}
        $$
        where $Y_{1}$ and $Y'_{1}$ have an ordinary double point. There is a smoothing of $Y_{1}$ and $Y'_{1}$ to a smooth Fano variety in the Mori-Mukai family No. 2.23. The Fano variety $Y_{2}$ has two ordinary double points. There is a smoothing of $Y_{2}$ to a smooth Fano variety $B_{4}$.
\end{example}

\subsection{Outline of the paper}
    The remainder of this paper is structured as follows:
    
    In \cref{section: preliminaries}, we introduced some preliminaries, including central models and toric Landau-Ginzburg models.

    In \cref{section: framework}, we construct the main framework of the program.

    In \cref{section: toric}, we prove results for smooth toric Fano varieties.

    In \cref{section: surface}, we prove results for general smooth del Pezzo surfaces.

    In \cref{section: threefold}, we prove results for general smooth Fano threefolds.

    Finally, in \cref{section: applications}, we applied the results in \cref{section: threefold} to compute elementary syzygies of smooth Fano threefolds.

\subsection{Acknowledgments}
The second author is supported by grants from Beijing Institute of Mathematical Sciences and Applications (BIMSA), the Beijiing NSF, and the China National Science Foundation (NSFC). He would also like to thank China's National Program of Overseas High Level Talent for generous support.
    
\section{Preliminaries}\label{section: preliminaries}

\subsection{Central models}

    In this subsection, we list the definitions and basic properties of central models and related topics. A more detailed explanation can be found in \cite{Myself_Syzygy}. We begin by clarifying what we mean by ``(birational) models''.

    \begin{definition}[Birational models]\label{def:birational model}
    Let $\pi:Y \rightarrow R$ be a morphism. A \emph{(birational) model of $Y/R$ (of morphism type)} is an \emph{equality class} of the following data:
    \begin{itemize}
    \item \textbf{Data:} a morphism $f: X \rightarrow Z$ together with a birational map $g: X \dashrightarrow Y$ and a morphism $h: Z \rightarrow R$ such that we have the following commutative diagram:
    $$
    \begin{tikzcd}
    Y  \arrow["\pi"]{dd} & X \arrow[dashed,"g"]{l} \arrow["f"]{d}\\
                            & Z \arrow["h"]{ld}\\
    R                              &
    \end{tikzcd}.
    $$
    
    In particular, for two models $X_{1}/Z_{1}$ and $X_{2}/Z_{2}$ of $Y/R$, we have a natural birational map $g_{2}^{-1}\circ g_{1} : X_{1} \dashrightarrow X_{2}$, where $g_{1} : X_{1} \dashrightarrow Y$ and $g_{2}: X_{2} \dashrightarrow Y$ are the birational maps in the data.
    
    \item \textbf{Equality:} Two models $X_{1}/Z_{1}$ and $X_{2}/Z_{2}$ are called \emph{equal} if there exists an isomorphism $Z_{1} \rightarrow Z_{2}$ such that we have the following commutative diagram:
    $$
    \begin{tikzcd}
    & Y   & \\
    X_{1} \arrow[ru,dashed,"g_{1}"] \arrow[rr,"g_{2}^{-1}\circ g_{1}",crossing over]\arrow[d,"f_{1}"]  & & X_{2}\arrow[d,"f_{2}"] \arrow[lu,dashed,"g_{2}"] \\
    Z_{1}\arrow[rd,"h_{1}"]  \arrow[rr,crossing over] & & Z_{2} \arrow[ld,"h_{2}"] \\
    & R \arrow[from=1-2,"\pi"] &
    \end{tikzcd}
    $$
    where the horizontal arrows are isomorphisms. A birational model is an equality classes of the above data.

    The collection of birational models of morphism type forms a set. Indeed, up to isomorphism, a rational map is determined by its underlying map on topological space and the local morphism on structure sheaves, both of which form a set. Hence the birational maps $g^{-1}:Y \dashrightarrow X$ together with the rational maps $f\circ g^{-1}: Y \dashrightarrow Z$ form a set. Denote by $\mathrm{BirMod}(Y/R)$ the set of birational models of $Y/R$ of morphism type.
    \item \textbf{Action of birational automorphisms:} Denote by $G = \mathrm{Bir}(Y/R)$ the group of birational automorphisms of $Y/R$. Then there is a natural action of $G$ on $\mathrm{BirMod}(Y/R)$. More precisely, for every $\sigma \in G$, we define the action to be:
    $$
    \sigma \left( \begin{tikzcd}
    Y  \arrow["\pi"]{dd} & X \arrow[dashed,"g"]{l} \arrow["f"]{d}\\
                             & Z \arrow["h"]{ld}\\
    R                              &
    \end{tikzcd}\right)
    = \left( \begin{tikzcd}
           Y  \arrow["\pi"]{dd} & X \arrow[dashed,"\sigma \circ g"]{l} \arrow["f"]{d}\\
                             & Z \arrow["h"]{ld}\\
    R                              &
       \end{tikzcd} \right)
    $$
    \end{itemize}
    \end{definition}

    Now we introduce two types of contractions firstly proposed by Shokurov, and a type of contractions firstly appeared in \cite{BLZ}. These definitions are closely related.

\begin{definition}[Central models]
A projective contraction $X \longrightarrow Z$ between normal quasi-projective varieties is called a \emph{central model} if

\begin{enumerate}[label=(CM\arabic*),align=left]
  \item $dim\ X\ >\ dim\ Z\ \geq\ 0$;
  \item $-K_{X}$ is ample over $Z$;
  \item $X$ has only terminal singularities.
\end{enumerate}

The \emph{rank} of a central model is its relative class number $r=\mathrm{Rank}\ \mathrm{Cl}(X/Z)$. Fix a morphism $Y/R$, we say that $X/Z$ is a \emph{central model of $Y/R$}, if $X/Z$ is a birational model of $Y/R$ and is a central model. A central model of rank 1 is actually a \emph{Mori model}.

\end{definition}

\begin{definition}[Weak central models]

A projective contraction $X \longrightarrow Z$ between normal quasi-projective varieties is called a \emph{weak central model} if

\begin{enumerate}[label=(WCM\arabic*),align=left]
  \item $dim\ X\ >\ dim\ Z\ \geq\ 0$;
  \item $-K_{X}$ is semiample over $Z$, and the contraction $X \longrightarrow X'/Z$ given by $-K_X$ over $Z$ is small;
  \item $X$ is $\mathbb{Q}$-factorial with only terminal singularities.
\end{enumerate}

The \emph{rank} of a weak central model is its relative Picard number $r = \rho(X/Z)$. Fix a morphism $Y/R$, we say that $X/Z$ is a \emph{weak central model of $Y/R$}, if $X/Z$ is a birational model of $Y/R$ and is a weak central model.

\end{definition}

    The following models were first introduced in \cite{BLZ}, where they were called ``rank $r$ fibrations''.

\begin{definition}[Intermediate models, cf. Definition 3.1, \cite{BLZ}]\label{def:rank r fibration}
A projective contraction $X \longrightarrow Z$ between normal quasi-projective varieties is called an \emph{intermediate model} if:

\begin{enumerate}[label=(IM\arabic*),align=left]
  \item $dim\ X\ >\ dim\ Z\ \geq\ 0$;
  \item $X/Z$ has Fano type. In particular, $-K_{X}$ is big over $Z$;
  \item $X$ is $\mathbb{Q}$-factorial with terminal singularities;
  \item For any divisor $D$ on $X$, every model in the $D$-MMP of $X$ over $Z$ is still $\mathbb{Q}$-factorial and terminal.
\end{enumerate}

  Here a \emph{$D$-MMP over $Z$} means a sequence of elementary $D$-negative divisorial contractions and $D$-flips over $Z$. The outcome of a $D$-MMP over $Z$ is either a model $X'/Z$ such that $D_{X'}$ is nef over $Z$, or a $D$-negative Mori fibration $X'/Z'$ over $Z$. The \emph{rank} of an intermediate model is its relative Picard number $r = \rho(X/Z)$. Fix a morphism $Y/R$, we say that $X/Z$ is an \emph{intermediate model of $Y/R$}, if $X/Z$ is a birational model of $Y/R$ and is an intermediate model of rank $r$.
\end{definition}

    The order between birational models is defined as follows:
    
\begin{definition}[Order between models]
    Let $Y/R$ be a morphism. Let $X_{1}/Z_{1}$ and $X_{2}/Z_{2}$ be two birational models of $Y/R$, and $g: X_{1} \dashrightarrow X_{2}$ be the natural birational map between them. We say that $X_{1}/Z_{1}$ is \emph{over} $X_{2}/Z_{2}$, or equivalently, $X_{2}/Z_{2}$ is \emph{under} $X_{1}/Z_{1}$, if $g$ is a birational 1-contraction and there exists a morphism $h:Z_{2} \rightarrow Z_{1}$ such that the following diagram commutes:
    $$
    \begin{tikzcd}
      X_{1} \arrow[dashed,"g"]{r} \arrow["f_{1}"]{dd} & X_{2} \arrow["f_{2}"]{d}\\
                                     & Z_{2} \arrow["h"]{ld}\\
      Z_{1}                              &
    \end{tikzcd}
    $$

    We denote this by $X_{1}/Z_{1} \succeq X_{2}/Z_{2}$.
\end{definition}

    In the remaining of this subsection, we list some properties of the above defined models without a proof. For a detailed proof, we refer to \cite{Myself_Syzygy}, Section 2.1. First, we examine the action of $G$ on birational models defined in \cref{def:birational model}.
    
\begin{proposition}[Action on models]\label{prop:action on models}
    Let $G = \mathrm{Bir}(Y/R)$ and $\mathrm{BirMod}(Y/R)$ be as in \cref{def:birational model}. The action of $G$ on $\mathrm{BirMod}(Y/R)$ has the following properties:
    \begin{enumerate}[align=left, label=(\arabic*)]
    \item \textbf{Stabilizer:} For any $X/Z \in \mathrm{BirMod}(Y/R)$, the stabilizer subgroup of $X/Z$ is isomorphic to $\mathrm{Aut}(X \rightarrow Z/R)$, the fibrewise regular automorphism group of $X \rightarrow Z$ over $R$. More precisely, denote by $g: X \dashrightarrow Y/R$ the natural birational map, then the stabilizer subgroup of $X/Z$ consists of the elements of $G$ of the form $g\circ \sigma \circ g^{-1}$, where $\sigma \in \mathrm{Aut}(X/R)$ is an element such that there exists an isomorphism $h: Z \rightarrow Z$ making the following diagram commutes:
    $$
    \begin{tikzcd}
      X \arrow["\sigma"]{r} \arrow["f"]{d} & X \arrow["f"]{d}\\
      Z                 & Z \arrow{l}{\simeq}[swap]{h}
    \end{tikzcd}
    $$

    \item \textbf{Orbit:} For any $X/Z \in \mathrm{BirMod}(Y/R)$, the orbit of $X/Z$ is the isomorphism class of $X/Z$. In particular, $G$ maps central models (resp. weak central models, intermediate models) to central models (resp. weak central models, intermediate models) of the same rank.

    \item \textbf{Order-preserving:} $G$ preserves the order of models, that is, for all $\sigma \in G$ we have
    $$
    X_{1}/Z_{1} \preceq X_{2}/Z_{2} \Leftrightarrow \sigma (X_{1}/Z_{1}) \preceq \sigma (X_{2}/Z_{2}).
    $$
  \end{enumerate}
\end{proposition}

    The next proposition shows the relation between the definition of central models, weak central models and intermediate models.

\begin{proposition}\label{equivalence of syzygies}
  Let $r$ be a positive integer.

  \begin{enumerate}[label=(\roman*),align=left]
    \item Every central model $X$ of rank $r$ has a small blow-up $\widetilde{X} \rightarrow X/Z$ which is a weak central model of the same rank $r$, and such weak central model is unique up to small flops. Conversely, every weak central model $\widetilde{X}/Z$ of rank $r$ has a small birational contraction $\widetilde{X} \rightarrow X/Z$ into a central model $X/Z$ of the same rank $r$, such central model is unique up to equality, cf. \cref{def:birational model}.
    \item Every weak central model of rank $r$ is an intermediate model of rank $r$. Conversely, every intermediate model $X/Z$ of rank $r$ is equivalent (cf. \cref{def:birational model}) to a weak central model $X'/Z$ of the same rank $r$, and all such weak central models are equal up to small flops.
    \item Every intermediate model $X/Z$ of rank $r$ is equivalent to a central model $X'/Z$ of the same rank $r$, and such central model is unique up to equality. Conversely, every central model $X$ of rank $r$ has a small blow-up $\widetilde{X} \rightarrow X/Z$ which is a weak central model of rank $r$, and in particular, an intermediate model.
  \end{enumerate}

\end{proposition}

\subsection{Period of Laurent polynomials}

    In this subsection, we define the period of a Laurent polynomial and some properties of it.

\begin{definition}\label{definition: period of Laurent polynomials}
    Let $f$ be a Laurent polynomial. The \emph{period} of $f$ is the formal power series
    $$
    P_{f}(t) = 1 + c(f)t + \frac{1}{2!}c(f^2)t^{2} + \cdots + \frac{1}{d!}c(f^{d})t^{d} + \cdots
    $$
    where $c(f^{d})$ is the constant term of $f^{d}$. The \emph{regularized period} of $f$ is the formal power series
    $$
    \hat{P}_{f} = 1 + c(f)t + c(f^2)t^{2} + \cdots + c(f^{d})t^{d} + \cdots.
    $$
\end{definition}

    The following lemma is useful in this article:

\begin{lemma}\label{lemma: period after shifting constants}
    Let $f$ be a Laurent polynomial and $a \in \mathbb{C}$ be a complex number. Then 
    $$
    P_{f+a}(t) = e^{at} P_{f}(t).
    $$
\end{lemma}

\begin{proof}
    We can write
    $$
    (f+a)^{d} = \sum\limits_{i=0}^{d} \frac{d!}{i!(d-i)!} a^{d-i}f^{i}.
    $$
    Taking the constant terms, we get
    $$
    \frac{c((f+a)^{d})}{d!} = \sum\limits_{i=0}^{d} \frac{a^{d-i}}{(d-i)!} \frac{c(f^{i})}{i!}.
    $$
    Notice that 
    $$
    e^{at} = \sum\limits_{i=0}^{+\infty} \frac{a^{i}}{i!}t^{i}.
    $$
    Hence we can conclude the lemma by comparing each term.
\end{proof}

\subsection{Mirror symmetry}

    The main goal of this article is to use the tool of mirror symmetry, especially Landau-Ginzburg models, to study the Sarkisov program. In this subsection, we introduce the definition and properties of Landau-Ginzburg models. First, we set up the data we need to define Landau-Ginzburg models. Now we recall the definition of Gromov-Witten invariants and regularized I-series. One can find details in \cite{Manin}, Chapter VI, \S 2.1 and \cite{Przyjalkowski_2018}, \S 2.1.

\begin{definition}
    The moduli space of stable maps to $X$ of rational curves of class $\beta \in H_{2}(X)$ with $n$ marked points (we denote it by $\overline{M}_{n}(X,\beta)$) is the Deligne-Mumford stack of stable maps $f: C \rightarrow X$ of curves of genus 0 with $n$ marked points such that $f_{*}C = \beta$.
\end{definition}

    We consider the evaluation maps 
    $$
    ev_{i}: \overline{M}_{n}(X,\beta) \rightarrow X, ev(C;p_{1}, \cdots , p_{n},f) = f(p_{i}).
    $$
    Let $\pi_{n+1}: \overline{M}_{n+1}(X,\beta) \rightarrow \overline{M}_{n}(X,\beta)$ be the forgetful map at the point $p_{n+1}$ that forgets this point and contracts the components which then become unstable. Consider the section $\sigma_{i}:\overline{M}_{n}(X,\beta) \rightarrow \overline{M}_{n+1}(X,\beta)$ corresponding to the marked points $p_{i}$ and defined as follows: the image of a curve $(C;p_{1}, \cdots , p_{n},f)$ under the map $\sigma_{i}$ is a curve $(C';p_{1}, \cdots , p_{n}, p_{n+1},f')$ where $C' = C \cup C_{0}$, $C_{0} \simeq \mathbb{P}^1$, $C_{0}$ and $C$ intersect in the unmarked point $p_{i}$ on $C'$, and $p_{n+1}$ and $p_{i}$ lie on $C_{0}$. The map $f'$ contracts the curve $C_{0}$ to a point and $f'|_{C} = f$. Consider the sheaf $L_{i}$ given by $L_i = \sigma^{*}_{i} \omega_{\pi_{n+1}}$, where $\omega_{\pi_{n+1}}$ is the relative dualizing sheaf of $\pi_{n+1}$. Its fibre over the point $(C; p_{1}, . . . , p_n, {f})$ is $T^{*}_{p_i}C$.

\begin{definition}
    The \emph{cotangent line class} is the class
    $$
    \psi_{i}= c_{1}(L_{i}) \in H^{2}(\overline{M}_{n}(X,\beta)).
    $$
\end{definition}

\begin{definition}
    Consider $\gamma_{1}, \cdots, \gamma_{n} \in H^{*}(X)$. Let $a_{1},\cdots, a_{n}$ be non-negative integers, and let $\beta \in H_{2}(X)$. Then the \emph{Gromov-Witten invariant with descendants} is the number given by
    $$
    \langle \tau_{a_{1}}\gamma_{1}, \cdots , \tau_{a_{n}}\gamma_{n} \rangle_{\beta} = \int_{[\overline{M}_{n}(X,\beta)]^{virt}} \psi_{1}^{a_{1}}ev_{1}^{*}(\gamma_{1}) \cdots \psi_{n}^{a_{n}}ev_{n}^{*}(\gamma_{n}).
    $$
\end{definition}

\begin{statement}[Set-up]\label{setup:I-series}
        During this subsection, we set up the following data:
    \begin{enumerate}
        \item Let $X$ be a smooth Fano variety of dimension $n$ and Picard number $r$.
        \item Let $D$ be a $\mathbb{C}$-divisor on $X$.
        \item Let $K \subseteq H_{2}(X,\mathbb{Z})$ be the monoid of classes of moving curves of $X$, that is, $K$ consists of classes $\beta \in H_{2}(X,\mathbb{Z})$ such that the morphism $ev_{1}: \overline{M}_{0,1}(X,\beta) \rightarrow X$ is surjective. In particular, for every class $0 \neq \beta \in K$ we have $K_{X} \cdot \beta < 0$.
        \item Let 
        $$
        \widetilde{I}^{X}_{0}(q=(q_{1},\cdots,q_{r})) = 1 + \sum\limits_{\beta \in K}(-K_{X} \cdot \beta)!\langle \tau_{-K_{X}\cdot\beta -2}\mathbf{1} \rangle_{\beta} \cdot q^{\beta}
        $$
        be \emph{the constant term of regularized $I$-series} or \emph{the regularized quantum period} of $X$. Here the class $\mathbf{1} \in H^{2 \mathrm{dim}X}(X)$ is the fundamental class of $X$ under the Poincare duality.
        \item Let 
        \begin{align*}
        \hat{G}_{X,D}(t) = \widetilde{I}_{0}^{X,D}(t) &= 1 + \sum\limits_{\beta \in K}(-K_{X} \cdot \beta)!\langle \tau_{-K_{X}\cdot\beta -2}\mathbf{1} \rangle_{\beta} \cdot e^{-D \cdot \beta}t^{-K_{X} \cdot \beta}\\
        &= 1 + a_{1}t + a_{2}t^2 + \cdots
        \end{align*} 
        be \emph{the restriction of the constant term of regularized $I$-series of $X$ to the
        anti-canonical direction corresponding to $D$}.
    \end{enumerate}
\end{statement}

    Now we give the definitions of mirrors and Landau-Ginzburg models.

\begin{definition}[Weak Landau-Ginzburg models]
    Let $X$ be a smooth Fano variety. A \emph{weak Landau-Ginzburg model of $X$} is a Calabi-Yau fibration $f: Y \rightarrow \mathbb{A}^1$ satisfying the period condition, i.e., the period of $f$ corresponds to the $I$-series of $X$. In particular, the anti-canonical sections of $X$ are mirrored to the fibres of $f$.
\end{definition}

    The most important and practically calculable cases of Landau-Ginzburg models are the toric Landau-Ginzburg models, which are defined as follows:

\begin{definition}[Toric Landau-Ginzburg models]
    Let $X$ be a smooth Fano variety and $D$ a $\mathbb{C}$-divisor on $X$. A \emph{toric Landau-Ginzburg model} of $(X,D)$ is a Laurent polynomial $f: (\mathbb{C}^{*})^{n} \rightarrow \mathbb{C}$ such that
    \begin{enumerate}
        \item \emph{Period condition.} $\hat{P}_{f}(t) = \hat{G}_{X,D}(t)$.
        \item \emph{Calabi-Yau compactification. } There exists a fiberwise compactification (the so called \emph{Calabi–Yau compactification}) $Y \rightarrow \mathbb{C}$ such that $Y$ is a smooth Calabi-Yau variety.
        \item \emph{Polytope condition. }There is a degeneration $X \leadsto X_{T}$ to a toric variety $X_{T}$ whose fan polytope (the convex hull of generators of its rays) coincides with the Newton polytope (the convex hull of non-zero coefficients) of $f$.
    \end{enumerate}

    A Laurent polynomial $f$ is said to be a \emph{divisorial toric Landau-Ginzburg model of $X$} if there exists a $\mathbb{C}$-divisor $D$ such that $f$ is a toric LG model of $(X,D)$.
\end{definition}

    We sometimes compactified the Calabi-Yau fibration to obtain a Calabi-Yau fibration over $\mathbb{P}^1$ as follows:

\begin{definition}[Tame compactified Landau-Ginzburg models]
    A \emph{proper, tame compactified Landau–Ginzburg model} is a triple $(Z,D,f)$ consisting of a
    smooth projective variety $Z$, a simple normal crossings (snc) divisor $D$, and a morphism $f : Z \rightarrow \mathbb{P}^1$ so that $f^{*}(\infty) = D$. We say that a tame compactified Landau–Ginzburg model satisfies the Calabi–Yau condition if $D$ is an anti-canonical divisor of $Z$. We say that $(Z,D,f)$ is \emph{a tame compactified Landau-Ginzburg model of $X$} if the restriction of $f$ over $\mathbb{A}^1$ is a Calabi-Yau compactification of a toric Landau-Ginzburg model of $X$.
\end{definition}

    Given a Landau-Ginzburg model of $X$, we can deform it by changing its coefficients. However, the new Laurent polynomial after a deformation may not be a Landau-Ginzburg model of $X$. Hence we only consider deformations satisfying certain conditions, namely, divisorial toric LG models.

\begin{definition}[Parametrized toric Landau-Ginzburg model]\label{definition: parametrized toric LG model}
    Let $X$ be a smooth Fano variety of dimension $n$ and Picard number $r$. Fix a basis of $\mathrm{Pic}(X)$ so that we can write $\mathrm{Pic}(X) \cong \mathbb{Z}^{r}$. We say that a Laurent polynomial $p \in \mathbb{Z}[a_{1}^{\pm 1}, a_{2}^{\pm 1}, \cdots, a_{r}^{\pm 1}][x_{1}^{\pm 1}, x_{2}^{\pm 1}, \cdots, x_{n}^{\pm 1}]$ is a \emph{parametrized Landau-Ginzburg model of $X$} if for any choice of the complexified divisor class $D = (\alpha_{1}, \cdots, \alpha_{r}) \in \mathrm{Pic}(X) \otimes \mathbb{C}$ the induced parameter specialization $p_{D} = p|_{a_{i} = e^{-\alpha_{i}}}: (\mathbb{C}^{*})^{n} \rightarrow \mathbb{C}$ is a toric Landau-Ginzburg model for the pair $(X,D)$.
\end{definition}

\begin{remark}
    Although a parametrized toric LG model can be evaluated on the complexified divisor class $\mathrm{Pic}(X) \otimes \mathbb{C}$, we usually restrict it on the real part $\mathrm{Pic}(X) \otimes \mathbb{R}$ for explicit computation.
\end{remark}

\begin{statement}[Action of the automorphism group]
    There is a natural action of the regular automorphism group $\mathrm{Aut}(X)$ on $\mathrm{Pic}(X)$. Hence this induces a natural action of $\mathrm{Aut}(X)$ on a parametrized toric Landau-Ginzburg model by sending the toric Landau-Ginzburg model of $(X,D)$ to the toric Landau-Ginzburg model of $(g(X),g(D))$ for any $g \in \mathrm{Aut}(X)$. This action is independent to the choice of basis of $\mathrm{Pic}(X)$.
\end{statement}

    Now we study the Picard group $\mathrm{Pic}(X_{T})$ under the toric degeneration. Let $\mathcal{X} \rightarrow Z \ni z_{0},z_{1}$ be a degeneration of $\mathcal{X}_{z_{1}} \cong X$ to $\mathcal{X}_{z_{0}} \cong X_{T}$. If the sheaf $U \mapsto \mathrm{Pic}(X_{U})$ is constant, then we say that the degeneration is \emph{small}, and it is easy to obtain the parametrized toric LG model in this case. In general, we can degenerate a divisor $D$ on $X$ as follows:

\begin{lemma}[cf. \cite{Shokurov20}, Section 4]\label{lemma: constant Picard}
     Let $\mathcal{X} \rightarrow Z \ni z_{0},z_{1}$ be a degeneration of $\mathcal{X}_{z_{1}} \cong X$ to $\mathcal{X}_{z_{0}} \cong X_{T}$ and $D_{1},\cdots , D_{l}$ be a finite number of prime divisors on $X$. Denote by $Z^{sm} \subseteq Z$ the open subset of $Z$ where the morphism $\mathcal{X} \rightarrow Z$ is smooth. Then after a finite base change, the sheaf $U \mapsto \mathrm{Pic}(X_{U})$ is constant on $Z^{sm}$ and we can find prime divisors $\mathcal{D}_{1}, \cdots , \mathcal{D}_{l}$ such that $\mathcal{D}_{i}|_{X} = D_{i}$.
\end{lemma}

\begin{construction}[degenerating divisors]\label{construction: degenerating divisors}
    Let $\mathcal{X} \rightarrow Z \ni z_{0},z_{1}$ be a degeneration of $\mathcal{X}_{z_{1}} \cong X$ to $\mathcal{X}_{z_{0}} \cong X_{T}$. We take $D_{1},\cdots , D_{l}$ to be the generators of the effective cone $\mathrm{Eff}_{\mathbb{R}}(X)$. Then by \cref{lemma: constant Picard}, we can take a finite base change $Z' \rightarrow Z$ such that the sheaf $U \mapsto \mathrm{Pic}(\mathcal{X}_{U})$ is constant on $Z^{sm}$ and we can find prime divisors $\mathcal{D}_{1}, \cdots , \mathcal{D}_{l}$ such that $\mathcal{D}_{i}|_{\mathcal{X}_{z'_{1}}} = D_{i}$, where $z'_{1}$ is a closed point over $z_{1}$. We can then choose a closed point $z'_{0}$ over $z_{0}$ and obtain divisors $D_{i,T}$ given by the divisorial part of $\mathcal{D}_{i}|_{\mathcal{X}_{z'_{0}}}$ on $X_{T}$.
\end{construction}

\begin{remark}
    Although any prime divisor can be degenerated, one must pay attention to the following facts:
    \begin{enumerate}
        \item The degeneration is not canonical. It depends on the family $\mathcal{X} \rightarrow Z$, the finite base change $Z'$ and the choice of the closed points $z'_{0}$ and $z'_{1}$.
        \item The restriction $\mathcal{D}_{i}|_{\mathcal{X}_{z'_{0}}}$ may not have any divisorial component if $\mathcal{D}_{i}$ is not $\mathbb{Q}$-Cartier. When $\mathcal{X}$ has klt singularities, one can remedy this problem by taking a $\mathbb{Q}$-factorization $\mathcal{X}' \rightarrow \mathcal{X}$ and degenerating the divisors to a higher model $X'_{T} = \mathcal{X}'_{z'_{0}}$ of $X_{T}$ instead.
        \item The degeneration may not be well-defined up to linear equivalence. See \cref{example: Mori-Mukai 2.15} for an example where the non-rational smooth cubic threefold $B_{3}$ degenerates to a rational terminal Gorenstein threefold $Y$. In particular, they have different principal divisors.
        \item The effective cone $\mathrm{Eff}_{\mathbb{R}}(X_{T})$ may be larger than the effective cone $\mathrm{Eff}_{\mathbb{R}}(X)$. See \cref{Case 3.2} for an example.
    \end{enumerate}
\end{remark}

\subsection{Mutations of Laurent polynomials}

    It is expected that deformations of Fano varieties are mirror to some operations on the Landau-Ginzburg models called mutations. We recall the definition of mutations.

\begin{definition}[Mutations, cf. \cite{CKPT}, Definition 1.6]
    Let $N$ be a lattice and $w \in M$ be a primitive vector in the dual lattice. Then $w$ induces a grading on $\mathbb{C}[N]$. Let $a \in \mathbb{C}[w^{\perp} \cap N]$ be a Laurent polynomial in the zeroth piece of $\mathbb{C}[N]$, where $w^{\perp} \cap N = \{v \in N \mid w(v)=0 \}$. The pair $(w,a)$ defines an automorphism of $\mathbb{C}(N)$ given by
    $$
    \mu_{w,a}: \mathbb{C}(N) \rightarrow \mathbb{C}(N), \qquad x^{v} \rightarrow x^{v}a^{w(v)}.
    $$
    Let $f \in \mathbb{C}[N]$. We say that \emph{$f$ is mutable with respect to $(w,a)$} if 
    $$
    g:=\mu_{w,a}(f) \in \mathbb{C}[N],
    $$
    in which case we call $g$ a \emph{mutation of $f$} and $a$ a \emph{factor}.
\end{definition}

    Here we mention the following basic property of mutations.

\begin{proposition}[cf. \cite{CKPT}, 1.5]
    Let $f$ be a Laurent polynomial. Then the period of $f$ is invariant under mutations.
\end{proposition}

\section{Landau-Ginzburg models in birational geometry}\label{section: framework}

    In this section, we establish several conjectures about the behavior of the corresponding Landau-Ginzburg models of $X$ when running the minimal model program on $X$. First, we recall the conjecture about the existence of toric LG models.

    \begin{conjecture}
        Let $X$ be a terminal Fano variety and $D$ an $\mathbb{R}$(or $\mathbb{C}$)-Cartier divisor on $X$. Then $(X,D)$ has a toric LG model.
    \end{conjecture}

    The expected correspondence between extremal contractions in the minimal model program and the behavior of the corresponding LG models is the following:
\[
\begin{matrix}
  X & & Y \\
  \textrm{flips/flops} & \longleftrightarrow & \textrm{wall-crossings}\\
  \textrm{divisorial contractions} & \longleftrightarrow & \textrm{irreducible degenerations}\\
  \textrm{Mori fibre spaces} & \longleftrightarrow & \textrm{fibring degenerations}\\
  \textrm{degenerations} & \longleftrightarrow & \textrm{mutations}
\end{matrix}
\]

\subsection{Divisorial contractions and degenerations of LG models}

\begin{conjecture}
    Let $f: X \rightarrow Y$ be a divisorial contraction between terminal Fano varieties. Then the LG model of $X$ degenerates to the LG model of $Y$.
\end{conjecture}

\subsection{Mori fibre spaces, fibring degenerations and Tyurin degenerations}

    In this section, we examine the toric Landau-Ginzburg models of $X$ when $X$ admits a non-trivial Mori fibre space $X \rightarrow Z$. The main conjecture in this section is that the Landau-Ginzburg models of a general fibre of $X/Z$ can be obtained by the fibring degenerations of the Landau-Ginzburg models of $X$.

\begin{conjecture}[Fibring degenerations for Mori fibre spaces]
    Let $X$ be a terminal Fano variety, $X \rightarrow Z$ be a Mori fibre space and $F$ be a general fibre. Let $Y \rightarrow \mathbb{A}^1$ be the Calabi-Yau compactification of a toric LG model of $X$ and $Y'$ be the Calabi-Yau compactification of a toric LG model of $F$. Then there exists a degeneration of $Y$ with the following commutative diagram:
    $$
    \begin{tikzcd}
        Y \arrow[d] \arrow[r,decorate,decoration={snake,segment length = 1.6mm,amplitude=0.2mm}] & Y^{*} \arrow[d] \\
        \mathbb{A}^1 & Y' \arrow[l]
    \end{tikzcd}
    $$
\end{conjecture}

    When $X$ admits a Fano fibration to a curve, there is another type of degenerations, which is an analogue of the Tyurin degenerations of the mirrors of Calabi-Yau varieties.

\begin{conjecture}[Tyurin degenerations for Mori fibre spaces]\label{conj: Tyurin degeneration for central models}
    Let $X$ be a terminal Fano variety and $X \rightarrow \mathbb{P}^1$ be a Mori fibre space. Let $Y \rightarrow \mathbb{A}^1$ be the Calabi-Yau compactification of a toric LG model of $X$. Then there exists a degeneration of $Y$ with the following commutative diagram:
    $$
    \begin{tikzcd}
        Y \arrow[d] \arrow[r,decorate,decoration={snake,segment length = 1.6mm,amplitude=0.2mm}] & Y_{0} \cup Y_{1} \arrow[d] \\
        \mathbb{A}^1 \arrow[equal,r] & \mathbb{A}^1
    \end{tikzcd}
    $$
    where
    \begin{enumerate}
        \item All $Y_{i} \rightarrow \mathbb{A}^1$ are Fano fibrations, and they have simple normal crossing intersections;
        \item The intersection $Y_{1} \cap Y_{2}$ is an anti-canonical divisor of $Y_{1}$ and $Y_{2}$.
        \item The intersection $Y_{1} \cap Y_{2} \rightarrow \mathbb{A}^1$ is the mirror of the general fibre of $X \rightarrow Z$.
    \end{enumerate}
    Conversely, suppose that $X$ is a smooth Fano variety and let $Y \rightarrow \mathbb{A}^1$ be an LG model of $X$. If there exists a degeneration of $Y \rightarrow \mathbb{A}^1$ to $Y_{1} \cap Y_{2} \rightarrow \mathbb{A}^1$ satisfying the above conditions, then there exists a fibration $X \rightarrow \mathbb{P}^1$ such that the general fibre has a toric LG model given by $Y_{1} \cap Y_{2} \rightarrow \mathbb{A}^1$.
    \end{conjecture}

\begin{example}[Some Tyurin degenerations for surfaces]\label{example: Tyurin degenerations for surfaces}
    We examine the Tyurin degenerations of the compactified Landau-Ginzburg models of some surfaces.

    An LG model of $\mathbb{P}^2$ is given by $f = x + y + \frac{1}{xy}$. If we compactify $\mathbb{C}^{*2}$ to $\mathbb{P}^2$, the level set $f=c$ can be written as 
    $$
    x^{2}y+xy^{2}+z^3 = cxyz.
    $$
    We think of it as a pencil of cubic curves. We proceed by 2 steps to make the pencil base-point free. The curve $C: x^{2}y+xy^{2}+z^3 = 0$ intersects the line $L_{x}: x=0$ at $[0,1,0]$ with multiplicity 3, the line $L_{y}: y=0$ at $[1,0,0]$ with multiplicity 3, and the line $L_{z}: z=0$ at $[0,1,0],[1,0,0],[-1,1,0]$ with multiplicity 1. The configuration is shown in the following graph:
    
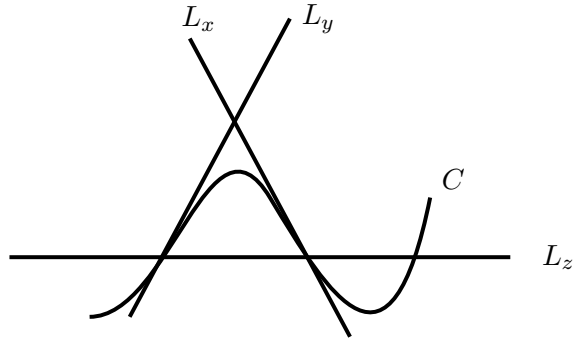
\begin{figure}[H]
        \centering
\ifx\du\undefined
  \newlength{\du}
\fi
\setlength{\du}{15\unitlength}
\begin{tikzpicture}
\pgftransformxscale{0.500000}
\pgftransformyscale{-0.500000}
\definecolor{dialinecolor}{rgb}{0.000000, 0.000000, 0.000000}
\pgfsetstrokecolor{dialinecolor}
\definecolor{dialinecolor}{rgb}{1.000000, 1.000000, 1.000000}
\pgfsetfillcolor{dialinecolor}
\pgfsetlinewidth{0.100000\du}
\pgfsetdash{}{0pt}
\pgfsetdash{}{0pt}
\pgfsetbuttcap
{
\definecolor{dialinecolor}{rgb}{0.000000, 0.000000, 0.000000}
\pgfsetfillcolor{dialinecolor}
\definecolor{dialinecolor}{rgb}{0.000000, 0.000000, 0.000000}
\pgfsetstrokecolor{dialinecolor}
\draw (29.000000\du,7.000000\du)--(21.000000\du,22.000000\du);
}
\pgfsetlinewidth{0.100000\du}
\pgfsetdash{}{0pt}
\pgfsetdash{}{0pt}
\pgfsetbuttcap
{
\definecolor{dialinecolor}{rgb}{0.000000, 0.000000, 0.000000}
\pgfsetfillcolor{dialinecolor}
\definecolor{dialinecolor}{rgb}{0.000000, 0.000000, 0.000000}
\pgfsetstrokecolor{dialinecolor}
\draw (24.000000\du,8.000000\du)--(32.000000\du,23.000000\du);
}
\pgfsetlinewidth{0.100000\du}
\pgfsetdash{}{0pt}
\pgfsetdash{}{0pt}
\pgfsetbuttcap
{
\definecolor{dialinecolor}{rgb}{0.000000, 0.000000, 0.000000}
\pgfsetfillcolor{dialinecolor}
\definecolor{dialinecolor}{rgb}{0.000000, 0.000000, 0.000000}
\pgfsetstrokecolor{dialinecolor}
\draw (15.000000\du,19.000000\du)--(40.000000\du,19.000000\du);
}

\definecolor{dialinecolor}{rgb}{0.000000, 0.000000, 0.000000}
\pgfsetstrokecolor{dialinecolor}
\node[anchor=west] at (36.000000\du,15.050000\du){$C$};
\definecolor{dialinecolor}{rgb}{0.000000, 0.000000, 0.000000}
\pgfsetstrokecolor{dialinecolor}
\node[anchor=west] at (23.000000\du,7.000000\du){$L_{x}$};
\definecolor{dialinecolor}{rgb}{0.000000, 0.000000, 0.000000}
\pgfsetstrokecolor{dialinecolor}
\node[anchor=west] at (29.000000\du,7.000000\du){$L_{y}$};
\definecolor{dialinecolor}{rgb}{0.000000, 0.000000, 0.000000}
\pgfsetstrokecolor{dialinecolor}
\node[anchor=west] at (41.000000\du,19.000000\du){$L_{z}$};

\useasboundingbox (20.000000\du,5.000000\du) rectangle (42.000000\du,20.000000\du);

\pgfsetlinewidth{0.100000\du}
\pgfsetdash{}{0pt}
\pgfsetdash{}{0pt}
\pgfsetmiterjoin
\pgfsetbuttcap
{
\definecolor{dialinecolor}{rgb}{0.000000, 0.000000, 0.000000}
\pgfsetfillcolor{dialinecolor}
\definecolor{dialinecolor}{rgb}{0.000000, 0.000000, 0.000000}
\pgfsetstrokecolor{dialinecolor}
\pgfpathmoveto{\pgfpoint{19.000000\du}{22.000000\du}}
\pgfpathcurveto{\pgfpoint{23.000000\du}{22.000000\du}}{\pgfpoint{25.000000\du}{11.000000\du}}{\pgfpoint{28.000000\du}{16.000000\du}}
\pgfpathcurveto{\pgfpoint{31.000000\du}{21.000000\du}}{\pgfpoint{34.000000\du}{26.000000\du}}{\pgfpoint{36.000000\du}{16.000000\du}}
\pgfusepath{stroke}
}
\end{tikzpicture}

    \end{figure}

    To make the intersections simple normal crossing, we successively blow up the base point on $L_{x}$ and $L_{y}$ 3 times. The configuration of curves is given in the following:

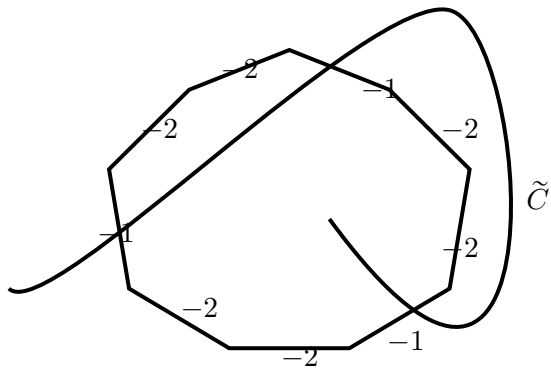
\begin{figure}[H]
  \centering
\ifx\du\undefined
  \newlength{\du}
\fi
\setlength{\du}{15\unitlength}
\begin{tikzpicture}
\pgftransformxscale{0.500000}
\pgftransformyscale{-0.500000}
\definecolor{dialinecolor}{rgb}{0.000000, 0.000000, 0.000000}
\pgfsetstrokecolor{dialinecolor}
\definecolor{dialinecolor}{rgb}{1.000000, 1.000000, 1.000000}
\pgfsetfillcolor{dialinecolor}
\pgfsetlinewidth{0.100000\du}
\pgfsetdash{}{0pt}
\pgfsetdash{}{0pt}
\pgfsetmiterjoin
\pgfsetbuttcap
\definecolor{dialinecolor}{rgb}{0.000000, 0.000000, 0.000000}
\pgfsetstrokecolor{dialinecolor}
\draw (26.000000\du,3.000000\du)--(31.000000\du,5.000000\du)--(35.000000\du,9.000000\du)--(34.000000\du,15.000000\du)--(29.000000\du,18.000000\du)--(23.000000\du,18.000000\du)--(18.000000\du,15.000000\du)--(17.000000\du,9.000000\du)--(21.000000\du,5.000000\du)--cycle;

\definecolor{dialinecolor}{rgb}{0.000000, 0.000000, 0.000000}
\pgfsetstrokecolor{dialinecolor}
\node[anchor=west] at (22.000000\du,4.000000\du){$-2$};
\definecolor{dialinecolor}{rgb}{0.000000, 0.000000, 0.000000}
\pgfsetstrokecolor{dialinecolor}
\node[anchor=west] at (20.000000\du,16.000000\du){$-2$};
\definecolor{dialinecolor}{rgb}{0.000000, 0.000000, 0.000000}
\pgfsetstrokecolor{dialinecolor}
\node[anchor=west] at (18.000000\du,7.000000\du){$-2$};
\definecolor{dialinecolor}{rgb}{0.000000, 0.000000, 0.000000}
\pgfsetstrokecolor{dialinecolor}
\node[anchor=west] at (33.000000\du,7.000000\du){$-2$};
\definecolor{dialinecolor}{rgb}{0.000000, 0.000000, 0.000000}
\pgfsetstrokecolor{dialinecolor}
\node[anchor=west] at (25.000000\du,18.500000\du){$-2$};
\definecolor{dialinecolor}{rgb}{0.000000, 0.000000, 0.000000}
\pgfsetstrokecolor{dialinecolor}
\node[anchor=west] at (33.000000\du,13.000000\du){$-2$};
\definecolor{dialinecolor}{rgb}{0.000000, 0.000000, 0.000000}
\pgfsetstrokecolor{dialinecolor}
\node[anchor=west] at (37.250000\du,10.300000\du){$\widetilde{C}$};
\definecolor{dialinecolor}{rgb}{0.000000, 0.000000, 0.000000}
\pgfsetstrokecolor{dialinecolor}
\node[anchor=west] at (29.000000\du,5.000000\du){$-1$};
\definecolor{dialinecolor}{rgb}{0.000000, 0.000000, 0.000000}
\pgfsetstrokecolor{dialinecolor}
\node[anchor=west] at (15.834100\du,12.251100\du){$-1$};
\definecolor{dialinecolor}{rgb}{0.000000, 0.000000, 0.000000}
\pgfsetstrokecolor{dialinecolor}
\node[anchor=west] at (30.305000\du,17.682500\du){$-1$};

\useasboundingbox (10.000000\du,0.000000\du) rectangle (31.000000\du,20.000000\du);

\pgfsetlinewidth{0.100000\du}
\pgfsetdash{}{0pt}
\pgfsetdash{}{0pt}
\pgfsetmiterjoin
\pgfsetbuttcap
{
\definecolor{dialinecolor}{rgb}{0.000000, 0.000000, 0.000000}
\pgfsetfillcolor{dialinecolor}
\definecolor{dialinecolor}{rgb}{0.000000, 0.000000, 0.000000}
\pgfsetstrokecolor{dialinecolor}
\pgfpathmoveto{\pgfpoint{12.000000\du}{15.000000\du}}
\pgfpathcurveto{\pgfpoint{14.000000\du}{17.000000\du}}{\pgfpoint{30.000000\du}{0.000000\du}}{\pgfpoint{34.000000\du}{1.000000\du}}
\pgfpathcurveto{\pgfpoint{38.000000\du}{2.000000\du}}{\pgfpoint{39.900000\du}{27.907524\du}}{\pgfpoint{28.000000\du}{11.500000\du}}
\pgfusepath{stroke}
}
\end{tikzpicture}
\end{figure}

    Now we consider the new pencil given by the strict transformation $\widetilde{C}$ of $C$ and the sum of the strict transformations of $L_{x},L_{y},L_{z}$ and the 6 exceptional curves. The two divisors intersect transversally at 3 base points, and both of them lie in the anti-canonical linear system. The anti-canonical linear system induces a morphism onto a cubic surface in $\mathbb{P}^3$. By the configuration above, the surface has 3 $\mathbf{A}_{2}$ singularities. This cubic surface is isomorphic to the surface $x_{0}x_{1}x_{2}-x_{3}^3=0$.

    Similarly, we compute for other del Pezzo surfaces and list the information about their LG models as follows:
        
\begin{center}
\begin{longtable}{|C|C|C|C|C|}
    \hline
    \text{Variety} & \text{An LG model} & \text{Config of }\overline{Y} & \overline{Y}_{\infty} & \text{Image in } \mathbb{P}^3 \\
    \hline
     \mathbb{P}^2 & x + y + \frac{1}{xy} & (4,4,1) & I_{9} & 3\mathbf{A}_{2}, x_{0}x_{1}x_{2}-x_{3}^3=0  \\
     \hline
     \mathbb{F}_{1} & x + y + \frac{1}{x} + \frac{1}{xy} & (4,3,1,1) & I_{8} & 1\mathbf{A}_{1}2\mathbf{A}_{2}\text{ or }1 \mathbf{A}_{1}1\mathbf{A}_{4},x_{0}x_{1}x_{2}-x_{0}x_{3}^2-x_{3}^3  \\
     \hline
     \mathbb{P}^1 \times \mathbb{P}^1 & x + y + \frac{1}{x} + \frac{1}{y} & (3,3,2,1) & I_{8} & 2\mathbf{A}_{1}1\mathbf{A}_{3},x_{0}x_{1}x_{2}-x_{0}x_{3}^2-x_{1}x_{3}^2 \\
     \hline
\end{longtable}
\end{center}
Here the configuration of $\overline{Y}$ represents the position of blown-up points of $\mathbb{P}^2$ when we represent $\overline{Y}$ as the blow-up of $\mathbb{P}^2$ at 9 points.

A general cubic surface can be degenerated to $Q \cup_{C} H$, where $Q$ is a smooth quadratic surface or a quadratic cone which has an $\mathbf{A}_{1}$ singularity on its vertex, $H$ is a plane and $C$ is a conic curve. The degeneration must fix the triangle of lines that passes through the 3 singular points.

The mirror of $\mathbb{P}^2$ doesn't have such degeneration. Indeed, since $Q \cup_{C} H$ doesn't have $\mathbf{A}_{2}$ singularities, the three singularities must be on $C$. However, in this case, the three lines passing through the 3 singular points degenerate to the plane $H$. This contradicts the flatness condition.

In the case of $\mathbb{P}^1 \times \mathbb{P}^1$, we can fix one of the $\mathbf{A}_{1}$ singularity to be the vertex of the quadratic cone, and degenerate the other 2 singularities to a nodal singularity along a curve. There are two choices, so $\mathbb{P}^1 \times \mathbb{P}^1$ admits two $\mathbb{P}^1$ fibrations.

The argument is the same for the case of $\mathbb{F}_{1}$. Notice that in the second configuration, we cannot fix the $\mathbf{A}_{1}$, since the intersection number at the nonsingular point is 1, and it cannot degenerate to nodal singularity.

 This coincides with the fact that $\mathbb{P}^1 \times \mathbb{P}^1$ admits two $\mathbb{P}^1$ fibrations, $\mathbb{F}_{1}$ admits one $\mathbb{P}^1$ fibration, and $\mathbb{P}^2$ doesn't admit any $\mathbb{P}^1$ fibration.
\end{example}

\subsection{Flips/flops and wall-crossings}

    Suppose we have a flipping/flopping diagram as follows:
    $$
    \begin{tikzcd}
        X \arrow[rr,dashed] \arrow[rd] & & X^{+} \arrow[ld]\\
        & T &
    \end{tikzcd}
    $$
    Then $\mathrm{Pic}_{\mathbb{R}}(T)$ is naturally a subspace of codimension 1 in $\mathrm{Pic}_{\mathbb{R}}(X)$. Hence we obtain a wall of codimension 1 in the moduli of LG models.

\subsection{A correspondence on geography}

    Let $X$ be a terminal Fano variety. Let $D$ be an effective $\mathbb{R}$-divisor on $X$. Then we can produce a Fano fibration $X' \rightarrow Z$ as follows: Let $0 \leq A \sim_{\mathbb{R}} -K_{X}$ be a sufficiently general $\mathbb{R}$-divisor. We take $X_{m}$ to be a minimal model of the pair $(X,A+\epsilon D)$ for $0< \epsilon \ll 1$, and $Z$ to be a log canonical model. Finally we let $X'$ be the ample model of $-K_{X_{m}}$ over $Z$. We have $\mathrm{dim}X' > \mathrm{dim}Z$ if and only if $D \in \partial\mathrm{Eff}_{\mathbb{R}}(X)$.

    In this section, we view this construction from the mirror side, that is, how to produce the toric Landau-Ginzburg models of $X'/Z$ from the toric Landau-Ginzburg models of $X$ and the $\mathbb{R}$-divisor $D$.

\subsubsection{Minimal degenerations and minimal models}

\begin{definition}
    Let $X$ be a terminal Fano variety and $D$ be an effective $\mathbb{R}$-divisor on $X$. Let $f$ be a toric Landau-Ginzburg model of $X$. We say that a degeneration $f_{t}$ of $f$ is a \emph{degeneration in the direction of $D$}, if $f_{0}=f$, $f_{t}$ is a toric Landau-Ginzburg model of $(X,tD)$, and the limit $f_{\infty} = \lim\limits_{t \to +\infty}f_{t}$ exists.

    We say that a degeneration $f_{min,D}$ (or simply $f_{min}$) of $f$ is a \emph{minimal degeneration in the direction of $D$}, if
    \begin{enumerate}
        \item For any degeneration $f_{t}$ of $f$ in the direction of $D$, we have $f_{min}$ degenerates to $f_{\infty}$.
        \item \emph{Universal property}. For any Laurent polynomial $f'$ satisfying the previous condition, we have $f'$ degenerates to $f_{min}$.
    \end{enumerate}
\end{definition}

We expect that minimal degenerations correspond to minimal models.

\begin{conjecture}\label{conjecture: minimal degeneration and minimal model}
    Let $X$ be a terminal Fano variety and $f$ be a toric Landau-Ginzburg model of $X$. Let $f_{min}$ be the minimal degeneration of $f$ in the direction of $D$. Let $X_{min}$ be the $D$-minimal model of $X$. Then $f_{min}$ is a toric Landau-Ginzburg model of $X_{min}$.
\end{conjecture}

\subsubsection{Maximal degenerations and general fibres}

\begin{definition}
    Let $X$ be a terminal Fano variety and $D$ be an effective $\mathbb{R}$-divisor on $X$. Let $f$ be a toric Landau-Ginzburg model of $X$. We say that a degeneration $f_{max,D}$ (or simply $f_{max}$) of $f$ is a \emph{maximal degeneration in the direction of $D$}, if
    \begin{enumerate}
        \item For any degeneration $f_{t}$ of $f$ in the direction of $D$, we have $f_{\infty}$ degenerates to $f_{max}$.
        \item \emph{Universal property}. For any Laurent polynomial $f'$ satisfying the previous condition, we have $f_{max}$ degenerates to $f'$.
    \end{enumerate}
\end{definition}

We expect that maximal degenerations correspond to the general fibres of the Iitaka fibration.

\begin{conjecture}\label{conjecture: maximal degeneration and Iitaka fibration}
    Let $X$ be a terminal Fano variety and $D$ be an effective $\mathbb{R}$-divisor on $X$. Let $f$ be a toric Landau-Ginzburg model of $X$ and $f_{max}$ the minimal degeneration of $f$ in the direction of $D$. Let $X_{min}$ be a $D$-minimal model of $X$ and $\pi: X_{min} \rightarrow Z$ its log canonical model. Let $F$ be a general fibre of $\pi$. Then $f_{max}$ is a toric LG model of $F$.
\end{conjecture}

\subsubsection{Information of the base}

    Finally, we briefly discuss the log canonical model $Z$. Notice that $Z$ may not have terminal singularities. In fact, $Z$ even may not be $\mathbb{Q}$-Gorenstein. Hence we cannot talk about the quantum period of $Z$ and the period condition of a toric LG model. In general, it's difficult to obtain information about $Z$.

\subsubsection{An example}

    In this subsection, we give an example to illustrate the above constructions.

\begin{example}
    Let $X$ be the blow-up of $\mathbb{P}^3$ at 2 points. Then $X$ is a smooth weak Fano variety. More precisely, the divisor $-K_{X}$ is semi-ample, and defines a projective birational morphism that contracts the strict transformation of the unique line passing through the two centers of the blow-up. The effective cone $\mathrm{Eff}_{\mathbb{R}}(X)$ is generated by the class of two exceptional divisors $E_{1}$, $E_{2}$, as well as the class of the strict transformation $\widetilde{H}$ of a hyperplane passing through the two centers of the blow-up. The geography is the following:
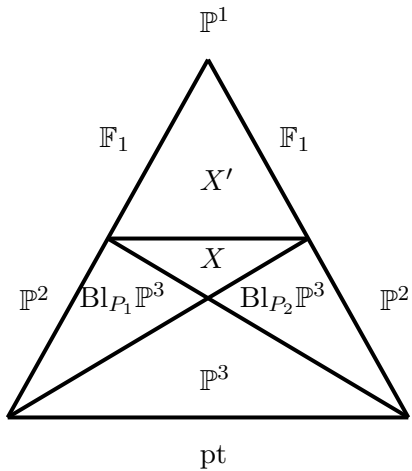
\begin{figure}[H]
    \centering
    \ifx\du\undefined
  \newlength{\du}
\fi
\setlength{\du}{15\unitlength}
\begin{tikzpicture}
\pgftransformxscale{1.000000}
\pgftransformyscale{-1.000000}
\definecolor{dialinecolor}{rgb}{0.000000, 0.000000, 0.000000}
\pgfsetstrokecolor{dialinecolor}
\definecolor{dialinecolor}{rgb}{1.000000, 1.000000, 1.000000}
\pgfsetfillcolor{dialinecolor}
\pgfsetlinewidth{0.100000\du}
\pgfsetdash{}{0pt}
\pgfsetdash{}{0pt}
\pgfsetbuttcap
{
\definecolor{dialinecolor}{rgb}{0.000000, 0.000000, 0.000000}
\pgfsetfillcolor{dialinecolor}
\definecolor{dialinecolor}{rgb}{0.000000, 0.000000, 0.000000}
\pgfsetstrokecolor{dialinecolor}
\draw (25.000000\du,10.000000\du)--(20.000000\du,19.000000\du);
}
\pgfsetlinewidth{0.100000\du}
\pgfsetdash{}{0pt}
\pgfsetdash{}{0pt}
\pgfsetbuttcap
{
\definecolor{dialinecolor}{rgb}{0.000000, 0.000000, 0.000000}
\pgfsetfillcolor{dialinecolor}
\definecolor{dialinecolor}{rgb}{0.000000, 0.000000, 0.000000}
\pgfsetstrokecolor{dialinecolor}
\draw (25.000000\du,10.000000\du)--(30.000000\du,19.000000\du);
}
\pgfsetlinewidth{0.100000\du}
\pgfsetdash{}{0pt}
\pgfsetdash{}{0pt}
\pgfsetbuttcap
{
\definecolor{dialinecolor}{rgb}{0.000000, 0.000000, 0.000000}
\pgfsetfillcolor{dialinecolor}
\definecolor{dialinecolor}{rgb}{0.000000, 0.000000, 0.000000}
\pgfsetstrokecolor{dialinecolor}
\draw (20.000000\du,19.000000\du)--(30.000000\du,19.000000\du);
}
\pgfsetlinewidth{0.100000\du}
\pgfsetdash{}{0pt}
\pgfsetdash{}{0pt}
\pgfsetbuttcap
{
\definecolor{dialinecolor}{rgb}{0.000000, 0.000000, 0.000000}
\pgfsetfillcolor{dialinecolor}
\definecolor{dialinecolor}{rgb}{0.000000, 0.000000, 0.000000}
\pgfsetstrokecolor{dialinecolor}
\draw (30.000000\du,19.000000\du)--(22.500000\du,14.500000\du);
}
\pgfsetlinewidth{0.100000\du}
\pgfsetdash{}{0pt}
\pgfsetdash{}{0pt}
\pgfsetbuttcap
{
\definecolor{dialinecolor}{rgb}{0.000000, 0.000000, 0.000000}
\pgfsetfillcolor{dialinecolor}
\definecolor{dialinecolor}{rgb}{0.000000, 0.000000, 0.000000}
\pgfsetstrokecolor{dialinecolor}
\draw (20.000000\du,19.000000\du)--(27.500000\du,14.500000\du);
}
\pgfsetlinewidth{0.100000\du}
\pgfsetdash{}{0pt}
\pgfsetdash{}{0pt}
\pgfsetbuttcap
{
\definecolor{dialinecolor}{rgb}{0.000000, 0.000000, 0.000000}
\pgfsetfillcolor{dialinecolor}
\definecolor{dialinecolor}{rgb}{0.000000, 0.000000, 0.000000}
\pgfsetstrokecolor{dialinecolor}
\draw (22.500000\du,14.500000\du)--(27.500000\du,14.500000\du);
}
\definecolor{dialinecolor}{rgb}{0.000000, 0.000000, 0.000000}
\pgfsetstrokecolor{dialinecolor}
\node[anchor=west] at (24.500000\du,13.000000\du){$X'$};
\definecolor{dialinecolor}{rgb}{0.000000, 0.000000, 0.000000}
\pgfsetstrokecolor{dialinecolor}
\node[anchor=west] at (24.500000\du,15.000000\du){$X$};
\definecolor{dialinecolor}{rgb}{0.000000, 0.000000, 0.000000}
\pgfsetstrokecolor{dialinecolor}
\node[anchor=west] at (24.500000\du,18.000000\du){$\mathbb{P}^3$};
\definecolor{dialinecolor}{rgb}{0.000000, 0.000000, 0.000000}
\pgfsetstrokecolor{dialinecolor}
\node[anchor=west] at (21.500000\du,16.000000\du){$\mathrm{Bl}_{P_{1}}\mathbb{P}^3$};
\definecolor{dialinecolor}{rgb}{0.000000, 0.000000, 0.000000}
\pgfsetstrokecolor{dialinecolor}
\node[anchor=west] at (25.500000\du,16.000000\du){$\mathrm{Bl}_{P_{2}}\mathbb{P}^3$};
\definecolor{dialinecolor}{rgb}{0.000000, 0.000000, 0.000000}
\pgfsetstrokecolor{dialinecolor}
\node[anchor=west] at (22.000000\du,12.000000\du){$\mathbb{F}_{1}$};
\definecolor{dialinecolor}{rgb}{0.000000, 0.000000, 0.000000}
\pgfsetstrokecolor{dialinecolor}
\node[anchor=west] at (21.000000\du,16.000000\du){};
\definecolor{dialinecolor}{rgb}{0.000000, 0.000000, 0.000000}
\pgfsetstrokecolor{dialinecolor}
\node[anchor=west] at (20.000000\du,16.000000\du){$\mathbb{P}^2$};
\definecolor{dialinecolor}{rgb}{0.000000, 0.000000, 0.000000}
\pgfsetstrokecolor{dialinecolor}
\node[anchor=west] at (24.500000\du,9.000000\du){$\mathbb{P}^1$};
\definecolor{dialinecolor}{rgb}{0.000000, 0.000000, 0.000000}
\pgfsetstrokecolor{dialinecolor}
\node[anchor=west] at (26.500000\du,12.000000\du){$\mathbb{F}_{1}$};
\definecolor{dialinecolor}{rgb}{0.000000, 0.000000, 0.000000}
\pgfsetstrokecolor{dialinecolor}
\node[anchor=west] at (29.000000\du,16.000000\du){$\mathbb{P}^2$};
\definecolor{dialinecolor}{rgb}{0.000000, 0.000000, 0.000000}
\pgfsetstrokecolor{dialinecolor}
\node[anchor=west] at (24.500000\du,20.000000\du){pt};
\end{tikzpicture}
\end{figure}

    where $X \dashrightarrow X'$ is a flop. On the other side, the toric LG model of $X$ is
    $$
    f = x+y+z+\frac{1}{x} + \frac{1}{y} + \frac{1}{xyz}.
    $$
    We take an effective divisor $D = \widetilde{H} + 2E_{1}$ as an example. A degeneration of $f$ in the direction of $D$ is given by
    $$
    f_{t} = x+y+z+\frac{e^{-2t}}{x} + \frac{1}{y} + \frac{e^{-t}}{xyz}.
    $$
    Hence
    $$
    f_{\infty} = x+y+z+\frac{1}{y}
    $$
    for this degeneration. Another degeneration of $f$ in the direction of $D$ is given by
    $$
    f'_{t} = e^{-t}x+y+z+\frac{e^{-t}}{x} + \frac{1}{y} + \frac{1}{xyz}.
    $$
    Hence
    $$
    f_{\infty} = y+z+\frac{1}{y}+\frac{1}{xyz}
    $$
    for this degeneration. One can verify that in this case, we have
    $$
    f_{min} = x+y+z + \frac{1}{y} + \frac{1}{xyz}
    $$
    and
    $$
    f_{max} = y + \frac{1}{y}.
    $$ 
    Hence we obtain a central model $\mathrm{Bl}\mathbb{P}^3 \rightarrow \mathbb{P}^2$, where the general fibre is isomorphic to $\mathbb{P}^1$. This coincides with our construction.
\end{example}

\subsection{Deformations and mutations of Landau-Ginzburg models}

    The correspondence in this subsection is not much related to our main results. We list the correspondence for completeness of the whole picture.

    \begin{conjecture}
        Let $X$ be a smooth Fano variety and $f$ be a toric LG model of $X$. Then there is a one-to-one correspondence between the small  toric degenerations of $X$ and the mutations of $f$.
    \end{conjecture}

\begin{example}[Toric degeneration of del Pezzo surfaces]
    Consider the $\mathbb{Q}$-Gorenstein degenerations of $\mathbb{P}^2$. Toric $\mathbb{Q}$-Gorenstein degenerations of $\mathbb{P}^2$ of Picard number 1 are weighted projective planes $\mathbb{P}(a^{2},b^{2},c^{2})$, where $a,b,c$ are positive integers satisfying the Markov equation
    $$
    a^{2}+b^{2}+c^{2}=3abc.
    $$
    A triple $(a,b,c)$ satisfying the above equation is called a Markov triple. Given a weighted projective plane $\mathbb{P}(a^{2},b^{2},c^{2})$ satisfying the above conditions, we can construct a new weighted projective plane $\mathbb{P}(a^{2},b^{2},(3ab-c)^{2})$. It is easy to verify that $a,b,3ab-c$ also satisfies the Markov equation. The triple $(a,b,3ab-c)$ is called a \emph{mutation} of the triple $(a,b,c)$.

    On the other side, the Hori-Vafa mirror of $\mathbb{P}(a^{2},b^{2},c^{2})$ is given by
    \begin{align*}
    f = u_{0}+u_{1}+u_{2}, \\
    u_{0}^{a^2}u_{1}^{b^2}u_{2}^{c^2}=1.
    \end{align*}
    We can take a change of variables given by
    \begin{align*}
    u_{0} = \frac{x^{c}}{y^{3bc-a}}, \\
    u_{1} = x^{c}y^{a}, \\
    u_{2} = \frac{y^{a}}{x^{3ab-c}}.
    \end{align*}
    Then the Landau-Ginzburg model can be rewritten as
    $$
    f = \frac{x^{c}}{y^{3bc-a}}(1+y^{3bc})+\frac{y^{a}}{x^{3ab-c}} \sim \frac{x^{c}}{y^{3bc-a}}(1+y^{3b})^{c}+\frac{y^{a}}{x^{3ab-c}}
    $$
    where $\sim$ denotes that the two Laurent polynomials have the same Newton polytope. Consider the mutation given by weight $w=(-1,0)$ and factor $a = 1+y^{3b}$, we get
    $$
    f' = \frac{x^{c}}{y^{3bc-a}} + \frac{y^{a}}{x^{3ab-c}}(1+y^{3b})^{3ab-c} \sim \frac{x^{c}}{y^{3bc-a}} + \frac{y^{a}}{x^{3ab-c}} + \frac{y^{9ab^{2}-3bc+a}}{x^{3ab-c}}.
    $$
    Let
    \begin{align*}
    v_{0} = \frac{y^{a}}{x^{3ab-c}}, \\
    v_{1} = \frac{y^{9ab^{2}-3bc+a}}{x^{3ab-c}}, \\
    v_{2} = \frac{x^{c}}{y^{3bc-a}}.
    \end{align*}
    Then we have
    \begin{align*}
    f' = v_{0}+v_{1}+v_{2}, \\
    v_{0}^{a^2}v_{1}^{b^2}v_{2}^{(3ab-c)^2}=1.
    \end{align*}
    Hence $f'$ becomes a Hori-Vafa mirror of $\mathbb{P}(a^{2},b^{2},(3ab-c)^{2})$. Notice that every Markov triple $(a,b,c)$ can be obtain from the triple $(1,1,1)$ by a sequence of mutations (cf. \cite{Baulina1996}). Hence the above diagram shows that we have the 1-1 correspondence
    \begin{multline*}
    \{\text{Toric }\mathbb{Q}\text{-Gorenstein degenerations of }\mathbb{P}^2 \text{ of Picard number 1 up to isomorphisms}\}  \\
    \longleftrightarrow \{\text{Mutations of the Laurent polynomial }x+y+\frac{1}{xy} \text{ up to change of coordinates} \}\\
    \longleftrightarrow \{ \text{Mutations of the Markov triple }(1,1,1) \text{ up to permutations} \}.
    \end{multline*}
    Moreover, the same argument applies to smooth del Pezzo surfaces of degree $\geq 5$.
\end{example}

\section{Toric case}\label{section: toric}
    In this section we consider the case when $X$ is a smooth toric Fano variety. 
    
\subsection{Toric LG models in an extremal contraction}

    First, we look at \cref{conj: mirror of extremal contractions}.
    
\begin{proposition}\label{prop: LG model for toric divisorial contraction}
    Let $g: X \rightarrow Y$ be a toric divisorial contraction between smooth toric Fano varieties. Then the Hori-Vafa mirror of $X$ degenerates to the Hori-Vafa mirror of $Y$.
\end{proposition}

\begin{proof}
    For smooth toric Fano varieties, the Hori-Vafa mirror is a Landau-Ginzburg model. Hence the result is trivial by the definition of Hori-Vafa mirrors.
\end{proof}

\begin{proposition}\label{prop: LG model for toric Mori fibre space}
    Let $X$ be a smooth toric Fano variety. Let $h: X \rightarrow Z$ be a toric Mori fibre space and $F$ be a general fibre of $h$. Then the Hori-Vafa mirror of $X$ degenerates to the Hori-Vafa mirror of $F$.
\end{proposition}

\begin{proof}
    Let $\rho: \Sigma_{X} \rightarrow \Sigma_{Z}$ be the associated projection of the toric fans. Let $\Sigma_{F}$ be the fan consisting of cones $C$ in $\Sigma_{X}$ such that $\rho(C) = 0$. Then $\Sigma_{F}$ is the toric fan of the general fibre $F$ of $h$. Hence the result is trivial by the definition of Hori-Vafa mirrors.
\end{proof}
        
    Now we study the induced degeneration on their Calabi-Yau compactifications.

\begin{statement}[degenerations of compactified LG models, divisorial contraction]
    Let notations be as in \cref{prop: LG model for toric divisorial contraction}. Let $\Sigma_{X}$ and $\Sigma_{Y}$ be the associated toric fan of $X$ and $Y$ respectively. Since $X$ and $Y$ are Fano, we can consider the dual fan $\Sigma_{X}^{\vee}$ and $\Sigma_{Y}^{\vee}$ of $\Sigma_{X}$ and $\Sigma_{Y}$ given by their dual polytopes, which corresponds to Fano varieties $X^{\vee}$ and $Y^{\vee}$ respectively. Let $\widetilde{X}^{\vee}$ and $\widetilde{Y}^{\vee}$ be the crepant terminalization of $X^{\vee}$ and $Y^{\vee}$ respectively. Then we have a dual morphism $\widetilde{g}^{\vee}: \widetilde{Y}^{\vee} \rightarrow \widetilde{X}^{\vee}$. Recall that the Calabi-Yau compactification of the Hori-Vafa mirrors of $X$ and $Y$ are given by pencils of anti-canonical divisors of $\widetilde{X}^{\vee}$ and $\widetilde{Y}^{\vee}$ respectively. Write
    $$
    \widetilde{g}^{\vee *} K_{\widetilde{X}^{\vee}} = K_{\widetilde{Y}^{\vee}} - E^{\vee}
    $$
    where $E^{\vee} > 0$ is an integral divisor. Take a pencil of divisors in $|-K_{\widetilde{X}^{\vee}}|$ generated by the toric boundary of $\widetilde{X}^{\vee}$ and a general anti-canonical divisor on $\widetilde{X}^{\vee}$. We can pull back this pencil to a pencil of divisors in $|-K_{\widetilde{Y}^{\vee}}+E^{\vee}|$ on $\widetilde{Y}^{\vee}$. We can construct a degeneration of this pencil to a pencil of divisors of the form $|-K_{\widetilde{Y}^{\vee}}| + E^{\vee}$, where $E$ is the fixed part of the pencil. The movable part now becomes the compactified Landau-Ginzburg model of $Y$. Hence for a divisorial contraction, the induced degeneration of the compactified LG models is an irreducible degeneration.
    \end{statement}
    
\begin{statement}[degenerations of compactified LG models, Mori fibre space]\label{degeneration of compactified LG models for MFS}

    Let the notations be as in \cref{prop: LG model for toric Mori fibre space}. Since $X$ and $F$ are smooth, we can assume that the standard basis of the lattice generates a top-dimensional cone of $\Sigma_{X}$ which contains a top-dimensional cone of $\Sigma_{F}$. Since $X$ and $F$ are Fano, we can consider the dual fan $\Sigma_{X}^{\vee}$ and $\Sigma_{F}^{\vee}$ of $\Sigma_{X}$ and $\Sigma_{F}$ given by their dual polytopes, which corresponds to Fano varieties $X^{\vee}$ and $F^{\vee}$ respectively. By our choice of basis, the polytope of $\Sigma_{X}^{\vee}$ contains the following extremal faces:
    
    \begin{enumerate}
        \item $L_{i}(x_{1},\cdots,x_{l}) = 1$, where $L_{i}$ defines an extremal face of the polytope of $\Sigma_{F}^{\vee}$;
        \item $x_{j} = 1$, where $j = l+1, \cdots , n$.
    \end{enumerate}

     Let $\widetilde{\Sigma_{X}^{\vee}}$ and $\widetilde{\Sigma_{F}^{\vee}}$ be the crepant refinement of $\Sigma_{X}^{\vee}$ and $\Sigma_{F}^{\vee}$, which corresponds to their crepant terminalization $\widetilde{X}^{\vee}$ and $\widetilde{F}^{\vee}$. Consider the projection
     $$
     p: \mathbb{Z}^{n} \rightarrow \mathbb{Z}^{l}
     $$
     given by the projection to the first $l$ components. Choosing an appropriate triangulation of the polytope, we obtain a morphism $\widetilde{h}: \widetilde{X}^{\vee} \rightarrow \widetilde{F}^{\vee}$. Now $f$ and $f'$ are compactified as pencils of anti-canonical divisors of $\widetilde{X}^{\vee}$ and $\widetilde{F}^{\vee}$ respectively. Let $Y$ and $Y'$ be the Calabi-Yau compactification of $f$ and $f'$ respectively.  Now we consider the degeneration of $f$ to $f'$. This corresponds to a pencil of anti-canonical divisors where the fixed components are given by primitive vectors of $\widetilde{\Sigma_{X}^{\vee}}$ which aren't contained in any face of the form $L_{i}(x_{1},\cdots,x_{l}) = 1$, where $L_{i}$ defines an extremal face of the polytope of $\Sigma_{F}^{\vee}$. The movable part of this linear system is the pull-back of the anti-canonical divisor of $\widetilde{F}^{\vee}$. Hence we obtain a degeneration
    $$
    \begin{tikzcd}
        Y \arrow[d] \arrow[r,decorate,decoration={snake,segment length = 1.6mm,amplitude=0.2mm}] & Y^{*} \arrow[d] \\
        \mathbb{A}^1 & Y' \arrow[l]
    \end{tikzcd}
    $$
\end{statement}

\begin{statement}[Tyurin degenerations]
    Let the notations be as above. We consider a special case where $h$ has relative dimension 1. Consider the projection
     $$
     p_{n}: \mathbb{Z}^{n} \rightarrow \mathbb{Z}
     $$
     given by the projection to the last component. Then there is a morphism $\widetilde{h}: \widetilde{X}^{\vee} \rightarrow \mathbb{P}^1$ with general fibre isomorphic to $\widetilde{F}^{\vee}$. Now $f$ and $f'$ are compactified as pencils of anti-canonical divisors of $\widetilde{X}^{\vee}$ and $\widetilde{F}^{\vee}$ respectively. Let $Y$ and $Y'$ be the Calabi-Yau compactification of $f$ and $f'$ respectively. Then there is a morphism
    $$
    \widetilde{f}: Y \rightarrow \mathbb{P}^1 \times \mathbb{A}^{1}
    $$
    where a general fibre over $\mathbb{P}^1$ is isomorphic to $Y'$. Let $D = \sum\limits_{p(v)=1}D_{v}$, then $D$ is the linear equivalence class of a general fibre of $\widetilde{h}$. A general member in the complete linear system $|-K_{\widetilde{X}^{\vee}}|$ degenerates to a general member in the linear system $|-K_{\widetilde{X}^{\vee}} - D| + |D|$. Let $Y_{1} \cong \widetilde{F}^{\vee} \times \mathbb{A}^1$. We obtain a degeneration
    $$
    \begin{tikzcd}
        Y \arrow[d] \arrow[r,decorate,decoration={snake,segment length = 1.6mm,amplitude=0.2mm}] & Y_{0} \cup Y_{1} \arrow[d] \\
        \mathbb{A}^1 \arrow[equal,r] & \mathbb{A}^1
    \end{tikzcd}
    $$
    One can check that
    \begin{enumerate}
        \item each component $Y_{0} \rightarrow \mathbb{A}^1$ and $Y_{1} \rightarrow \mathbb{A}^1$ are weak Fano fibrations;
        \item $Y_{0} \cap Y_{1} = Y'$;
        \item $Y'$ is an anti-canonical divisor of both $Y_{0}$ and $Y_{1}$.
    \end{enumerate}
    This can be think of as the Fano version of the Tyurin degenerations of the mirrors of Calabi-Yau varieties.
\end{statement}

\subsection{Toric LG models in a toric minimal model program}

      Now we show that \cref{conjecture: minimal degeneration and minimal model} holds for smooth toric varieties.

\begin{proposition}\label{prop: toric LG model for toric pairs}
    Let $X$ be a smooth toric weak Fano variety and $D$ be an effective $\mathbb{R}$-divisor on $X$. Let $D_{1},\cdots,D_{l}$ be the toric invariant divisors of $X$. Write
    $$
    D \sim_{\mathbb{R}} \sum\limits_{i=1}^{l}d_{i}D_{i}.
    $$
    Then the LG model
    $$
    f = \sum\limits_{i=1}^{l}e^{-d_{i}}x^{v_{i}}
    $$
    is a toric LG model of $(X,D)$.
\end{proposition}

\begin{proof}
    Denote by $L \subseteq \mathbb{Z}_{\geq 0}^{l}$ the monoid consisting of tuples $(k_{1},\cdots,k_{l})$ such that $\sum\limits_{i=1}^{k}k_{i}v_{i}=0$. Then the constant term of $f^{d}$ is
    $$
    c(f^{d}) = \sum\limits_{\substack{(k_{1},\cdots,k_{l}) \in L \\ \sum\limits_{i=1}^{l} k_{i} = d}} \frac{d!}{k_{1}!\cdots k_{l}!}e^{-(\sum\limits_{i=1}^{l} k_{i}d_{i})}.
    $$
    
    We recall the regularized $I$-series
    $$
    \widetilde{I}_{0}^{X,tD}(s) = 1 + \sum\limits_{\beta \in K}(-K_{X} \cdot \beta)!\langle \tau_{-K_{X}\cdot\beta -2}\mathbf{1} \rangle_{\beta} \cdot e^{-D \cdot \beta}s^{-K_{X} \cdot \beta}.  
    $$

    Since $X$ is $\mathbb{Q}$-factorial, the cones in the fan of $X$ are simplicial. The moving part of the homology group $H_{2}(X,\mathbb{Z})$ are generated by toric invariant curves, which 1-1 corresponds to a tuple $(k_{1},\cdots,k_{l}) \in L$. Using \cite{Batyrev_1995}, Proposition 6.2.1, the intersection pair $L \times \mathrm{Pic}(X)$ is given by
    $$
    \beta \cdot D : (k_{1},\cdots,k_{l}) \times (d_{1},\cdots,d_{l}) \mapsto \sum\limits_{i=1}^{l} k_{i}d_{i}.
    $$
    By Givental's formula (cf. \cite{givental1997mirrortheoremtoriccomplete}, Theorem 0.1), the Gromov-Witten invariant $\langle \tau_{-K_{X}\cdot\beta -2}\mathbf{1} \rangle_{\beta} = \frac{1}{k_{1}!\cdots k_{l}!}$. Hence we have $\widetilde{I}_{0}^{X,D}(s) = 1 + \sum\limits_{d=1}^{+\infty} a_{0}(f^{d})$. This is the period condition.

\end{proof}

\begin{remark}
    In general, the representation 
    $$
    D \sim_{\mathbb{R}} \sum\limits_{i=1}^{l}d_{i}D_{i}
    $$
    is not unique. Hence the above proposition can produce many toric LG models of the pair $(X,D)$.
\end{remark}

\begin{proposition}
    Let $f$ be a toric Landau-Ginzburg model of $X$ and $f_{min}$ a minimal degeneration in the direction of $D$. Let $X_{min}$ be a toric $D$-minimal model of $X$. Then $f_{min}$ is a Landau-Ginzburg model of $X_{min}$.
\end{proposition}

\begin{proof}
    Let $f_{min}$ be a toric LG model of $X_{min}$. We show that $f_{min}$ satisfies the conditions of minimal degenerations. Consider the 
    degeneration
    $$
    f_{t} = \sum\limits_{i=1}^{l}e^{-d_{i}t}x^{v_{i}}
    $$
    where $v_{i}$ is the generating vector of the divisor $D_{i}$. Then $f_{t}$ is a toric LG model of $(X,tD)$ by \cref{prop: toric LG model for toric pairs}. As $t \rightarrow +\infty$, the term $x^{v_{i}}$ vanishes if and only if $d_{i}>0$.

    Hence the term $x^{v_{i}}$ vanishes in $f_{min}$ if and only if $d_{i}>0$ for every representation 
    $$
    D \sim_{\mathbb{R}} \sum\limits_{i=1}^{l}d_{i}D_{i}
    $$
    where $d_{i} \geq 0$, i.e., $D_{i}$ is contained in the support of the $\mathbb{R}$-fixed part of $D$. We conclude that $f_{min}$ is a LG model of the $D$-minimal model of $X$.
\end{proof}

    Now we show that \cref{conjecture: maximal degeneration and Iitaka fibration} holds for smooth toric Fano varieties.

\begin{proposition}
    Let $X$ be a smooth toric Fano variety and $D$ an effective $\mathbb{R}$-divisor on $X$. Let $f$ be a toric Landau-Ginzburg model of $X$ and $f_{max}$ a minimal degeneration in the direction of $D$. Let $X_{min}$ be a toric $D$-minimal model of $X$ and $f: X_{min} \rightarrow Z$ its log canonical model. Let $F$ be a general fibre of $f$. Then $f_{max}$ is a toric LG model of $F$.
\end{proposition}

\begin{proof}
    By our assumptions, we have $D_{min} \sim_{\mathbb{R}} f^{*} A$ for some ample divisor $A$ on $Z$. Let $\Sigma_{X_{min}} \rightarrow \Sigma_{Z}$ be the morphism between the corresponding fans. Then the fan $\Sigma_{F}$ of $F$ can be represented by the fan of cones of $\Sigma_{X_{min}}$ that maps to the origin of $\Sigma_{Z}$. In particular, the rays of the $\Sigma_{F}$ correspond to horizontal toric invariant divisors. 
\end{proof}

\subsection{Parametrized toric LG models for toric complete intersections}

    In this subsection, we compute a parametrized toric LG model for toric complete intersections.

\begin{proposition}
    Let $X$ be a smooth toric Fano variety and $\Sigma_{X}$ be its associated fan. Let $S_{0}, S_{1}, \cdots , S_{s}$ be a nef partition of $\Sigma_{X}$, i.e., we have
    $$
    \{ 1, 2, \cdots , l \} = S_{0} \sqcup S_{1} \cdots \sqcup S_{s},
    $$
    where 
    $$
    L_{i} = \sum\limits_{j \in S_{i}} D_{j}
    $$
    is nef for $i \geq 1$ and ample when $i = 0$. Let $Y \subseteq X$ be a smooth complete intersection in $X$ of dimension $\geq 3$ given by general sections of $L_{1}, L_{2}, \cdots, L_{s}$.
    
    Let $D$ be an effective $\mathbb{R}$-divisor on $X$ and $D_{1},\cdots , D_{l}$ be the toric invariant divisors of $X$. Write
    $$
    D \sim_{\mathbb{R}} \sum\limits_{i=1}^{l}d_{i}D_{i}.
    $$
    Then a parametrized toric LG model of $Y$ can be given by the equations
    $$
    f_{i} = \sum\limits_{j \in S_{i}} e^{-d_{j}} x^{v_{j}} = 1,\qquad i=1,2,\cdots s,
    $$
    with the potential
    $$
    f_{0} = \sum\limits_{i \in S_{0}}e^{-d_{i}}x^{v_{i}}.
    $$
\end{proposition}

\begin{proof}
    The strategy is similar to the proof of \cref{prop: toric LG model for toric pairs}. Denote by $\overline{c}(f_{0}^{d})$ the constant term of $f_{0}^{d}$ considered as a Laurent polynomial in $(\mathbb{C}^{*})^{l-s}$. We have
    $$
    \overline{c}(f_{0}^{d}) = \sum\limits_{(a_{1},\cdots, a_{s}) \in \mathbb{Z}^{s}} c(\prod\limits_{i=1}^{s}f_{i}^{a_{i}}f_{0}^{d}).
    $$
    
    Denote by $L \subseteq \mathbb{Z}_{\geq 0}^{l}$ the monoid consisting of tuples $(k_{1},\cdots,k_{l})$ such that $\sum\limits_{i=1}^{k}k_{i}v_{i}=0$. Then we can write
    $$
    \overline{c}(f_{0}^{d}) = \sum\limits_{\substack{(k_{1},\cdots,k_{l}) \in L \\ \sum\limits_{i \in S_{0}} k_{i} = d}} \frac{\prod\limits_{i = 0}^{s} (\sum\limits_{j \in S_{i}} k_{j})!}{\prod\limits_{i=1}^{l} k_{i}!}e^{-(\sum\limits_{i=1}^{l} k_{i}d_{i})}.
    $$

    On the other side, $L$ represents toric invariant curve classes in $H_{2}(X,\mathbb{Z})$. Since $\mathrm{dim}Y \geq 3$, the natural restrictions $\mathrm{Pic}(X) \rightarrow \mathrm{Pic}(Y)$ and $H_{2}(X,\mathbb{Z}) \rightarrow H_{2}(Y, \mathbb{Z})$ are isomorphisms. By \cite{CCGK}, the quantum period of $(Y,D|_{Y}))$ is
    $$
    G_{Y,D}(t) = e^{-ct} \sum\limits_{\substack{\beta \in H_{2}(Y,\mathbb{Z}), \\ \forall i, D_{i} \cdot \beta \geq 0}} \frac{\prod\limits_{i=1}^{s} (L_{i} \cdot \beta)!}{\prod\limits_{i=1}^{l}(D_{i} \cdot \beta)!} e^{-D \cdot \beta} t^{L_{0} \cdot \beta}.
    $$
    By \cref{lemma: period after shifting constants} the two periods coincide after shifting $f$ by a constant $c$. This is the period condition. The Calabi-Yau condition and the toric condition are already known.
\end{proof}

\section{General case (surfaces)}\label{section: surface}
    In this section we prove our main results \cref{conj: mirror of extremal contractions}, \cref{conj: mirror of MMP}, \cref{conj: local geography} and \cref{conj: main theorem} for surfaces.

\subsection{Existence of Landau-Ginzburg models for surfaces}\label{section:existence of LG models for surfaces}

    Before the proof, we list some preliminary results on the existence of toric LG models. First, for smooth del-Pezzo surfaces, the existence of the moduli of Landau-Ginzburg models are known.

    \begin{theorem}[\cite{CKP}, Theorem 28]\label{theorem:LG model of del Pezzo surfaces}
    There exists a moduli stack of LG models of smooth del-Pezzo surfaces.
    \end{theorem}

More explicitly, we list the following explicit moduli space of Landau-Ginzburg models for smooth del-Pezzo surfaces.

\begin{example}[cf. \cite{DHKOP} 2023, Appendix K]\label{example: parametrized LG model}
    Consider the parametrized Landau-Ginzburg model of del Pezzo surfaces of degree $\geq 3$. In this case $-K_{X}$ is very ample. Recall that for a complexified divisor $(\alpha_{1},\cdots, \alpha_{r})$, the coefficient specialization is given by $a_{i} = e^{-\alpha_{i}}$.
\begin{figure}[H]
\centering
\begin{longtable}{|C|C|}
    \hline
    \text{Variety} & \text{parametrized LG model} \\
    \hline
    \mathbb{P}^2 & x + y + \frac{a}{xy} \\
    \hline
    \mathbb{P}^1 \times \mathbb{P}^1 & x + y + \frac{a}{x} + \frac{b}{y} \\
    \hline
    \mathbb{F}_{1} & x + y + \frac{a_{1}}{xy} + \frac{a_{1}a_{2}}{x} \\
    \hline
    \mathrm{Bl}_{2}\mathbb{P}^2 & x + y + \frac{a_{1}}{xy} + \frac{a_{1}a_{2}}{x} + \frac{a_{1}a_{3}}{y} \\
    \hline
    \mathrm{Bl}_{3}\mathbb{P}^2 & x + y + \frac{a_{1}}{xy} + \frac{a_{1}a_{2}}{x} + \frac{a_{1}a_{3}}{y} + a_{4}xy\\
    \hline
    \mathrm{Bl}_{4}\mathbb{P}^2 & x + (1+a_{1}a_{2}a_{4}a_{5})y + \frac{a_{1}}{xy} + \frac{a_{1}(a_{2}+a_{5})}{x} + \frac{a_{1}a_{3}}{y} + a_{4}xy + \frac{a_{1}a_{2}a_{5}y}{x}\\
    \hline
    \multirow{2}{*}{$\mathrm{Bl}_{5}\mathbb{P}^2$} & x + (1+a_{1}a_{2}a_{4}a_{5})y + \frac{a_{1}(1 + a_{1}a_{3}a_{6}(a_{2}+a_{5}))}{xy} + \frac{a_{1}(a_{1}a_{2}a_{3}a_{5}a_{6} + a_{2} + a_{5})}{x} + \frac{a_{1}(a_{3}+a_{6})}{y} \\
    & + a_{4}xy + \frac{a_{1}a_{2}a_{5}y}{x} + \frac{a_{1}^{2}a_{3}a_{6}}{xy^2}\\
    \hline
    \multirow{3}{*}{$\mathrm{Bl}_{6}\mathbb{P}^2$} & (a_{1}a_{4}a_{7}(a_{3}+a_{6})+1)x + (a_{1}a_{2}a_{5}(a_{4}+a_{7})+1)y + \frac{a_{1}(1 + a_{1}a_{3}a_{6}(a_{2}+a_{5}))}{xy} \\ 
    & + \frac{a_{1}(a_{1}a_{2}a_{3}a_{5}a_{6} + a_{2} + a_{5})}{x} + \frac{a_{1}(a_{1}a_{3}a_{4}a_{6}a_{7} + a_{3} + a_{6})}{y} + (a_{1}a_{2}a_{4}a_{5}a_{7}+a_{4}+a_{7})xy\\
    & + \frac{a_{1}a_{2}a_{5}y}{x} + \frac{a_{1}^{2}a_{3}a_{6}}{xy^2} + a_{4}a_{7}x^{2}y \\
    \hline
    \end{longtable}
    \end{figure}

    The standard Landau-Ginzburg models can be obtained by letting $a_{i}=1$ for all $i$. We list them as follows:

    \begin{figure}[H]
\centering
\begin{longtable}{|C|C|}
    \hline
    \text{Variety} & \text{standard LG model} \\
    \hline
    \mathbb{P}^2 & x + y + \frac{1}{xy} \\
    \hline
    \mathbb{P}^1 \times \mathbb{P}^1 & x + y + \frac{1}{x} + \frac{1}{y} \\
    \hline
    \mathbb{F}_{1} & x + y + \frac{1}{xy} + \frac{1}{x} \\
    \hline
    \mathrm{Bl}_{2}\mathbb{P}^2 & x + y + \frac{1}{xy} + \frac{1}{x} + \frac{1}{y} \\
    \hline
    \mathrm{Bl}_{3}\mathbb{P}^2 & x + y + \frac{1}{xy} + \frac{1}{x} + \frac{1}{y} + xy\\
    \hline
    \mathrm{Bl}_{4}\mathbb{P}^2 & x + 2y + \frac{1}{xy} + \frac{2}{x} + \frac{1}{y} + xy + \frac{y}{x}\\
    \hline
    \mathrm{Bl}_{5}\mathbb{P}^2 & x + 2y + \frac{3}{xy} + \frac{3}{x} + \frac{2}{y} + xy + \frac{y}{x} + \frac{1}{xy^2}\\
    \hline
    \mathrm{Bl}_{6}\mathbb{P}^2 & 3x + 3y + \frac{3}{xy} + \frac{3}{x} + \frac{3}{y} + 3xy + \frac{y}{x} + \frac{1}{xy^2} + x^{2}y\\
    \hline
    \end{longtable}
    \end{figure}
\end{example}

\subsection{Correspondence between the MMP and Landau-Ginzburg models}
    
    In this subsection we prove \cref{conj: mirror of extremal contractions} and \cref{conj: mirror of MMP}.
    
\begin{theorem}\label{theorem: surface MMP and LG models}
    \cref{conj: mirror of extremal contractions} and \cref{conj: mirror of MMP} hold in dimension 2.
\end{theorem}

\begin{proof}
    A divisorial contraction $X \rightarrow X'$ is a contraction of a $(-1)$-curve $E$. Let $\widetilde{f}$ be a parametrized toric LG model of $X$. Let $f_{t} = \widetilde{f}_{tE}$ be the associated toric LG model of $(X,tE)$ and $f' = \lim\limits_{t \rightarrow + \infty} f_{t}$. Then the regularized period of $f'$ is given by
    $$
    \hat{P}_{f'}(t) = 1 + \sum\limits_{\substack{\beta \in H_{2}(X;\mathbb{Z}) \\ E \cdot \beta = 0}} (-K_{X} \cdot \beta)!\langle \tau_{-K_{X} \cdot \beta - 2} \mathbf{1} \rangle_{\beta}^{X} \cdot t^{-K_{X} \cdot \beta}.
    $$
    Since $X$ and $X'$ are del Pezzo surfaces, $H_{2}(X;\mathbb{Z})$ and $H_{2}(X';\mathbb{Z})$ are generated by the classes of curves. There is a 1-1 correspondence between the linearly equivalent classes of curves $C'$ on $X'$ and the linear equivalence classes of curves $C$ on $X$ such that $C \cdot E = 0$, and it's well-known that the Gromov-Witten invariants of such $\beta$ are the same for a blow-up (for instance, see \cite{Manolache_2012}). Hence $\hat{P}_{f'} = \hat{G}_{X'}(t)$. The toric condition and the Calabi-Yau condition directly follows, so $f'$ is a toric LG model of $X'$. This is the divisorial contraction case of \cref{conj: mirror of extremal contractions}. In dimension 2 there are only 2 cases of Mori fibre space, namely $\mathbb{F}_{0}/\mathbb{P}^1$ and $\mathbb{F}_{1}/\mathbb{P}^1$. These are toric varieties, so the result follows from \cref{prop: LG model for toric Mori fibre space}. This completes the proof of \cref{conj: mirror of extremal contractions}.

    In dimension 2, a minimal model program on a smooth del Pezzo surface is simply a sequence of contractions of $(-1)$-curves
    $$
    \begin{tikzcd}
        X \arrow[r] & X_{1} \arrow[r] & \cdots \arrow[r] & X_{m} \arrow[d] \\
         & & & Z
    \end{tikzcd}
    $$
    where $X_{m}/Z$ is $\mathbb{P}^2/pt$ or $\mathbb{F}_{0}/\mathbb{P}^1$ or $\mathbb{F}_{1}\mathbb{P}^1$. Hence it suffices to consider a divisorial contraction and a Fano fibration. The case of divisorial contractions follows from \cref{conj: mirror of extremal contractions}. Let $X \rightarrow \mathbb{P}^1$ be a Fano fibration with a general fibre $F$, and $\widetilde{f}$ be a parametrized toric LG model of $X$. Let $f_{t} = \widetilde{f}_{tF}$ be the associated toric LG model of $(X,tF)$ and $f' = \lim\limits_{t \rightarrow + \infty} f_{t}$. Then the regularized period of $f'$ is given by
    $$
    \hat{P}_{f'}(t) = 1 + \sum\limits_{\substack{\beta \in H_{2}(X;\mathbb{Z}) \\ F \cdot \beta = 0}} (-K_{X} \cdot \beta)!\langle \tau_{-K_{X} \cdot \beta - 2} \mathbf{1} \rangle_{\beta}^{X} \cdot t^{-K_{X} \cdot \beta}.
    $$
    We have $\langle \tau_{-K_{X} \cdot \beta - 2} \mathbf{1} \rangle_{\beta}^{X} \neq 0$ only if $\beta$ is a multiple of the class of $F$. Hence we can blow down the $(-1)$-curves and compute the Gromov-Witten invariants on $\mathbb{F}_{0}$ or $\mathbb{F}_{1}$, where it can be verified explicitly. This completes the proof of \cref{conj: mirror of extremal contractions}.
\end{proof}

\subsection{Local geography}
    In this subsection we prove \cref{conj: local geography} for surfaces.

    \begin{theorem}
        \cref{conj: local geography} holds in dimension 2.
    \end{theorem}

    \begin{proof}
    In dimension 2, a minimal model program on a smooth del Pezzo surface is just a sequence of divisorial contractions
    $$
    \begin{tikzcd}
        X \arrow[r] & X_{1} \arrow[r] & \cdots \arrow[r] & X_{m} \arrow[d] \\
         & & & Z
    \end{tikzcd}
    $$
    where $X_{m}/Z$ is $\mathbb{P}^2/pt$ or $\mathbb{F}_{0}/\mathbb{P}^1$ or $\mathbb{F}_{1}\mathbb{P}^1$. Hence it suffices to show that we have the torus of given dimension for a divisorial contraction and a Fano fibration. Let $\widetilde{f}$ be a parametrized toric LG model of $X$. Let $X \rightarrow Y$ be a divisorial contraction of an exceptional divisor $E$. Then we can construct a deformation of $\widetilde{f}$ given by
    $$
    \widetilde{f}_{t,D} = \widetilde{f}_{D + tE}.
    $$
    In other words, a toric LG model of $(X,D)$ deforms to a toric LG model of $(X,D+tE)$. Let $\widetilde{f}' = \lim\limits_{t \rightarrow + \infty} \widetilde{f}_{t}$. Then by a parametrized version of \cref{conj: mirror of extremal contractions}, $\widetilde{f}'$ together with the projection of $\mathrm{Pic}(X)$ to $\mathrm{Pic}(X')$ is a parametrized toric LG model of $X'$. Hence we obtain the torus parametrizing the toric LG models of $X'$.

    Let $X \rightarrow \mathbb{P}^1$ be a fibration with a general fibre $F$. Then we can construct a deformation of $\widetilde{f}$ given by
    $$
    \widetilde{f}_{t,D} = \widetilde{f}_{D + tF}.
    $$
    Let $\widetilde{f}' = \lim\limits_{t \rightarrow + \infty} \widetilde{f}_{t}$. Then by a parametrized version of \cref{conj: mirror of MMP}, $\widetilde{f}'$ together with the projection of $\mathrm{Pic}(X)$ to $\mathrm{Pic}(X/Z)$ is a family of Laurent polynomials parametrized by a torus of expected dimension, whose maximal degeneration gives the toric Landau-Ginzburg models of $F$. This completes the proof.
    \end{proof}

    \begin{example}\label{example:local geography}
    Let $X$ be the blow-up of $\mathbb{P}^2$ at 2 points $P_{1},P_{2}$. Then $\mathrm{Eff}_{\mathbb{R}}(X)$ is generated by 2 exceptional divisors $E_{1},E_{2}$ and the strict transformation $\widetilde{L}$ of the line $L$ passing through $P_{1},P_{2}$. Its geography is given by the following:

    \begin{figure}[H]
    \centering
    \ifx\du\undefined
  \newlength{\du}
\fi
\setlength{\du}{15\unitlength}
\begin{tikzpicture}
\pgftransformxscale{1.000000}
\pgftransformyscale{-1.000000}
\definecolor{dialinecolor}{rgb}{0.000000, 0.000000, 0.000000}
\pgfsetstrokecolor{dialinecolor}
\definecolor{dialinecolor}{rgb}{1.000000, 1.000000, 1.000000}
\pgfsetfillcolor{dialinecolor}
\pgfsetlinewidth{0.100000\du}
\pgfsetdash{}{0pt}
\pgfsetdash{}{0pt}
\pgfsetbuttcap
{
\definecolor{dialinecolor}{rgb}{0.000000, 0.000000, 0.000000}
\pgfsetfillcolor{dialinecolor}
\definecolor{dialinecolor}{rgb}{0.000000, 0.000000, 0.000000}
\pgfsetstrokecolor{dialinecolor}
\draw (25.000000\du,10.000000\du)--(20.000000\du,19.000000\du);
}
\pgfsetlinewidth{0.100000\du}
\pgfsetdash{}{0pt}
\pgfsetdash{}{0pt}
\pgfsetbuttcap
{
\definecolor{dialinecolor}{rgb}{0.000000, 0.000000, 0.000000}
\pgfsetfillcolor{dialinecolor}
\definecolor{dialinecolor}{rgb}{0.000000, 0.000000, 0.000000}
\pgfsetstrokecolor{dialinecolor}
\draw (25.000000\du,10.000000\du)--(30.000000\du,19.000000\du);
}
\pgfsetlinewidth{0.100000\du}
\pgfsetdash{}{0pt}
\pgfsetdash{}{0pt}
\pgfsetbuttcap
{
\definecolor{dialinecolor}{rgb}{0.000000, 0.000000, 0.000000}
\pgfsetfillcolor{dialinecolor}
\definecolor{dialinecolor}{rgb}{0.000000, 0.000000, 0.000000}
\pgfsetstrokecolor{dialinecolor}
\draw (20.000000\du,19.000000\du)--(30.000000\du,19.000000\du);
}
\pgfsetlinewidth{0.100000\du}
\pgfsetdash{}{0pt}
\pgfsetdash{}{0pt}
\pgfsetbuttcap
{
\definecolor{dialinecolor}{rgb}{0.000000, 0.000000, 0.000000}
\pgfsetfillcolor{dialinecolor}
\definecolor{dialinecolor}{rgb}{0.000000, 0.000000, 0.000000}
\pgfsetstrokecolor{dialinecolor}
\draw (30.000000\du,19.000000\du)--(22.500000\du,14.500000\du);
}
\pgfsetlinewidth{0.100000\du}
\pgfsetdash{}{0pt}
\pgfsetdash{}{0pt}
\pgfsetbuttcap
{
\definecolor{dialinecolor}{rgb}{0.000000, 0.000000, 0.000000}
\pgfsetfillcolor{dialinecolor}
\definecolor{dialinecolor}{rgb}{0.000000, 0.000000, 0.000000}
\pgfsetstrokecolor{dialinecolor}
\draw (20.000000\du,19.000000\du)--(27.500000\du,14.500000\du);
}
\pgfsetlinewidth{0.100000\du}
\pgfsetdash{}{0pt}
\pgfsetdash{}{0pt}
\pgfsetbuttcap
{
\definecolor{dialinecolor}{rgb}{0.000000, 0.000000, 0.000000}
\pgfsetfillcolor{dialinecolor}
\definecolor{dialinecolor}{rgb}{0.000000, 0.000000, 0.000000}
\pgfsetstrokecolor{dialinecolor}
\draw (22.500000\du,14.500000\du)--(27.500000\du,14.500000\du);
}
\definecolor{dialinecolor}{rgb}{0.000000, 0.000000, 0.000000}
\pgfsetstrokecolor{dialinecolor}
\node[anchor=west] at (23.500000\du,13.000000\du){$\mathbb{P}^1 \times \mathbb{P}^1$};
\definecolor{dialinecolor}{rgb}{0.000000, 0.000000, 0.000000}
\pgfsetstrokecolor{dialinecolor}
\node[anchor=west] at (24.500000\du,15.000000\du){$X$};
\definecolor{dialinecolor}{rgb}{0.000000, 0.000000, 0.000000}
\pgfsetstrokecolor{dialinecolor}
\node[anchor=west] at (24.500000\du,18.000000\du){$\mathbb{P}^2$};
\definecolor{dialinecolor}{rgb}{0.000000, 0.000000, 0.000000}
\pgfsetstrokecolor{dialinecolor}
\node[anchor=west] at (21.500000\du,16.000000\du){$\mathrm{Bl}_{P_{1}}\mathbb{P}^2$};
\definecolor{dialinecolor}{rgb}{0.000000, 0.000000, 0.000000}
\pgfsetstrokecolor{dialinecolor}
\node[anchor=west] at (25.500000\du,16.000000\du){$\mathrm{Bl}_{P_{2}}\mathbb{P}^2$};
\definecolor{dialinecolor}{rgb}{0.000000, 0.000000, 0.000000}
\pgfsetstrokecolor{dialinecolor}
\node[anchor=west] at (22.000000\du,12.000000\du){$\mathbb{P}^1$};
\definecolor{dialinecolor}{rgb}{0.000000, 0.000000, 0.000000}
\pgfsetstrokecolor{dialinecolor}
\node[anchor=west] at (21.000000\du,16.000000\du){};
\definecolor{dialinecolor}{rgb}{0.000000, 0.000000, 0.000000}
\pgfsetstrokecolor{dialinecolor}
\node[anchor=west] at (20.000000\du,16.000000\du){$\mathbb{P}^1$};
\definecolor{dialinecolor}{rgb}{0.000000, 0.000000, 0.000000}
\pgfsetstrokecolor{dialinecolor}
\node[anchor=west] at (24.500000\du,9.000000\du){pt};
\definecolor{dialinecolor}{rgb}{0.000000, 0.000000, 0.000000}
\pgfsetstrokecolor{dialinecolor}
\node[anchor=west] at (26.500000\du,12.000000\du){$\mathbb{P}^1$};
\definecolor{dialinecolor}{rgb}{0.000000, 0.000000, 0.000000}
\pgfsetstrokecolor{dialinecolor}
\node[anchor=west] at (29.000000\du,16.000000\du){$\mathbb{P}^1$};
\definecolor{dialinecolor}{rgb}{0.000000, 0.000000, 0.000000}
\pgfsetstrokecolor{dialinecolor}
\node[anchor=west] at (24.500000\du,20.000000\du){pt};
\end{tikzpicture}
\caption{Geography of $\mathrm{Bl}_{P_{1},P_{2}}\mathbb{P}^2$}
\end{figure}
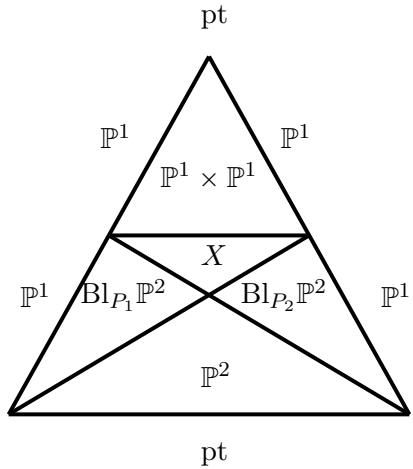
\end{example}

    Let $\Sigma$ be the associated fan and $X_{\Sigma}$ the associated toric variety. A parametrized toric LG model of $X$ is given by
    $$
    \widetilde{f} = x + y + \frac{a_{1}}{xy} + \frac{a_{2}}{x} + \frac{a_{3}}{y}
    $$
    which corresponds to the maximal torus of $X_{\Sigma}$. There are 5 rays in $\Sigma$ on the boundary of $\mathrm{Eff}_{\mathbb{R}}(X)$, which corresponds to
    \begin{align*}
        f_{1} &= x + y + \frac{a_{2}}{x} + \frac{a_{3}}{y}, \\
        f_{2} &= x + \frac{a'_{1}}{xy} + \frac{a_{2}}{x} + \frac{1}{y}, \\
        f_{3} &= x + y + \frac{a_{1}}{xy} + \frac{a_{2}}{x}, \\
        f_{4} &= x + y + \frac{a_{1}}{xy} + \frac{a_{3}}{y}, \\
        f_{5} &= y + \frac{a''_{1}}{xy} + \frac{1}{x} + \frac{a_{3}}{y}.
    \end{align*}
    The 5 rays in the boundary correspond to the parametrized LG models of the 5 central models of rank 2 under $X$. Similarly, the 5 faces in the boundary give
    \begin{align*}
        f'_{1} &= x + \frac{a_{2}}{x} + \frac{1}{y}, \\
        f'_{2} &= x + \frac{1}{xy} + \frac{a_{2}}{x}, \\
        f'_{3} &= x + y + \frac{a_{1}}{xy}, \\
        f'_{4} &= y + \frac{1}{xy} + \frac{a_{3}}{y}, \\
        f'_{5} &= y + \frac{1}{x} + \frac{a_{3}}{y}.
    \end{align*}
    They correspond to the parametrized LG models of the 5 central models of rank 1 under $X$.

\subsection{A local-to-global construction}

    Recall that a model is a variety together with a birational map (cf. \cref{def:birational model}). As a result, to parametrize toric LG models for all central models, one must put the birational maps into the parametrization. In this subsection, we present a local-to-global construction, which begins with a parametrization for a fixed central model and ends with a parametrization for a family of central models lying in the same isomorphism class. More precisely, the result can be stated as follows:

  \begin{proposition}\label{prop:local-to-global}
      Let $X$ be a central model such that a parametrized LG model of $X$ exists. Then there exists a moduli space $M(X)$ parametrizing toric LG models of $(X',D')$, where $X'$ is a central model isomorphic to $X$ and $D' \in \mathrm{Pic}(X') \otimes \mathbb{C}$. Moreover, $M(X)$ admits a natural action of $\mathrm{Bir}(X)$.
  \end{proposition}

  \begin{proof}
    Write $r = \rho(X) = \mathrm{rank} \mathrm{Pic}(X)$. A parametrized toric LG model $\widetilde{f}$ of $X$ is a family of toric LG models parametrized by $\mathrm{Pic}(X) \otimes \mathbb{C} \cong \mathbb{C}^{r}$. Notice that a toric LG model of $(X,D)$ is also a toric LG model of $(X, D + 2\pi i D')$ for any Cartier divisor $D'$. Hence we can consider $\widetilde{f}$ to be a family of toric LG models parametrized by $\mathrm{Pic}(X) \otimes \mathbb{C}^{*} \cong \mathbb{C}^{*r}$ under the ``natural'' morphism $\mathrm{Pic}(X) \otimes \mathbb{C} \rightarrow \mathrm{Pic}(X) \otimes \mathbb{C}^{*}$ given by $z \mapsto e^{-z}$. Write $G = \mathrm{Bir}(X)$ and consider the product $G \times \mathbb{C}^{*r}$. There is a free action of $\mathrm{Aut}(X)$ on $G \times \mathbb{C}^{*r}$ where
    \begin{itemize}
        \item the action on the first component $G$ is the left multiplication;
        \item the action on the second component $\mathbb{C}^{*r}$ is the natural action induced by the action of $\mathrm{Aut}(X)$ on $\mathrm{Pic}(X)$.
    \end{itemize}
    Now we take the quotient space
    $$
    M^{\circ}(X) = (G \times \mathbb{C}^{*r}) / \mathrm{Aut}(X).
    $$
    The quotient is a family of affine spaces parametrized by $G/\mathrm{Aut}(X)$, with a natural action of $G$ given as follows:
    \begin{itemize}
        \item the action on the first component $G$ is the left multiplication;
        \item the action on the second component $\mathbb{C}^{*r}$ is trivial.
    \end{itemize}
    This completes the construction.
  \end{proof}

\subsection{Global geography}\label{section:geography of LG models for surfaces}

    In this subsection, we prove a weaker version of \cref{theorem: surface geography} in dimension 2.

\begin{theorem}\label{theorem: surface geography}
    There exists a moduli space $M'_{2}$ parametrizing LG models of all central models under some smooth del Pezzo surfaces.
\end{theorem}

    Note that we require the central models to be under some smooth del Pezzo surfaces, since we don't have a definition of toric Landau-Ginzburg models of a central model $X/Z$ when $X$ is not Fano.

\begin{proof}[Proof of \cref{theorem: surface geography}]
    The proof consists of two parts. 
    
    \textbf{Step 1:} For a fixed smooth del Pezzo surface $X$, we show that we can compactify the moduli space of toric LG models such that the boundary corresponds to central models of lower ranks. Write $r = \rho(X) = \mathrm{rank} \mathrm{Pic}(X)$. A parametrized toric LG model $\widetilde{f}$ of $X$ is a family of toric LG models parametrized by $\mathrm{Pic}(X) \otimes \mathbb{C} \cong \mathbb{C}^{r}$. Notice that a toric LG model of $(X,D)$ is also a toric LG model of $(X, D + 2\pi i D')$ for any Cartier divisor $D'$. Hence we can consider $\widetilde{f}$ to be a family of toric LG models parametrized by $\mathrm{Pic}(X) \otimes \mathbb{C}^{*} \cong \mathbb{C}^{*r}$ under the ``natural'' morphism $\mathrm{Pic}(X) \otimes \mathbb{C} \rightarrow \mathrm{Pic}(X) \otimes \mathbb{C}^{*}$ given by $z \mapsto e^{-z}$. We compactify it to the toric variety $X_{\Sigma}$, where $\Sigma$ is the fan given by the cone of effective divisors $\mathrm{Eff}_{\mathbb{R}}(X)$. By \cref{theorem: surface MMP and LG models}, the points in the boundary correspond to a sequence of contractions of $(-1)$-curves or a $\mathbb{P}^1$-fibration, which is a central model of lower rank.
    
    \textbf{Step 2:} We glue up these moduli spaces together to obtain a ``global'' moduli space. We have a local-to-global construction by \cref{prop:local-to-global}. It can be directly verified that this construction is compatible with the geography on $\mathrm{Eff}_{\mathbb{R}}(X)$. Hence we also have a compactification of $M^{\circ}(X)$ given by
    $$
    M(X) = (G \times X_{\Sigma}) / \mathrm{Aut}(X).
    $$
    This completes the construction.
\end{proof}

\begin{example}[Case of rational surfaces]
    We look at some pieces of the total space $M'$ corresponding to central models of small rank. We begin with $M_{0} = p$, where $p$ is a point. We regard $p$ as \emph{the moduli of divisorial degenerations corresponding to the identity map}.

    \textbf{Rank 1:} There are two types of central models of rank 1 for surfaces: the projective planes $\mathbb{P}^2$ and the ruled surfaces $\mathbb{F}_{n}/\mathbb{P}^1$ for $n = 0,1$.

    \textbf{Projective planes:} In this case, the LG models are of the form $x + y + \frac{a}{xy}$. The automorphism group $\mathrm{Aut}(\mathbb{P}^2)$ acts trivially on $\mathrm{Pic}(X)$. By \cref{prop:local-to-global}, the moduli space of LG models is just $G/\mathrm{Aut}(\mathbb{P}^2) \times \mathbb{C}$. We attach this to $M_{0}$ by gluing $ G/\mathrm{Aut}(\mathbb{P}^2) \times \{0\}$ to $p$. Hence for each central model $g: \mathbb{P}^2 \dashrightarrow 
    \mathbb{P}^2$, $M(\mathbb{P}^2)$ is an affine line.

    \textbf{Ruled surfaces:} Ruled surfaces are of the form $\mathbb{F}_{n}/\mathbb{P}^1$ for $n = 0,1$. A moduli space of fibring degenerations corresponding to a ruled surface is $x + y + \frac{a}{x}$. In this case, the subgroup $\mathrm{Bir}(\mathbb{F}_{n} \rightarrow \mathbb{P}^1)$ of $G$ that fixes the ruled structure is a \emph{Jonquière subgroup} $\mathrm{Jon}(\mathbb{P}^1)$ of $G$. The induced action of $\mathrm{Jon}(\mathbb{P}^1)$ on $\mathrm{Cl}(\mathbb{P}^1)$ is trivial. By \cref{prop:local-to-global}, the moduli space is $G/\mathrm{Jon}(\mathbb{P}^1) \times \mathbb{C}$. We attach this to $M_{0}$ by gluing $G/\mathrm{Jon}(\mathbb{P}^1) \times \{0\}$ to $p$. For each ruled structure $\mathbb{P}^2 \dashrightarrow \mathbb{P}^1$, $T(\mathbb{P}^1)$ is an affine line.

    Next, we attach the moduli space of LG models of $\mathbb{F}_{1}$. A moduli space is given by $x + y + \frac{a}{x} + \frac{b}{xy}$. The action of $\mathrm{Aut}(\mathbb{F}_{n})$ on $\mathrm{Pic}(\mathbb{F}_{n})$ is trivial. By \cref{prop:local-to-global}, the total moduli space is $G/\mathrm{Aut}(X) \times \mathbb{C}^2$. We then attach the space to $M_{0}$ by gluing the line $b=0$ naturally to $T(\mathbb{P}^1)$.

    \textbf{Rank 2:} There are 3 types of central models of rank 2 for surfaces: The Hirzebruch surface $\mathbb{F}_{0},\mathbb{F}_{1}$, and the fibration $\mathrm{Bl}_{2}\mathbb{P}^2 \rightarrow \mathbb{P}^1$.

    \textbf{Hirzebruch surface $\mathbb{F}_{0}$:} A moduli space of LG model of $\mathbb{F}_{0}$ is given by $x + y + \frac{a}{x} + \frac{b}{y}$. Notice that in this case, the action of the automorphism group $\mathrm{Aut}(\mathbb{F}_{0})$ on $\mathrm{Pic}(\mathbb{F}_{0})$ is not trivial. Moreover, the moduli space of LG models is already partially attached in rank 1 case. We further attach the space to $M_{1}$ by gluing both lines $a=0$ and $b=0$ to the corresponding divisorial degenerations.

    \textbf{Hirzebruch surface $\mathbb{F}_{1}$:} The parametrized toric LG model of $\mathbb{F}_{1}$ is given by $x + y + \frac{a_{1}}{xy} + \frac{a_{1}a_{2}}{x}$. Let $a=a_{1}$ and $b = a_{1}a_{2}$. The moduli of LG models is already partially attached in the rank 1 case. We further attach the space to $M_{1}$ by gluing both lines $a=0$ and $b=0$ to the corresponding moduli of LG models and divisorial degenerations.

    \textbf{Fibration $\mathrm{Bl}_{2}\mathbb{P}^2 \rightarrow \mathbb{P}^1$:} The parametrized toric LG model of $\mathrm{Bl}_{2}\mathbb{P}^2$ after a change of basis is $x+y+\frac{a_{1}}{xy}+\frac{a_{2}}{x}+\frac{a_{3}}{y}$. The divisorial degenerations are described in \cref{example: Tyurin degenerations for surfaces} and other degenerations are described in \cref{example: parametrized LG model}. In particular, the two divisorial degenerations are given by $a_{1},a_{2} \rightarrow 0$ and $a_{1}, a_{3} \rightarrow 0$.  The compactification of the moduli space is described in \cref{example:local geography}. The action of an element $g \in \mathrm{Aut}(\mathrm{Bl}_{2}\mathbb{P}^2)$ on $\mathrm{Pic}(\mathrm{Bl}_{2}\mathbb{P}^2)$ either fix everything, or permute the two exceptional divisors, that is, permute $a_{2}$ and $a_{3}$ in the parametrized LG model. We can attach the space $G \times X_{\Sigma}/\mathrm{Aut}(\mathrm{Bl}_{2}\mathbb{P}^2)$ into $M'_{2}$ by the above information.
\end{example}

\section{General case (threefolds)}\label{section: threefold}

    In this section, we study the threefold case. First, we list some preliminaries about Landau-Ginzburg models for terminal Fano threefolds.

\subsection{Landau-Ginzburg models for terminal Fano threefolds}

    In this section, we briefly recall the Gromov-Witten theory for terminal Fano threefolds. Firstly, orbifold Gromov-Witten theory was defined in \cite{ChenRuan2001}. On the other side, terminal Fano threefolds deform to orbifolds if $X$ has ordinary terminal singularities (which is called the \emph{$\mathbb{Q}$-smoothing} of $X$):

\begin{theorem}[cf. \cite{Sano15}]
    Let $X$ be a Fano threefold with only ordinary terminal singularities, then $X$ deforms to a Fano threefold with only terminal cyclic quotient singularities.
\end{theorem}

    Hence we can define the quantum period for these varieties as follows:

\begin{convention}
    Let $X$ be a Fano threefold with only ordinary terminal singularities and $\beta \in H_{2}(X;\mathbb{Z})$ be a curve class. We write $M_{g,n}(X,\beta) = M_{g,n}(X,J,\beta,\vec{1})$, where $J$ is the natural complex structure on $X$ and $\vec{1} = (1,1,\cdots,1)$. In other words, we always assume that the marked points are smooth.
\end{convention}

\begin{definition}
    Let $X$ be a Fano threefold with only ordinary terminal singularities. Then the \emph{regularized quantum period} of $X$ is
    $$
    \widetilde{I}_{0}^{X,D}(t) = 1 + \sum\limits_{\beta \in K}(-K_{X} \cdot \beta)!\langle \tau_{-K_{X}\cdot\beta -2}\mathbf{1} \rangle_{\beta} \cdot e^{-D \cdot \beta}t^{-K_{X} \cdot \beta}
    $$
    where
    $$
    \langle \tau_{-K_{X}\cdot\beta -2}\mathbf{1} \rangle_{\beta} = \int_{[\overline{M}_{0,1}(X,\beta)]^{virt}} \psi_{1}^{-K_{X}\cdot\beta -2}ev_{1}^{*}(\mathbf{1}_{X}).
    $$
\end{definition}
    
    Hence we can generalize the definitions of toric Landau-Ginzburg models in Section 2 to many terminal Fano threefolds. It is also conjectured that every terminal Fano threefold can be deformed to a Fano threefold with only terminal cyclic quotient singularities.

\subsection{Existence of LG models}

    In this subsection, we recall some results on the existence of LG models in dimension 3.

    \begin{theorem}[cf. \cite{Przyjalkowski_2008}; \cite{Przyjalkowski_2013}; \cite{ACGK12}; \cite{CCG+}]
        Smooth Fano threefolds have toric Landau-Ginzburg models.
    \end{theorem}

    Moreover, explicit LG models are constructed, for instance, in \cite{CCG+} when $-K_{X}$ is very ample, and \cite{CCGK} in general. When $-K_{X}$ is very ample, we can construct a parametrized toric LG model from parametrized Minkowski polynomials:

    \begin{theorem}
         Let $X$ be a smooth Fano variety such that $-K_{X}$ is very ample. Then $X$ has a parametrized toric LG model.
    \end{theorem}

    \begin{proof}
         In this case $X$ has a Gorenstein terminal toric degeneration $X_{T}$ with reflexive polytope $P$. Denote by $\mathcal{X} \rightarrow \mathbb{A}^1$ the toric degeneration from $X$ to $X_{T}$. Then by \cite{Kawamata_deformation_canonical_singularities}, $\mathcal{X}$ has canonical singularities. We take prime divisors $D_{1},\cdots, D_{r}$ which form a basis of the Picard group $\mathrm{Pic}(X)$. Then by \cref{construction: degenerating divisors}, after a finite base change, we can degenerate every $D_{i}$ to an effective divisor $D_{i,T}$ on $X_{T}$. After taking a $\mathbb{Q}$-factorization of $\mathcal{X}$, we can degenerate every $D_{i}$ to an effective $\mathbb{Q}$-Cartier divisor $D'_{i,T}$ on $X'_{T}$, where $X'_{T}$ is a crepant extraction of $X_{T}$. By \cref{prop: toric LG model for toric pairs} we can construct a parametrized toric LG model $\widetilde{f}'$ on $\widetilde{X}_{T}$, which can be restricted to the subspace of $\mathrm{Pic}(\widetilde{X}_{T})$ generated by $\widetilde{D}_{i,T}$. Finally, we can modify the coefficients of $\widetilde{f}$ by the Minkowski data on $P$ to obtain a parametrized toric LG model of $X$ (cf. \cite{DHKOP}, Section 3).
    \end{proof}

\subsection{Threefold extremal contractions}

    We recall some results about threefold extremal contractions, which are used in this section and \cref{section: applications}.

\begin{theorem}[Divisorial contractions for smooth threefolds, cf. \protect{\cite[Theorem 3.3]{Mori1982}}]
    Let $X$ be a smooth projective variety of dimension 3. Let $R$ be an extremal ray on $X$, and let $\phi: X \rightarrow Y$ be the extremal contraction corresponding to $R$. Then one of the following holds:
    \begin{enumerate}
        \item $\phi$ is a conic bundle of one of the following types:
        \begin{longtable}{|C|C|C|C|}
            \hline
            \text{Type} & f & l(R) & C_{R} \\
            \hline
            \multirow{2}{*}{C1} & \text{contains a reducible or non-reduced fibre} & 1 & \text{a reduced and irreducible} \\
            & & & \text{component of a degenerate fibre} \\
            \hline
            \text{C2} & \text{a bundle of a locally free sheaf of rank 2} & 2 & \text{a fibre}\\
            \hline
        \end{longtable}
        \item $\phi$ is a del Pezzo fibration of one of the following types:
        \begin{longtable}{|C|C|C|C|}
            \hline
            \text{Type} & \text{general fibres} & l(R) & C_{R} \\
            \hline
            \text{D1} & \text{smooth del Pezzo surfaces of degree } \leq 6 & 1 & \text{a line in a fibre} \\
            \hline
            \text{D2} & \text{quadric surfaces} & 2 & \text{a line in a fibre} \\
            \hline
            \text{D3} & \mathbb{P}^2 & 3 & \text{a line in a fibre}\\
            \hline
        \end{longtable}
        \item $\phi$ is a divisorial contraction of one of the following types:
        \begin{longtable}{|C|C|C|C|}
            \hline
            \text{Type} & E\text{ and }\phi(E) & l(R) & C_{R} \\
            \hline
            \multirow{2}{*}{E1} & E\text{ is a ruled surface,} & \multirow{2}{*}{$1$} & \multirow{2}{*}{a ruling fibre of $E$} \\
            & \phi(E)\text{ is a smooth curve and }Y \text{is smooth} & & \\
            \hline
            \multirow{2}{*}{E2} & E \simeq \mathbb{P}^2, \mathcal{O}_{E}(E) \simeq \mathcal{O}_{\mathbb{P}^{2}}(-1), & \multirow{2}{*}{$2$} & \multirow{2}{*}{a line in $E$} \\
            & \phi(E)\text{ is a smooth point of }Y & & \\
            \hline
            \multirow{2}{*}{E3} & E \simeq \mathbb{P}^1 \times \mathbb{P}^1, \mathcal{O}_{E}(E) \simeq \mathcal{O}_{\mathbb{P}^1 \times \mathbb{P}^1}(-1,-1), & \multirow{2}{*}{$1$} & \multirow{2}{*}{a line in $E$} \\
            & \phi(E)\text{ is an ordinary double point of }Y & & \\
            \hline
            \multirow{2}{*}{E4} & E\text{ is a quadric cone in }\mathbb{P}^3, \mathcal{O}_{E}(E) \simeq \mathcal{O}_{E}\otimes \mathcal{O}_{\mathbb{P}^3}(-1) & \multirow{2}{*}{$1$} & \multirow{2}{*}{a ruling line of $E$} \\
            & \phi(E)\text{ is a double (cDV)-point of }Y & & \\
            \hline
            \multirow{2}{*}{E5} & E \simeq \mathbb{P}^2, \mathcal{O}_{E}(E) \simeq \mathcal{O}_{\mathbb{P}^{2}}(-2) & \multirow{2}{*}{$1$} & \multirow{2}{*}{a line in $E$} \\
            & \phi(E)\text{ is a quadruple point of }Y & & \\
            \hline            
        \end{longtable}
        \end{enumerate}
        Here $l(R)$ is the \emph{length} of $R$, i.e. 
        $$
        l(R) = \min\{ (-K_{X})\cdot C \mid C \in R \text{ is a rational curve} \},
        $$ 
        and $C_{R}$ is a \emph{primitive generator} of $R$, i.e. $C_{R} \in R$ and $(-K_{X})\cdot C_{R} = l(R)$.
\end{theorem}

\begin{theorem}[Mori fibre spaces for smooth threefolds, cf. \protect{\cite{Cutkosky1988}}]
    Let $\phi: X \rightarrow Y$ be a Mori fibration from a terminal Gorenstein threefold to a surface. Then $Y$ is smooth and $\phi$ is a flat embedded conical fibration.
\end{theorem}

\subsection{Divisorial contractions}

    Now we prove the case of divisorial contractions of \cref{conj: mirror of extremal contractions} in dimension 3.

    \begin{theorem}\label{theorem: threefold divisorial contraction and LG models}
        Let $X$ be a smooth Fano threefold with a parametrized toric Landau-Ginzburg model. Let $g: X \rightarrow Y$ be a $K_{X}$-negative divisorial contraction. Then the toric LG models of $X$ degenerate to the toric LG models of $Y$. In particular, $Y$ has a toric LG model.
    \end{theorem}

    \begin{proof}
        Let $E$ be the exceptional divisor of the divisorial contraction and $f$ be a toric LG model of $X$. Let $f_{t}$ be a degeneration of $f$ in the direction of $E$, i.e. $f_{t}$ is a toric LG model of $(X,tE)$. Let $f' = \lim\limits_{t \rightarrow +\infty} f_{t}$. We claim that $f'$ is a toric LG model of $Y$. Indeed, the period given by $f'$ is
        $$
        \hat{G}'_{X}(t) = 1 + \sum\limits_{\substack{\beta \in H_{2}(X;\mathbb{Z}) \\ E \cdot \beta = 0}}(-K_{X} \cdot \beta)!\langle \tau_{-K_{X}\cdot\beta -2}\mathbf{1} \rangle_{\beta}^{X} \cdot t^{-K_{X} \cdot \beta}.
        $$
        On the other side, the quantum period of $Y$ is given by
        $$
        \hat{G}_{Y}(t) = 1 + \sum\limits_{\beta \in H_{2}(Y;\mathbb{Z})}(-K_{Y} \cdot \beta)!\langle \tau_{-K_{Y}\cdot\beta -2}\mathbf{1} \rangle_{\beta}^{Y} \cdot t^{-K_{Y} \cdot \beta}.
        $$
        Since $X$ and $Y$ are terminal Fano varieties, the singular homology $H_{2}(X;\mathbb{Z})$ and $H_{2}(Y;\mathbb{Z})$ are generated by algebraic 1-cycles. Hence the natural morphism $H_{2}(X;\mathbb{Z}) \rightarrow H_{2}(Y;\mathbb{Z})$ is surjective by Tsen's theorem. Let $l \in H_{2}(X;\mathbb{Z})$ be the primitive class of the extremal ray of the divisorial contraction. If the divisorial contraction is of type E1 or E2, then $Y$ is smooth and $E \cdot l = -1$. Hence for every class $\beta \in H_{2}(Y;\mathbb{Z})$, we can find a unique class $\widetilde{\beta} \in H_{2}(X;\mathbb{Z})$ such that $E \cdot \beta = 0$. If the divisorial contraction is of type E3, E4, or E5, then $Y$ is terminal and $K_{X} \cdot l = -1$. For every class $\beta \in H_{2}(Y;\mathbb{Z})$, let $\widetilde{\beta} \in H_{2}(X;\mathbb{Z})$ be a lifting of $\beta$ in $H_{2}(X;\mathbb{Z})$. Write
        $$
        g^{*}(K_{Y}) = K_{X} - aE
        $$
        where $a > 0$. Then we have
        \begin{align*}
            K_{X} \cdot \widetilde{\beta} &= K_{Y} \cdot \beta + a E \cdot \widetilde{\beta} \\
            K_{X} \cdot l &= a E \cdot l = -1.
        \end{align*}
        Hence $E \cdot (\widetilde{\beta} + ml) = 0$ for some integer $m$ if and only if $K_{Y} \cdot \beta$ is an integer. Notice that if $K_{Y} \cdot \beta$ is not an integer then $\langle \tau_{-K_{Y}\cdot\beta -2}\mathbf{1} \rangle_{\beta} = 0$. Hence there is a 1-1 correspondence between non-zero terms in $\hat{G}_{Y}(t)$ and non-zero terms in $\hat{G}'_{X}(t)$. Let $\widetilde{\beta} \in H_{2}(X;\mathbb{Z})$ and $\beta \in H_{2}(Y;\mathbb{Z})$ be a corresponding pair such that $\langle \tau_{-K_{Y}\cdot \beta -2}\mathbf{1} \rangle_{\beta} \neq 0$. Then the curve class $\beta$ is a moving class of $Y$. One can easily check that $D \cdot \widetilde{\beta} \geq 0$ for any effective $\mathbb{R}$-divisor $D$ on $X$. Hence $\widetilde{\beta}$ is contained in the cone of moving curves of $X$. In particular, we can find a curve $C \in \widetilde{\beta}$ which doesn't intersect $E$. Consider the Cartesian product
        $$
        \begin{tikzcd}
            \overline{M}_{0,1}(X,\beta) \arrow[d,"ev_{1}"] \arrow[r,"g"] & \overline{M}_{0,1}(Y,\beta) \arrow[d,"ev_{1}"] \\
            X \arrow[r,"g"] & Y
        \end{tikzcd}
        $$
        Then $g^{*}\mathbf{1}_{Y} = \mathbf{1}_{X}$. Now we can follow the strategy of \cite{Conifold_transition}, Proposition 2.1 to show that $\langle \tau_{-K_{X}\cdot\widetilde{\beta} -2}\mathbf{1} \rangle_{\widetilde{\beta}}^{X} = \langle \tau_{-K_{Y}\cdot \beta -2}\mathbf{1} \rangle_{\beta}^{Y}$. More precisely, we can construct a deformation $\mathcal{W} \rightarrow \mathbb{A}^1$ satisfying the following conditions:
        \begin{enumerate}
            \item The general fibre is isomorphic to $Y$ if the divisorial contraction is of type E1, E2 or E5, and is isomorphic to a smoothing of $Y$ if the divisorial contraction is of type E3 or E4.
            \item The special fibre is isomorphic to $X \cup_{E} W$. Here $W$ is given by the following table:
            \begin{center}
            \begin{longtable}{|C|C|}
            \hline
            \text{Type} & W \\
            \hline
            \text{E1} & \mathbb{P}(\mathcal{O}_{C} \oplus \mathcal{N}_{C/Y}) \\
            \hline 
            \text{E2} & \mathbb{P}^3 \\
            \hline
            \text{E3} & \text{smooth quadric threefold} \\
            \hline
            \text{E4} & \text{quadric cone} \\
            \hline
            \text{E5} & \mathbb{P}(1,1,1,2) \\
            \hline
            \end{longtable}
            \end{center}
        \end{enumerate}
        For smooth Fano threefolds, divisorial contractions of type E4 can be deformed to divisorial contractions of type E3. Hence we can assume that the intersection is transversal and apply the degeneration formula to get
        $$
        \langle \tau_{-K_{Y}\cdot \beta -2}\mathbf{1} \rangle_{\beta}^{Y} = \sum\limits_{\widetilde{\gamma}} \langle \tau_{-K_{X}\cdot\widetilde{\gamma} -2}\mathbf{1} \mid \emptyset \rangle_{\widetilde{\gamma}}^{(X,E)}
        $$
        where the sum is taken over classes $\widetilde{\gamma}$ on 
        $X \cup_{E} W$ such that
        $$
        g_{*}\widetilde{\gamma} = \beta, \qquad E \cdot \widetilde{\gamma} = 0, \qquad \widetilde{\gamma}_{W} = 0.
        $$
        One can immediately see that such class $\widetilde{\gamma}$ is unique and is exactly the class $\widetilde{\beta}$. Hence the right-hand side equals to $\langle \tau_{-K_{X}\cdot\widetilde{\beta} -2}\mathbf{1} \rangle_{\widetilde{\beta}}^{X}$. This is the period condition. It follows that the corresponding degeneration of the Calabi-Yau compactification of $f$ is irreducible. Hence the degeneration is a Calabi-Yau compactification of $f'$. Finally, the toric condition is trivial.
    \end{proof}

\subsection{Mori fibre spaces}
    Now we prove the case of Mori fibre space of \cref{conj: mirror of extremal contractions} in dimension 3.

    \begin{theorem}\label{theorem: threefold Mori fibre spaces and LG models}
        Let $X$ be a smooth Fano threefold with a parametrized toric Landau-Ginzburg model. Let $h: X \rightarrow Z$ be a $K_{X}$-negative Mori fibre space and $F$ be a general fibre. Then the toric LG models of $X$ degenerate to the toric LG models of $F$. In particular, $F$ has a toric LG model.
    \end{theorem}

\begin{proof}
    Let $D = h^{*}A$ for some very ample divisor $A$ on $Z$. Let $f$ be a toric LG model of $X$. Let $f_{t}$ be a degeneration of $f$ in the direction of $D$, i.e. $f_{t}$ is a toric LG model of $(X,tD)$. Let $f' = \lim\limits_{t \rightarrow +\infty} f_{t}$. We claim that $f'$ is a toric LG model of $F$. Indeed, the regularized period of $f'$ is given by
    $$
    \hat{P}_{f'}(t) = 1 + \sum\limits_{\substack{\beta \in H_{2}(X;\mathbb{Z}) \\ D \cdot \beta = 0}}(-K_{X} \cdot \beta)!\langle \tau_{-K_{X}\cdot\beta -2}\mathbf{1} \rangle_{\beta}^{X} \cdot t^{-K_{X} \cdot \beta}.
    $$
    On the other side, the quantum period of $F$ is given by
    $$
    \hat{G}_{F}(t) = 1 + \sum\limits_{\beta \in H_{2}(F;\mathbb{Z})}(-K_{F} \cdot \beta)!\langle \tau_{-K_{F}\cdot\beta -2}\mathbf{1} \rangle_{\beta}^{F} \cdot t^{-K_{F} \cdot \beta}.
    $$
    Since $X$ and $F$ are terminal Fano varieties, the singular homology $H_{2}(X;\mathbb{Z})$ and $H_{2}(F;\mathbb{Z})$ are generated by algebraic 1-cycles. Let $l \in H_{2}(F;\mathbb{Z})$ be the class given by $D \cap F$. We claim that the natural morphism $H_{2}(F;\mathbb{Z}) \rightarrow H_{2}(X;\mathbb{Z})$ is injective. Indeed, suppose there are two effective curves $C_{1}$ and $C_{2}$ on $F$ such that $[C_{1} - C_{2}] = 0 \in H_{2}(X;\mathbb{Z})$. Since $F$ is a general fibre, $X$ is smooth near $F$. By Ehresmann's lemma, $X$ is homeomorphic to a product of $F$ and $Z$ in a analytic neighborhood of $F$. Hence we conclude that $[C_{1} - C_{2}] = 0 \in H_{2}(F;\mathbb{Z})$. Conversely, if $\widetilde{\beta} \in H_{2}(X;\mathbb{Z})$ is a curve class such that $D \cdot \widetilde{\beta} = 0$ and the evaluation map $ev_{1}: \overline{M}_{0,1}(X,\beta) \rightarrow X$ is surjective, then $\widetilde{\beta}$ must have a representative lying in $F$. Hence there is a 1-1 correspondence between moving classes $\beta \in H_{2}(F;\mathbb{Z})$ and moving classes $\widetilde{\beta} \in H_{2}(X;\mathbb{Z})$ such that $D \cdot \widetilde{\beta} = 0$. Since $\beta$ is a fibre class, for any stable map $\phi: C \rightarrow X$ in the class $\beta$, $\phi$ factors through the fibre containing $F$. Hence we have a morphism $\overline{M}_{0,1}(X,\beta) \rightarrow Z$ with general fibre $\overline{M}_{0,1}(F,\beta)$. Consider the inclusion $i: P \rightarrow Z$ of a closed point and the Cartesian product
    $$
    \begin{tikzcd}
        \overline{M}_{0,1}(F,\beta) \arrow[d,"ev_{1}"] \arrow[r,"i"] & \overline{M}_{0,1}(X,\beta) \arrow[d,"ev_{1}"] \\
        F \arrow[r,"i"] & X
    \end{tikzcd}
    $$
    For genus $0$ stable maps, the obstruction bundle of $\overline{M}_{0,1}(Z,0)$ is trivial. Hence $[\overline{M}_{0,1}(F,\beta)]^{virt} = i^{!}[\overline{M}_{0,1}(X,\beta)]^{virt}$. We have 
    \begin{align*}
        \langle \tau_{-K_{X}\cdot\beta -2}\mathbf{1} \rangle_{\beta} & = \psi^{-K_{X}\cdot\beta -2}\cdot ev_{1}^{*} \mathbf{1}_{X} \cdot [\overline{M}_{0,1}(X,\beta)]^{virt} \\
         & = \psi^{-K_{X}\cdot\beta -2}\cdot ev_{1}^{*} \mathbf{1}_{F} \cdot [\overline{M}_{0,1}(F,\beta)]^{virt} \\
         & = \langle \tau_{-K_{F}\cdot\beta -2}\mathbf{1} \rangle_{\beta}
    \end{align*}
    This is the period condition. Let $f''$ be the maximal degeneration of $f$ in the direction of $D$. Then $f''$ and $f'$ have the same period, and $f''$ satisfies the toric condition. Finally, the Newton polytope of $f''$ corresponds to a terminal weak Fano variety of dimension $\leq 2$, which is smooth. Hence the Calabi-Yau condition holds.
\end{proof}

\section{Application: Threefold elementary syzygies}\label{section: applications}

    As an application of the results in \cref{section: threefold}, we want to use this method to compute some elementary syzygies for smooth Fano threefolds. In the following discussion, the families are numbered as in the classification in \cite{MoriMukai}. The geography of smooth Fano varieties is studied in \cite{Matsuki1995}. The structure of the elementary relations can be directly read from the geography. In particular, we are interested in the cases where singular Fano threefolds or del Pezzo fibrations appear. In such cases, we can obtain the toric LG models of the singular Fano threefolds and the general fibre of the del Pezzo fibrations using the following algorithm:

\begin{algorithm}
    Let $X$ be a smooth Fano threefold of Picard number $\geq 2$. Let $f$ be a toric LG model of $X$ and $X_{T}$ be the corresponding toric degeneration of $X$. Then we compute the elementary syzygy of $X$ as follows:

    \textbf{Step 1:} Obtain the cone of effective divisors $\mathrm{Eff}_{\mathbb{R}}(X)$.

    \textbf{Step 2:} Take the extremal rays of $\mathrm{Eff}_{\mathbb{R}}(X)$ and take the primitive vectors $v_{1},v_{2},\cdots, v_{s}$. Then any effective divisor $D$ can be written as non-negative linear combination of $v_{1},v_{2},\cdots, v_{s}$, and the representation is unique modulo the linear dependence relations among these generators.

    \textbf{Step 3:} Construct the parametrized toric LG model $\widetilde{f}$ of $X$. Represent $\widetilde{f}$ in terms of the coordinates $a_{1}, \cdots, a_{s}$.

    \textbf{Step 4:} Now for every divisor $D$ in the boundary of $\mathrm{Eff}_{\mathbb{R}}(X)$, consider the Landau-Ginzburg model for the divisor $tD$ and let $t \to +\infty$.
\end{algorithm}

\begin{convention}
    In this section, we ignore the constant terms of the Laurent polynomials when we talk about toric Landau-Ginzburg models. In other words, we say that a Laurent polynomial $f$ is a toric Landau-Ginzburg model of $X$ if $f - c(f)$ is a toric Landau-Ginzburg model of $X$.
\end{convention}
    
\subsection{Picard number 1}
    First, we recall the toric LG model for smooth Fano threefolds of Picard number 1.

    \begin{statement}[Picard number 1, cf. \cite{Przyjalkowski_2018}, Table 1]
        The following table presents the toric LG model of smooth Fano threefolds of Picard number 1.

\begin{center}
\begin{longtable}{|C|C|C|C|}
\caption{Some toric LG models for threefolds of Picard number 1}
\label{table: toric LG model picard number 1} \\
    \hline
    \text{No.} & \text{Notation}  & \text{Description} & \text{A toric LG model} \\
    \hline
    1.1 & V_{2} & \text{A hypersurface of degree 6 in } \mathbb{P}(1,1,1,1,3) & \frac{(x+y+z+1)^{6}}{xyz} \\
    \hline
    1.2 & V_{4} & \text{General member a quartic hypersurface in $\mathbb{P}^4$} & \frac{(x+y+z+1)^{4}}{xyz} \\
    \hline
    \multirow{2}{*}{1.3} & \multirow{2}{*}{$V_{6}$} & \text{A complete intersection} & \multirow{2}{*}{$\frac{(x+1)^{2}(y+z+1)^{3}}{xyz}$} \\
    & & \text{of a cubic and a quadric in } \mathbb{P}^{5} & \\
    \hline
    1.4 & V_{8} & \text{A complete intersection of 3 quadrics in } \mathbb{P}^{6} & \frac{(x+1)^{2}(y+1)^{2}(z+1)^2}{xyz} \\
    \hline
    1.5 & V_{10} & \text{Gushel–Mukai threefold} & \frac{(xy+yz+xz+x+y+z+1)^{2}}{xyz} \\
    \hline
    \multirow{3}{*}{1.6} & \multirow{3}{*}{$V_{12}$} & \text{Section of half-spinor embedding} & \multirow{3}{*}{$\frac{(x+z+1)(x+y+z+1)(z+1)(y+z)}{xyz}$} \\
    & & \text{of a connected component of $OGr_{+}(5,10)$} & \\
    & & \text{by codimension 7 subspace} & \\
    \hline
    \multirow{2}{*}{1.7} & \multirow{2}{*}{$V_{14}$} & \text{Section of Plücker embedding of $Gr(2,6)$} & \frac{(x+y+z+1)^{2}}{x} \\
    & & \text{by codimension 5 subspace} &  + \frac{(x+y+z+1)(y+z+1)(z+1)^{2}}{xyz} \\
    \hline
    \multirow{2}{*}{1.8} & \multirow{2}{*}{$V_{16}$} & \text{Section of Plücker embedding of $SGr(3,6)$} & \multirow{2}{*}{$\frac{(x+y+z+1)(x+1)(y+1)(z+1)}{xyz}$} \\
    & & \text{ by codimension 3 subspace} & \\
    \hline
    \multirow{2}{*}{1.9} & \multirow{2}{*}{$V_{18}$} & \text{Section of the adjoint $G_2$-Grassmannian} & \multirow{2}{*}{$\frac{(x+y+z)(x+y+z+xz+yz+xy+xyz)}{xyz}$} \\
    & & \text{$G_{2}Gr(2,7)$ by codimension 2 subspace} & \\
    \hline
    1.10 & V_{22} & \text{Zero locus of $(\bigwedge^{2} \mathcal{U}^{\vee})^{\oplus 3}$ on $Gr(3,7)$} & \frac{(z+1)(x+y+1)(xy+z)}{xyz} + \frac{xy}{z} + z + 3 \\
    \hline
    1.11 & B_{1} & \text{A hypersurface of degree 6 in } \mathbb{P}(1,1,1,2,3) & \frac{(x+y+1)^{6}}{xy^{2}z} + z \\
    \hline
    1.12 & B_{2} & \text{A hypersurface of degree 4 in } \mathbb{P}(1,1,1,1,2) & \frac{(x+y+1)^{4}}{xyz} + z \\
    \hline
    1.13 & B_{3} & \text{A cubic hypersurface in } \mathbb{P}^{4} & \frac{(x+y+1)^{3}}{xyz} + z \\
    \hline
    1.14 & B_{4} & \text{The intersection of two quadrics in } \mathbb{P}^{5} & \frac{(x+1)^{2}(y+1)^{2}}{xyz} + z \\
    \hline
    1.15 & B_{5} & \text{A linear section of }Gr(2, 5)\text{ of
codimension 3} & x+y+z + \frac{1}{x} + \frac{1}{y} + \frac{1}{z} + \frac{1}{xyz} \\
    \hline
    1.16 & Q & \text{A quadric hypersurface in } \mathbb{P}^4 & \frac{(x+1)^2}{xyz}+y+z \\
    \hline
    1.17 & \mathbb{P}^3 & \text{The projective space} & x+y+z+\frac{1}{xyz} \\
    \hline
\end{longtable}
\end{center}
\end{statement}

\begin{example}[Some mutations]
    Consider the Landau-Ginzburg model of $Q$ given by 
    $$
    f = \frac{(x+1)^2}{xyz}+y+z.
    $$
    Consider the mutation given by weight $w=(0,1,1)$ and $a=x+1$. Then we get
    $$
    f' = \frac{1}{xyz} + (y+z)(x+1) = xy+xz+y+z+\frac{1}{xyz}.
    $$
    After a change of variable $x'=\frac{1}{xyz} ,y'=xy, z'=xz$, we get
    $$
    f' = x'+y'+z'+\frac{1}{x'y'}+\frac{1}{x'z'}.
    $$
\end{example}

\subsection{Picard number 2}

    Now we study smooth Fano threefolds of Picard number 2. The elementary syzygies associated to central models of rank 2 are Sarkisov links. In this subsection, we compute them in multiple ways including the LG models method, while some of them can be directly derived from \cite{Matsuki1995}. When we use toric LG models to compute, we put the parametrized toric LG models at the beginning of the computation, otherwise we put the toric LG models at the end.

\begin{statement}[No. 2.1]
    In this case $X$ is a blow-up of $B_{1}$ in the intersection of two divisors from $|-\frac{1}{2}K_{B_{1}}|$. A toric LG model of $X$ is given by
    $$
    f = \frac{(x+y+1)^{6}(z+1)}{xy^{2}z} + z.
    $$
    A parametrized LG model of $X$ is given by
    $$
    \widetilde{f} = \frac{(x+y+1)^{6}(a_{2}z+1)}{xy^{2}z} + a_{1}z.
    $$
    Hence the associated Sarkisov link is
    $$
    \begin{tikzcd}
     B_{1} \arrow[d] & X \arrow[l] \arrow[equal,r] \arrow[dd] & X \arrow[d] \\
     pt \arrow[equal,dr] & & \mathbb{P}^1 \arrow[dl] \\
     & pt &
   \end{tikzcd}    
    $$
    The general fibre $F$ of $X \rightarrow \mathbb{P}^1$ can be computed by setting $a_{1}=0$. The resulting toric LG model is
    $$
    f_{F} = \frac{(x+y+1)^6}{xy^{2}}.
    $$
    This corresponds to del Pezzo surfaces of degree 1.
\end{statement}

\begin{statement}[No. 2.2]
    In this case $X$ is a double cover of $\mathbb{P}^1 \times \mathbb{P}^2$ branched over a $(2,4)$-divisor. The associated Sarkisov link is
    $$
    \begin{tikzcd}
     X \arrow[d] & X \arrow[equal,l] \arrow[equal,r] \arrow[dd] & X \arrow[d] \\
     \mathbb{P}^2 \arrow[dr] & & \mathbb{P}^1 \arrow[dl] \\
     & pt &
   \end{tikzcd}
    $$
    where a general fibre of $X \rightarrow \mathbb{P}^1$ is a del Pezzo surface of degree 2. A toric LG model of $X$ is given by
    $$
    f = \frac{(x+y+z+1)^2}{x} + \frac{(x+y+z+1)^4}{yz}.
    $$
\end{statement}

\begin{statement}[No. 2.3]
    In this case $X$ is a blow-up of $B_{2}$ in the intersection of two divisors from $|-\frac{1}{2}K_{B_{2}}|$. A toric LG model of $X$ is given by
    $$
    f = \frac{(x+y+1)^{4}(z+1)}{xyz}+z.
    $$
    A parametrized LG model of $X$ is given by
    $$
    \widetilde{f} = \frac{(x+y+1)^{4}(a_{2}z+1)}{xyz} + a_{1}z.
    $$
    Hence the associated Sarkisov link is
    $$
    \begin{tikzcd}
     B_{2} \arrow[d] & X \arrow[l] \arrow[equal,r] \arrow[dd] & X \arrow[d] \\
     pt \arrow[equal,dr] & & \mathbb{P}^1 \arrow[dl] \\
     & pt &
    \end{tikzcd}
    $$
    The general fibre $F$ of $X \rightarrow \mathbb{P}^1$ can be computed by setting $a_{1}=0$. The resulting toric LG model is
    $$
    f_{F} = \frac{(x+y+1)^4}{xy}.
    $$
    This corresponds to del Pezzo surfaces of degree 2.
\end{statement}

\begin{statement}[No. 2.4]
    In this case $X$ is a blow-up of $\mathbb{P}^3$ in the intersection of two cubics. By the construction, there is a natural pencil of cubics on $X$. Hence the associated Sarkisov link is 
    $$
    \begin{tikzcd}
     \mathbb{P}^3 \arrow[d] & X \arrow[l] \arrow[equal,r] \arrow[dd] & X \arrow[d] \\
     pt \arrow[equal,dr] & & \mathbb{P}^1 \arrow[dl] \\
     & pt &
   \end{tikzcd}
    $$
    where the Mori fibre space on the right-hand side is a del Pezzo fibration of degree 3. A toric LG model of $X$ is given by
    $$
    f = \frac{x^{2}}{z}+ \frac{xy}{z} + \frac{x^2}{y} + 4x + 2y + \frac{2xz}{y} + 4z + \frac{yz}{x} + \frac{z^{2}}{y} + \frac{z^{2}}{x}+ \frac{3x}{z} +
 \frac{2y}{z} + \frac{2x}{y} + \frac{2y}{x} + \frac{2z}{y} + \frac{3z}{x} + \frac{3}{z} + \frac{y}{xz}  + \frac{1}{y} + \frac{3}{x} + \frac{1}{xz}.
    $$
\end{statement}

\begin{statement}[No. 2.5]
    In this case $X$ is a blow-up of $B_{3}$ in a plane cubic. A toric LG model of $X$ is given by
    \begin{align*}
    f & = x+y+z+\frac{x^2}{yz} + \frac{3x}{z} + \frac{3y}{z} + \frac{y^2}{xz} + \frac{3x}{y} + \frac{3y}{x} + \frac{3z}{y}  + \frac{3z}{x} + \frac{z^2}{xy} + \frac{x}{yz} + \frac{2}{z} + \frac{y}{xz} + \frac{2}{y} + \frac{2}{x} + \frac{z}{xy} \\
    & = x+y+z + \frac{(x+y+z)^{3}}{xyz} + \frac{(x+y+z)^{2}}{xyz}.
    \end{align*}
    Taking a change of coordinate $x \mapsto xz$ and $y \mapsto yz$, we get
    $$
    f = z(x+y+1) + \frac{(x+y+1)^3}{xy} + \frac{(x+y+1)^{2}}{xyz}.
    $$
    Taking a mutation given by weight $w = (0,0,-1)$ and factor $a = x+y+z$, we get 
    $$
    f' = \frac{(x+y+1)^{3}(z+1)}{xyz} + z.
    $$
    A parametrized LG model of $X$ is given by
    $$
    \widetilde{f'} = \frac{(x+y+1)^{3}(a_{2}z+1)}{xyz} + a_{1}z.
    $$
    Hence the associated Sarkisov link is
    $$
    \begin{tikzcd}
     B_{3} \arrow[d] & X \arrow[l] \arrow[equal,r] \arrow[dd] & X \arrow[d] \\
     pt \arrow[equal,dr] & & \mathbb{P}^1 \arrow[dl] \\
     & pt &
    \end{tikzcd}
    $$
    The general fibre $F$ of $X \rightarrow \mathbb{P}^1$ can be computed by setting $a_{1}=0$. The resulting toric LG model is
    $$
    f_{F} = \frac{(x+y+1)^3}{xy}.
    $$
    This corresponds to del Pezzo surfaces of degree 3.
\end{statement}

\begin{statement}[No. 2.6]
    In this case, the family has two different types of members.
\begin{enumerate}
    \item A general member $X$ in this family is a divisor on $\mathbb{P}^2 \times \mathbb{P}^2$ of bidegree $(2,2)$. The associated Sarkisov link is
    $$
    \begin{tikzcd}
     X \arrow[d] & X \arrow[equal,l] \arrow[equal,r] \arrow[dd] & X \arrow[d] \\
     \mathbb{P}^2 \arrow[dr] & & \mathbb{P}^2 \arrow[dl] \\
     & pt &
   \end{tikzcd}
    $$
    This is a Sarkisov link of type IV. 
    
    \item A special member $X$ in this family is a double cover of a divisor on $\mathbb{P}^2 \times \mathbb{P}^2$ of bidegree $(1,1)$, branched over an anti-canonical divisor. The associated Sarkisov link is
    $$
    \begin{tikzcd}
     X \arrow[d] & X \arrow[equal,l] \arrow[equal,r] \arrow[dd] & X \arrow[d] \\
     \mathbb{P}^2 \arrow[dr] & & \mathbb{P}^2 \arrow[dl] \\
     & pt &
   \end{tikzcd}
    $$
    This is a Sarkisov link of type IV. 
\end{enumerate}
    A toric LG model of $X$ is given by
    $$
    f = x + y + \frac{xz}{y} + 2z + \frac{yz}{x}+ \frac{z^{2}}{y} + \frac{z^{2}}{x} + \frac{x}{z} + \frac{y}{z} + \frac{2x}{y} + \frac{2y}{x} +
\frac{3z}{y} + \frac{3z}{x} + \frac{x}{yz}+ \frac{2}{z} + \frac{y}{xz} + \frac{3}{y} + \frac{3}{x} + \frac{1}{yz} + \frac{1}{xz}.
    $$
\end{statement}

\begin{statement}[No. 2.7]
    In this case $X$ is the blow-up of $Q$ in the intersection of two divisors from $|\mathcal{O}_{Q}(2)|$. A toric Landau-Ginzburg model of $X$ is given by
    $$
    f = x + \frac{x}{y} + z + \frac{z}{y} + \frac{xy}{z} + \frac{2x}{z} + \frac{x}{yz} + 2y + \frac{2}{y} + \frac{yz}{x} + \frac{2z}{x} + \frac{z}{xy} + \frac{2y}{z} + \frac{2}{z} + \frac{2y}{x}+ \frac{2}{x} + \frac{y}{xz}.
    $$
    We take a change of coordinate $y \mapsto \frac{1}{y}$ and a mutation given by weight $w = (-1,0,-1)$ and factor $a = y+1$. The new toric LG model is 
    $$
    f' = \frac{(y+1)^{2}(x+z+1)^{2}}{xyz} + x + z.
    $$
    A parametrized LG model is 
    $$
    \widetilde{f'} = \frac{(y+1)^{2}(a_{2}x+a_{2}z+1)^{2}}{xyz} + a_{1}x + a_{1}z.
    $$
    Hence the associated Sarkisov link is
    $$
    \begin{tikzcd}
     Q \arrow[d] & X \arrow[l] \arrow[equal,r] \arrow[dd] & X \arrow[d] \\
     pt \arrow[equal,dr] & & \mathbb{P}^1 \arrow[dl] \\
     & pt &
    \end{tikzcd}
    $$
    The general fibre $F$ of $X \rightarrow \mathbb{P}^1$ can be computed by setting $a_{1}=0$. The resulting toric LG model is
    $$
    f_{F} = \frac{(y+1)^{2}(x+z)^{2}}{xyz}.
    $$
    This corresponds to del Pezzo surfaces of degree 4.
\end{statement}

\begin{statement}[No. 2.8]
    In this case, a toric LG model of $X$ is given by
    \begin{align*}
        f & = \frac{xy}{z} + 2x + 2y + \frac{xz}{y} + 2z + 
        \frac{yz}{x} + \frac{2x}{z} + \frac{2x}{y} +
        \frac{x}{yz} + \frac{2}{z} + \frac{2}{y} + \frac{2}{x} + \frac{2}{yz} + \frac{1}{xyz} \\
        & = \frac{(xy+yz+xz+x+1)^2}{xyz} - 2.
    \end{align*}
    A parametrized toric LG model of $X$ is given by
    $$
    \widetilde{f} = \frac{(xy+yz+xz+a_{2}x+a_{1})^2}{xyz}.
    $$
    This family has two different types of members.
    \begin{enumerate}
    \item A general member $X$ in this family is the double cover of $\mathrm{Bl}_{p}\mathbb{P}^3$ branched over a divisor from $|-K_{\mathrm{Bl}_{p}\mathbb{P}^3}|$ such that the intersection with the exceptional divisor is smooth. The associated Sarkisov link is 
    $$
    \begin{tikzcd}
       Y_{1} \arrow[d] & X \arrow[dd] \arrow[l] \arrow[equal,r] & X \arrow[d] \\\
       pt \arrow[equal,dr] & & \mathbb{P}^2 \arrow[dl] \\
        & pt &
    \end{tikzcd}
    $$
    where $X \rightarrow Y_{1}$ is a contraction of type E3.
    \item A special member $X$ in this family is the double cover of $\mathrm{Bl}_{p}\mathbb{P}^3$ branched over a divisor from $|-K_{\mathrm{Bl}_{p}\mathbb{P}^3}|$ such that the intersection with the exceptional divisor is singular but reduced. The associated Sarkisov link is 
    $$
    \begin{tikzcd}
       Y_{2} \arrow[d] & X \arrow[dd] \arrow[l] \arrow[equal,r] & X \arrow[d] \\\
       pt \arrow[equal,dr] & & \mathbb{P}^2 \arrow[dl] \\
        & pt &
    \end{tikzcd}
    $$
    where $X \rightarrow Y_{1}$ is a contraction of type E4.
    \end{enumerate}
    These links are of type I/III. A toric LG model of $Y_{1}$ and $Y_{2}$ is
    $$
    f_{Y_{1}} = f_{Y_{2}} = \frac{(xy+yz+xz+1)^2}{xyz} = f_{B_{2}}.
    $$
    Hence $Y_{1}$ and $Y_{2}$ deform to $B_{2}$.
\end{statement}

\begin{statement}[No. 2.9]
    In this case, $X$ is a complete intersection of a divisor of bidegree $(1,1)$ and a divisor of bidegree $(2,1)$ in $\mathbb{P}^3 \times \mathbb{P}^2$. Hence the associated Sarkisov link is
    $$
    \begin{tikzcd}
     \mathbb{P}^3 \arrow[d] & X \arrow[l] \arrow[equal,r] \arrow[dd] & X \arrow[d] \\
     pt \arrow[equal,dr] & & \mathbb{P}^2 \arrow[dl] \\
     & pt &
   \end{tikzcd}    
    $$
    This is a link of type I/III. A toric LG model of $X$ is given by
    $$
    f = x + y + z + \frac{x}{z} + \frac{y}{z} + \frac{x}{y} + \frac{y}{x} + \frac{2z}{y} + \frac{2z}{x} + \frac{z^{2}}{xy} + \frac{x}{yz} + \frac{2}{z} + \frac{y}{xz} + \frac{2}{y} + \frac{2}{x} + \frac{z}{xy}.
    $$
\end{statement}

\begin{statement}[No. 2.10]
    In this case, $X$ is the blow-up of $B_{4}$ in an intersection of two hyperplane sections. A toric LG model of $X$ is given by
    $$
    f = x + y + 2z + \frac{yz}{x} + \frac{z^{2}}{x} + 
    \frac{x}{z} + \frac{x}{y} + \frac{y}{x} + \frac{z}{y} +
    \frac{3z}{x} + \frac{x}{yz} + \frac{2}{z} + \frac{2}{y} + \frac{3}{x} + \frac{1}{yz} + \frac{1}{xz}.
    $$
    Adding a constant term and taking a change of coordinates $x \mapsto \frac{1}{x}$, we can rewrite the Laurent polynomial as
    $$
    f = y + \frac{x(z+1)^{3}}{z} + \frac{2(z+1)^2}{z} + \frac{z+1}{xz} + \frac{(z+1)}{xyz} + xy(z+1).
    $$
    Make a first mutation given by weight $w_{1} = (-1,0,0)$ and factor $a_{1} = z+1$, we get the new toric LG model
    $$
    f' = y(x+1) + \frac{(x+1)^{2}(z+1)^{2}}{xz} + \frac{(z+1)^2}{xyz}.
    $$
    Make a second mutation given by weight $w_{2} = (0,-1,-1)$ and factor $a_{2} = x+1$, we get the new toric LG model
    $$
    \frac{(x+1)^{2}(z+1)^{2}}{xyz} + y + \frac{(x+1)(x+z+1)^{2}}{xz}.
    $$
    A parametrized toric LG model of $X$ is given by
    $$
    \frac{(x+1)^{2}(z+1)^{2}}{xyz} + a_{1}y + \frac{a_{2}(x+1)(x+z+1)^{2}}{xz}.
    $$
    Hence the associated Sarkisov link is
    $$
    \begin{tikzcd}
     B_{4} \arrow[d] & X \arrow[l] \arrow[equal,r] \arrow[dd] & X \arrow[d] \\
     pt \arrow[equal,dr] & & \mathbb{P}^1 \arrow[dl] \\
     & pt &
    \end{tikzcd}
    $$
    The general fibre $F$ of $X \rightarrow \mathbb{P}^1$ can be computed by setting $a_{1}=0$. The resulting toric LG model is
    $$
    f_{F} = \frac{(x+1)(x+z+1)^{2}}{xz}.
    $$
    This corresponds to del Pezzo surfaces of degree 4.
\end{statement}

\begin{statement}[No. 2.11]
    In this case $X$ is the blow-up of $B_{3}$ in a line. A toric LG model of $X$ is given by
    $$
    f = y + z + \frac{yz}{x} + \frac{2z^{2}}{x} + 
    \frac{z^{3}}{xy} + \frac{x}{z} + \frac{y}{z} + \frac{2y}{x} + \frac{2z}{y} + \frac{2z}{x} + \frac{x}{yz} + \frac{2}{z} + \frac{y}{xz}.
    $$
    Taking a change of coordinates $x \mapsto xz^{2}, y\mapsto yz^{2}$, we can rewrite the polynomial $f$ as
    $$
    f = \frac{(x+y(1+z)+1)^2}{xyz} + z(x+y(1+z)+1).
    $$
    We make two mutations, the first mutation is given by the weight $w_{1} = (0,-1,0)$ and factor $a_{1} = z+1$. The new toric LG model is
    $$
    f' = \frac{(x+y+1)^{2}(z+1)}{xyz} + z(x+y+1).
    $$
    The second mutation is given by the weight $w_{2} = (0,0,-1)$ and factor $a_{2} = x+y+1$. The new toric LG model is
    $$
    f'' = \frac{(x+y+1)^{2}(x+y+z+1)}{xyz} + z.
    $$
    A parametrized LG model of $X$ is given by
    $$
    \widetilde{f''} = \frac{(x+y+1)^{2}(x+y+a_{2}z+1)}{xyz} + a_{1}z.
    $$
    Hence the associated Sarkisov link of $X$ is
    $$
    \begin{tikzcd}
     B_{3} \arrow[d] & X \arrow[l] \arrow[equal,r] \arrow[dd] & X \arrow[d] \\
     pt \arrow[equal,dr] & & \mathbb{P}^{2} \arrow[dl] \\
     & pt &
    \end{tikzcd}
    $$
    This is a link of type I/III.
\end{statement}

\begin{statement}[No. 2.12]
    In this case $X$ is an intersection of 3 divisors of type $(1,1)$ in $\mathbb{P}^3 \times \mathbb{P}^3$. The associated Sarkisov link is
    $$
    \begin{tikzcd}
     \mathbb{P}^3 \arrow[d] & X \arrow[l] \arrow[r] \arrow[dd] & \mathbb{P}^3 \arrow[d] \\
     pt \arrow[equal,dr] & & pt \arrow[equal,dl] \\
     & pt &
   \end{tikzcd}    
    $$
    This is a link of type II. This link is known as the \emph{cubo-cubic transformation} of $\mathbb{P}^3$. A toric LG model of $X$ is given by 
    $$
    f = \frac{xy}{z} + x + y + z + \frac{2x}{z} + \frac{2y}{z} + \frac{x}{yz} + \frac{2}{z} + \frac{y}{xz} + \frac{2}{y} + \frac{2}{x} + \frac{z}{xy}.
    $$
\end{statement}

\begin{statement}[No. 2.13]
    In this case $X$ is a blow-up of $Q$ in a curve of degree 6 and genus 2. A toric LG model of $X$ is given by
    $$
    f = x + y + z + \frac{yz}{x} + \frac{x}{y} + \frac{2y}{x} + \frac{z}{y} + \frac{x}{yz} + \frac{2}{z} + \frac{y}{xz} + \frac{2}{y} + \frac{1}{yz}.
    $$
    Taking a change of coordinates $x \mapsto xy, z \mapsto \frac{1}{z}$, we can write the Laurent polynomial as
    $$
    f = x + y + \frac{1}{z} + xy + \frac{(xz+z+1)^{2}}{xz} + \frac{(z+1)^2}{yz}.
    $$
    Here the constant term is ignored. A parametrized toric LG model of $X$ is given by
    $$
    \widetilde{f} = a_{1}x + a_{1}^{2}y + \frac{a_{2}}{z} + a_{1}a_{2}xy + \frac{a_{1}a_{2}x^{2}z^{2} + (a_{1}^{2}+a_{2}^{3})xz^{2} + a_{1}a_{2}^{2}z^{2} + 2a_{1}^{2}a_{2}z + a_{1}}{xz} + \frac{(a_{2}z+1)^2}{yz}.
    $$
    Hence the associated Sarkisov link is
    $$
    \begin{tikzcd}
        Q \arrow[d] & X \arrow[dd] \arrow[l] \arrow[equal,r] & X \arrow[d] \\
        pt \arrow[equal,dr] & & \mathbb{P}^2 \arrow[dl] \\
        & pt &
    \end{tikzcd}
    $$
    This is a link of type I/III.
\end{statement}

\begin{statement}[No. 2.14]
    In this case $X$ is the blow-up of $B_{5}$ in the intersection of two hyperplane sections. There is a natural pencil on $X$ of hyperplane sections of $B_{5}$. Hence the associated Sarkisov link is
    $$
    \begin{tikzcd}
        B_{5} \arrow[d] & X \arrow[dd] \arrow[l] \arrow[equal,r] & X \arrow[d] \\
        pt \arrow[equal,dr] & & \mathbb{P}^1 \arrow[dl] \\
        & pt &
    \end{tikzcd}
    $$
    where a general fibre of $X \rightarrow \mathbb{P}^1$ is a smooth del Pezzo surface of degree 5. This is a link of type I/III. A toric LG model of $X$ is given by
    $$
    \frac{xy}{z} + x + z + \frac{2y}{z} + \frac{z}{x} + \frac{z^{2}}{xy} + \frac{2}{z} + \frac{y}{xz} + \frac{3}{x} + \frac{3z}{xy} + \frac{2}{xz} + \frac{3}{xy} + \frac{1}{xyz}.
    $$
    Taking a change of coordinates $y \mapsto xyz$, we can rewrite the Laurent polynomial as
    $$
    f = x + y(x+1)^{2} + z + \frac{2}{z} + \frac{(z+1)(z+2)}{xz} + \frac{(z+1)^{3}}{x^{2}yz^{2}}.
    $$
    Make two mutations given by $w_{1} = (0,-2,0)$, $a_{1} = x+1$ and $w_{2} = (1,1,0)$, $a_{2} = z + 1$. We obtain a new toric LG model
    $$
    f' = x + y + z + xz + yz + \frac{1}{x} + \frac{2}{z} + \frac{2}{xz} + \frac{(xz + x + 1)^{2}}{x^{2}yz^{2}}.
    $$
    Taking a change of coordinates $y \mapsto \frac{xy}{z}, z \mapsto \frac{z}{x}$, we can rewrite the Laurent polynomial as
    $$
    f' = x + y + z + \frac{xy}{z} + \frac{z}{x} + \frac{2x}{z} + \frac{1}{x} + \frac{2}{z} + \frac{(x+z+1)^2}{xyz}
    $$
    A parametrized toric LG model of $X$ is given by
    \begin{multline*}
    \widetilde{f'} = x + y + z + \frac{a_{2}^{2}xy}{z} + \frac{a_{2}z}{x} + \frac{2a_{2}x}{z} + \frac{a_{1}^{2}}{x} + \frac{a_{1}^{2}(1 + a_{2}^2)}{z} \\
    + \frac{a_{2}^{2}x^{2} + a_{2}^{2}z^{2} +  (a_{2}^{4}+1)xz + 2a_{1}^{2}a_{2}z + 2a_{1}^{2}a_{2}x + a_{1}^{4}}{xyz}.    
    \end{multline*}
\end{statement}

\begin{statement}[No. 2.15]\label{example: Mori-Mukai 2.15}
    In this case $X$ is a blow-up of $\mathbb{P}^3$ in the intersection of a quadric $S_{2}$ and a cubic $S_{3}$. A toric LG model is given by
    $$
    f = x+y+z+\frac{x}{z} + \frac{y}{z} + \frac{x}{yz} + \frac{y}{xz} +\frac{2}{z} + \frac{2}{y} + \frac{2}{x} + \frac{z}{xy}.
    $$
    Taking a change of coordinates $x \mapsto xz, y \mapsto yz$, we can rewrite the Laurent polynomial as
    \begin{align*}
    f &= x+y+z+xz+yz+\frac{x}{yz} + \frac{y}{xz}+\frac{2}{z}+\frac{2}{xz}+\frac{2}{yz}+\frac{1}{xyz} \\ 
     &= x + y + z(x+y+1) + \frac{(x+y+1)^{2}}{xyz}.
    \end{align*}
    Take a mutation by weight $w = (0,0,-1)$ and factor $a = x + y + 1$. We obtain a new toric LG model
    $$
    f' = x + y + z + \frac{(x+y+1)^{3}}{xyz}.
    $$
    \begin{enumerate}
    \item  The quadric $S_{2}$ is smooth. In this case, the associated Sarkisov link of $X$ is
    $$
    \begin{tikzcd}
     \mathbb{P}^3 \arrow[d] & X \arrow[l] \arrow[r] \arrow[dd] & Y_{1} \arrow[d] \\
     pt \arrow[equal,dr] & & pt \arrow[equal,dl] \\
     & pt &
   \end{tikzcd}
    $$
    where $X \rightarrow Y_{1}$ is a contraction of type E3. This is a link of type II.
    \item The quadric $S_{2}$ is singular, i.e. a quadratic cone. In this case, the associated Sarkisov link of $X$ is
    $$
    \begin{tikzcd}
     \mathbb{P}^3 \arrow[d] & X \arrow[l] \arrow[r] \arrow[dd] & Y_{2} \arrow[d] \\
     pt \arrow[equal,dr] & & pt \arrow[equal,dl] \\
     & pt &
   \end{tikzcd}
    $$
    where $X \rightarrow Y_{2}$ is a contraction of type E4 to a terminal point. This is a link of type II.
    \end{enumerate}
    A toric LG model of $Y_{1}$ and $Y_{2}$ is given by
    $$
    f_{Y_{1}} = f_{Y_{2}} = z + \frac{(x+y+1)^{3}}{xyz}.
    $$
    Hence $Y_{1}$ and $Y_{2}$ deform to the smooth cubic threefold $B_{3}$.
\end{statement}

\begin{statement}[No. 2.16]
    In this case $X$ is a blow-up of $B_{4}$ in a conic. A toric LG model of $X$ is given by
    $$
    f = x + z + \frac{x}{z} + \frac{y}{z} + \frac{x}{y} + \frac{y}{x} + \frac{z}{y} + \frac{z}{x} + \frac{x}{yz} + \frac{1}{z} + \frac{2}{y} + \frac{1}{x} + \frac{z}{xy}.
    $$
    Adding a constant term, we can rewrite the Laurent polynomial as
    $$
    f = \frac{(y+1)(x+z)}{y} + \frac{(y+1)(x+z)^2}{xyz} + \frac{(y+1)(x+z)}{xz}.
    $$
    Taking a change of coordinates $x \mapsto xyz$, $z \mapsto yz$, we can rewrite the Laurent polynomial as
    $$
    f = z(x+1)(y+1) + \frac{(x+1)^{2}(y+1)}{xy} + \frac{(x+1)(y+1)}{xyz}.
    $$
    Make a mutation given by weight $w = (0,0,-1)$ and factor $a = (x+1)(y+1)$, we obtain a new toric LG model
    $$
    f' = \frac{(x+1)^{2}(y+1)^{2}}{xyz} + z + \frac{(x+1)^{2}(y+1)}{xy}.
    $$
    A parametrized toric LG model of $X$ is
    $$
    \widetilde{f'} = \frac{(x+1)^{2}(y+1)^{2}}{xyz} + a_{1}z + \frac{a_{2}(x+1)^{2}(y+1)}{xy}.
    $$
    Hence the associated Sarkisov link is
    $$
    
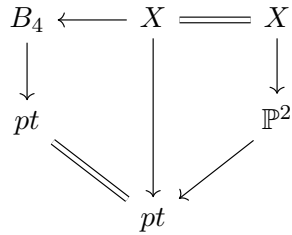
\begin{tikzcd}
     B_{4} \arrow[d] & X \arrow[l] \arrow[equal,r] \arrow[dd] & X \arrow[d] \\
     pt \arrow[equal,dr] & & \mathbb{P}^2 \arrow[dl] \\
     & pt &
    \end{tikzcd}
    $$
    This is a link of type I/III.
\end{statement}

\begin{statement}[No. 2.17]
    In this case, $X$ is a blow-up of $Q$ in an elliptic curve of degree 5. The associated Sarkisov link is
    $$
    
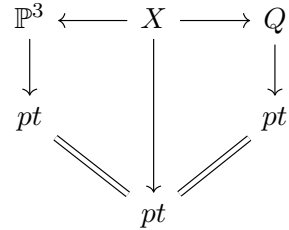
\begin{tikzcd}
     \mathbb{P}^3 \arrow[d] & X \arrow[l] \arrow[r] \arrow[dd] & Q \arrow[d] \\
     pt \arrow[equal,dr] & & pt \arrow[equal,dl] \\
     & pt &
   \end{tikzcd}    
    $$
    This is a link of type II. A toric LG model of $X$ is given by
    $$
    f = x + y + z + \frac{x}{y} + \frac{y}{x} + \frac{z}{y} + \frac{z}{x} + \frac{1}{z} + \frac{2}{y} + \frac{1}{x} + \frac{z}{xy} + \frac{1}{xz} + \frac{1}{xy}.
    $$
\end{statement}

\begin{statement}[No. 2.18]
    In this case $X$ is a double cover of $\mathbb{P}^1 \times \mathbb{P}^2$ ramified in a divisor of bidegree $(2,2)$. The associated Sarkisov link is
    $$
    
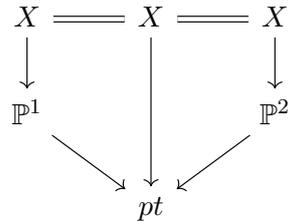
\begin{tikzcd}
     X \arrow[d] & X \arrow[equal,l] \arrow[equal,r] \arrow[dd] & X \arrow[d] \\
     \mathbb{P}^1 \arrow[dr] & & \mathbb{P}^2 \arrow[dl] \\
     & pt &
   \end{tikzcd}
    $$
    where a general fibre of $X \rightarrow \mathbb{P}^1$ is isomorphic to $\mathbb{P}^1 \times 
    \mathbb{P}^1$. This is a link of type IV. A toric LG model of $X$ is given by
    $$
    f = x + y + z + \frac{y}{x} + \frac{z}{x} + \frac{x}{yz} + \frac{1}{z} + \frac{1}{y}+ \frac{1}{x} + 
    \frac{2}{yz} + \frac{1}{xz} + \frac{1}{xy} + \frac{1}{xyz}.
    $$
\end{statement}

\begin{statement}[No. 2.19]
    In this case $X$ is the blow-up of $\mathbb{P}^3$ in a curve of degree 5 and genus 2. Alternatively, $X$ is the blow-up of $B_{4}$ in a line. Hence the associated Sarkisov link is
    $$
    \begin{tikzcd}
     \mathbb{P}^3 \arrow[d] & X \arrow[l] \arrow[r] \arrow[dd] & B_{4} \arrow[d] \\
     pt \arrow[equal,dr] & & pt \arrow[equal,dl] \\
     & pt &
   \end{tikzcd}    
    $$
    This is a link of type II. A toric LG model of $X$ is given by
    $$
    f = x + y + z + \frac{yz}{x} + \frac{x}{z} + \frac{y}{z} + \frac{x}{y} + \frac{y}{x} + \frac{z}{y} + \frac{x}{yz} + \frac{1}{y}.
    $$
\end{statement}

\begin{statement}[No. 2.20]
    In this case $X$ is a blow-up of $B_{5}$ in a twisted cubic. A toric LG model of $X$ is given by
    $$
    f = x+y+z+\frac{1}{x}+\frac{1}{y}+\frac{1}{z}+ \frac{y}{z} + \frac{y}{x} + \frac{y}{xz} + \frac{1}{xz} + \frac{xz}{y}
    $$
    After a change of coordinates, we can write it as
    $$
    f = x+y+z+\frac{1}{x}+\frac{1}{y}+\frac{1}{z} + xy + yz + xz + xyz +\frac{1}{xyz}.
    $$
    A parametrized toric LG model of $X$ is given by
    $$
    \widetilde{f} = a_{1}x+a_{1}y+z+\frac{1}{x}+\frac{1}{y}+\frac{a_{1}}{z} + a_{1}a_{2}xy + a_{2}yz + a_{2}xz + a_{2}^{2}xyz +\frac{1}{xyz}.
    $$
    Hence the associated Sarkisov link is
    $$
    \begin{tikzcd}
     B_{5} \arrow[d] & X \arrow[l] \arrow[r,equal] \arrow[dd] & X \arrow[d] \\
     pt \arrow[equal,dr] & & \mathbb{P}^2 \arrow[dl] \\
     & pt &
   \end{tikzcd}    
    $$
    This is a link of type I/III.
\end{statement}

\begin{statement}[No. 2.21]
    In this case, $X$ is a blow-up of $Q$ in a twisted quartic. A toric LG model of $X$ is given by
    $$
    f = x+y+z+\frac{x}{z} + \frac{y}{z} + \frac{x}{y} + \frac{y}{x} + \frac{z}{y} + \frac{1}{x} + \frac{1}{z}.
    $$
    After a change of coordinates $x \mapsto \frac{1}{x}, z \mapsto \frac{1}{z}$, we can rewrite it as
    $$
    f = x+y+z  + \frac{1}{xy} + \frac{1}{yz} + xy + yz + \frac{1}{x} + \frac{1}{z} + \frac{z}{x}.
    $$
    A parametrized toric LG model of $X$ is given by
    $$
    \widetilde{f} = x + a_{1}^{2}y + a_{1}z + \frac{a_{1}}{xy} + \frac{1}{yz} + a_{2}xy + a_{1}a_{2}yz + \frac{a_{1}a_{2}}{x} + \frac{a_{2}}{z} + \frac{a_{2}^{2}z}{x}.
    $$
    Hence the associated Sarkisov link is
    $$
    \begin{tikzcd}
        Q \arrow[d] & X \arrow[dd] \arrow[l] \arrow[r] & Q \arrow[d] \\
        pt \arrow[equal,dr] & & pt \arrow[equal,dl] \\
        & pt &
    \end{tikzcd}
    $$
    This is a link of type II.
\end{statement}

\begin{statement}[No. 2.22]
    In this case, $X$ is a blow-up of $B_{5}$ in a conic. Alternatively, $X$ is a blow-up of $\mathbb{P}^3$ in a curve of degree 4 and genus 0. Hence the associated Sarkisov link is
    $$
    
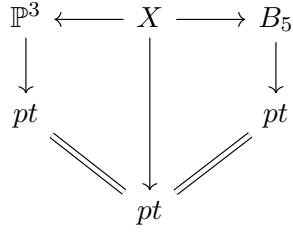
\begin{tikzcd}
     \mathbb{P}^3 \arrow[d] & X \arrow[l] \arrow[r] \arrow[dd] & B_{5} \arrow[d] \\
     pt \arrow[equal,dr] & & pt \arrow[equal,dl] \\
     & pt &
   \end{tikzcd}    
    $$
    This is a link of type II. A toric LG model of $X$ is given by
    $$
    f = x + y + \frac{xz}{y} + z + \frac{y}{x} + \frac{1}{z} + \frac{1}{y} + \frac{1}{x} + \frac{1}{xz}.
    $$
\end{statement}

\begin{statement}[No. 2.23]
    In this case $X$ is the blow-up of $Q$ in a space quartic curve $C$. A toric LG model of $X$ is given by
    $$
    f = x+y+z+\frac{x}{z}+\frac{y}{z}+\frac{z}{y}+\frac{z}{x}+\frac{1}{x}+\frac{1}{y}.
    $$
    After a change of coordinates we can rewrite the Laurent polynomial as
    $$
    f = x+y+z+\frac{1}{xz}+\frac{1}{yz}+xz+yz+\frac{1}{x}+\frac{1}{y}.
    $$

    Let $H$ be the hyperplane in $\mathbb{P}^4$ containing $C$. There are two cases:
    \begin{enumerate}
        \item The intersection $H \cap Q$ is a smooth quadratic surface. The associated Sarkisov link is
        $$
        
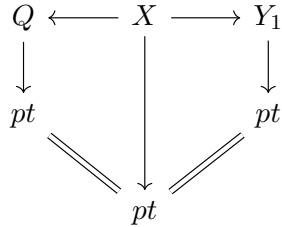
\begin{tikzcd}
         Q \arrow[d] & X \arrow[l] \arrow[r] \arrow[dd] & Y_{1} \arrow[d] \\
         pt \arrow[equal,dr] & & pt \arrow[equal,dl] \\
         & pt &
       \end{tikzcd}
        $$
        \item The intersection $H \cap Q$ is a singular quadratic surface. The associated Sarkisov link is
        $$
        
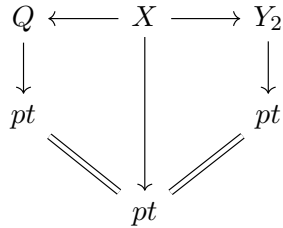
\begin{tikzcd}
         Q \arrow[d] & X \arrow[l] \arrow[r] \arrow[dd] & Y_{2} \arrow[d] \\
         pt \arrow[equal,dr] & & pt \arrow[equal,dl] \\
         & pt &
       \end{tikzcd}
       $$
    \end{enumerate}
    where $Y_{1}$ has an ordinary double point and $Y_{2}$ has a double (cDV)-point. A toric LG model of $Y_{1}$ and $Y_{2}$ is given by
    $$
    f_{Y_{1}} = f_{Y_{2}} = x+y+\frac{1}{xz}+\frac{1}{yz}+xz+yz+\frac{1}{x}+\frac{1}{y}.
    $$
    Taking a change of coordinates $y \mapsto xy$, and making a mutation given by weight $w = (-1,0,0)$ and factor $a = (y+1)(z+1)$, we obtain a new toric LG model
    $$
    f'_{Y_{1}} = f'_{Y_{2}} = x + \frac{(y+1)^{2}(z+1)^{2}}{xyz}.
    $$
    Hence $Y_{1}$ and $Y_{2}$ deform to $B_{4}$.
\end{statement}

\begin{statement}[No. 2.24]
    In this case $X$ is a divisor on $\mathbb{P}^2 \times \mathbb{P}^2$ of bidegree $(1,2)$. The associated Sarkisov link is
    $$
    
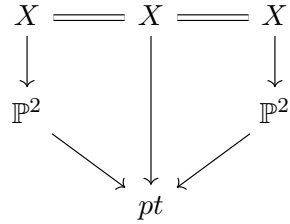
\begin{tikzcd}
     X \arrow[d] & X \arrow[equal,l] \arrow[equal,r] \arrow[dd] & X \arrow[d] \\
     \mathbb{P}^2 \arrow[dr] & & \mathbb{P}^2 \arrow[dl] \\
     & pt &
   \end{tikzcd}
    $$
    This is a link of type IV. A toric LG model of $X$ is given by
    $$
    f = x + y + z + \frac{xy}{z} + \frac{x}{z} + \frac{y}{x} + \frac{z}{y} + \frac{1}{x} + \frac{1}{y}.
    $$
\end{statement}

\begin{statement}[No. 2.25]
    In this case $X$ is a blow-up of $\mathbb{P}^3$ in the intersection of two smooth quadrics. There is a natural pencil of smooth quadrics on $X$. Hence the associated Sarkisov link is
    $$
    
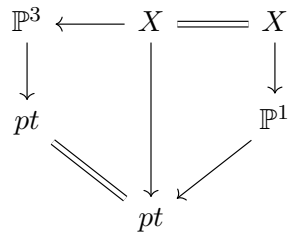
\begin{tikzcd}
     \mathbb{P}^3 \arrow[d] & X \arrow[l] \arrow[equal,r] \arrow[dd] & X \arrow[d] \\
     pt \arrow[equal,dr] & & \mathbb{P}^1 \arrow[dl] \\
     & pt &
   \end{tikzcd}    
    $$
    where a general fibre of $X \rightarrow \mathbb{P}^1$ is isomorphic to $\mathbb{P}^1 \times \mathbb{P}^1$. This is a link of type I/III. A toric LG model of $X$ is given by
    $$
    f = x + y + z + xyz + \frac{1}{x} + \frac{1}{y} + \frac{1}{xz}  + \frac{1}{yz}.
    $$
\end{statement}

\begin{statement}[No. 2.26]
    In this case, $X$ is a blow-up of $B_{5}$ in a line. The associated Sarkisov link is
    $$
    
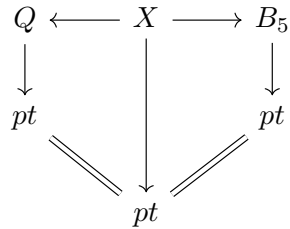
\begin{tikzcd}
     Q \arrow[d] & X \arrow[l] \arrow[r] \arrow[dd] & B_{5} \arrow[d] \\
     pt \arrow[equal,dr] & & pt \arrow[equal,dl] \\
     & pt &
   \end{tikzcd}    
    $$
    This is a link of type II. A toric LG model of $X$ is given by
    $$
    f = x + y + z + \frac{1}{x} + \frac{1}{y} + \frac{1}{z} + \frac{1}{xyz} + xy.
    $$
\end{statement}

\begin{statement}[No. 2.27]
    In this case, $X$ is the blow-up of $\mathbb{P}^3$ in a twisted cubic curve. In this case $X$ is known to admit a structure of $\mathbb{P}^1$-bundle over $\mathbb{P}^2$. Hence the associated Sarkisov link is
    $$
    
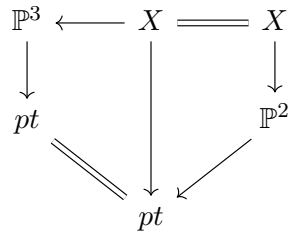
\begin{tikzcd}
     \mathbb{P}^3 \arrow[d] & X \arrow[l] \arrow[equal,r] \arrow[dd] & X \arrow[d] \\
     pt \arrow[equal,dr] & & \mathbb{P}^2 \arrow[dl] \\
     & pt &
   \end{tikzcd}    
    $$
    This is a link of type I/III. A toric LG model of $X$ is given by
    $$
    f = x+y+z+xz+\frac{1}{xz}+\frac{1}{yz}+\frac{1}{xyz}.
    $$
\end{statement}

\begin{statement}[No. 2.28]\label{case 2.28}
    In this case $X$ is the blow-up of $\mathbb{P}^3$ in a plane cubic curve. A toric LG model of $X$ is given by
    $$
    f = \frac{(xyz+1)(xyz^2 + xyz + x + y)}{xyz} - 1 = xyz^{2}+xyz+x+y+z+\frac{1}{yz}+\frac{1}{xz}.
    $$
    Make a change of coordinates $x \mapsto \frac{1}{xz}, y \mapsto \frac{1}{yz}$, we can write the Laurent polynomial as
    $$
    f = x+y+z+\frac{1}{xz}+\frac{1}{yz}+\frac{1}{xy}+\frac{1}{xyz}.
    $$
    A parametrized LG model is given by
    $$
    f = x+y+z+\frac{a_{2}}{xz}+\frac{a_{2}}{yz}+\frac{a_{2}}{xy}+\frac{a_{1}}{xyz}.
    $$
    Hence the associated Sarkisov link is
    $$
    \begin{tikzcd}
     \mathbb{P}^3 \arrow[d] & X \arrow[l] \arrow[r] \arrow[dd] & Y \arrow[d] \\
     pt \arrow[equal,dr] & & pt \arrow[equal,dl] \\
     & pt &
   \end{tikzcd}
    $$
    This is a link of type II. Here $X \rightarrow Y$ is a contraction of type E5 to a variety with singularity of type $\frac{1}{2}(1,1,1)$. A toric LG model of $Y$ is given by
    $$
    f_{Y} = x+y+z+\frac{1}{xz}+\frac{1}{yz}+\frac{1}{xy}.
    $$
\end{statement}

\begin{statement}[No. 2.29]
    In this case $X$ is the blow-up of $Q$ in a conic. A toric LG model of $X$ is given by
    $$
    f = x+y+z+\frac{1}{x} + \frac{1}{y} + \frac{1}{xz} + \frac{1}{yz}.
    $$
    A parametrized toric LG model is given by
    $$
    \widetilde{f} = x + y + a_{1}z + \frac{a_{2}}{x} + \frac{a_{2}}{y} + \frac{1}{xz} + \frac{1}{yz}.
    $$
    Hence the associated Sarkisov link is
    $$
    \begin{tikzcd}
     Q \arrow[d] & X \arrow[l] \arrow[equal,r] \arrow[dd] & X \arrow[d] \\
     pt \arrow[equal,dr] & & \mathbb{P}^1 \arrow[dl] \\
     & pt &
   \end{tikzcd}
    $$
    where a general fibre of $X \rightarrow \mathbb{P}^1$ is isomorphic to $\mathbb{P}^1 \times \mathbb{P}^1$. This is a link of type I/III.
\end{statement}

\begin{statement}[No. 2.30]
    In this case, $X$ is a blow-up of $\mathbb{P}^3$ in a conic. A toric LG model of $X$ is given by
    $$
    f = x+y+z+\frac{1}{xy}+\frac{1}{xz} + \frac{1}{xyz}.
    $$
    A parametrized toric LG model is given by
    $$
    \widetilde{f} = x + y + z + \frac{a_{2}}{xy} + \frac{a_{2}}{xz} + \frac{a_{1}}{xyz}.
    $$
    Hence the associated Sarkisov link is
    $$
    
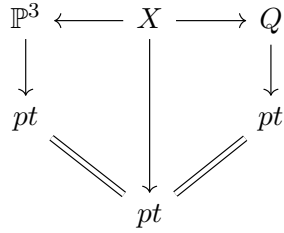
\begin{tikzcd}
     \mathbb{P}^3 \arrow[d] & X \arrow[l] \arrow[r] \arrow[dd] & Q \arrow[d] \\
     pt \arrow[equal,dr] & & pt \arrow[equal,dl] \\
     & pt &
   \end{tikzcd}
    $$
    This is a link of type II.
\end{statement}

\begin{statement}[No. 2.31]
    In this case, $X$ is a blow-up of $Q$ in a line. A toric LG model of $X$ is given by
    $$
    f = x + y + z + \frac{1}{x} + \frac{1}{xy} + \frac{1}{yz}.
    $$
    A parametrized toric LG model is given by
    $$
    \widetilde{f} = x + a_{1}y + z + \frac{a_{2}}{x} + \frac{1}{xy} + \frac{1}{yz}.
    $$
    Hence the associated Sarkisov link is
    $$
    
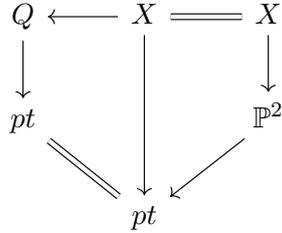
\begin{tikzcd}
     Q \arrow[d] & X \arrow[l] \arrow[equal,r] \arrow[dd] & X \arrow[d] \\
     pt \arrow[equal,dr] & & \mathbb{P}^2 \arrow[dl] \\
     & pt &
    \end{tikzcd}
    $$
    This is a link of type I/III.
\end{statement}

\begin{statement}[No. 2.32]
    In this case $X$ is a general divisor of bidegree $(1,1)$ on $\mathbb{P}^2 \times \mathbb{P}^2$. The associated Sarkisov link is
    $$
    
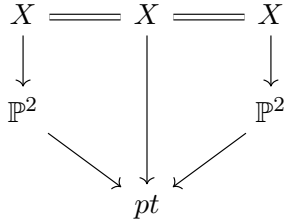
\begin{tikzcd}
     X \arrow[d] & X \arrow[equal,l] \arrow[equal,r] \arrow[dd] & X \arrow[d] \\
     \mathbb{P}^2 \arrow[dr] & & \mathbb{P}^2 \arrow[dl] \\
     & pt &
    \end{tikzcd}
    $$
    This is a link of type IV. A toric LG model of $X$ is given by
    $$
    f = x + y + z + \frac{1}{x} + \frac{1}{y} + \frac{1}{xyz}.
    $$
\end{statement}

\begin{statement}[No. 2.33]
    In this case $X \simeq \mathrm{Bl}_{L}\mathbb{P}^3$, where $L$ is a line on $\mathbb{P}^3$. The associated Sarkisov link is
    $$
    
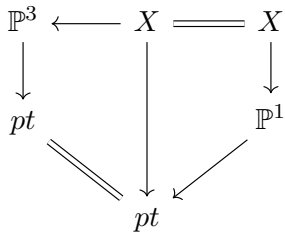
\begin{tikzcd}
    \mathbb{P}^3 \arrow[d] & X \arrow[l] \arrow[equal,r] \arrow[dd] & X \arrow[d] \\
    pt \arrow[equal,dr] & & \mathbb{P}^1 \arrow[dl] \\
    & pt &
    \end{tikzcd}
    $$
    where a general fibre of $X \rightarrow \mathbb{P}^1$ is isomorphic to $\mathbb{P}^2$. This is a link of type I/III. A toric LG model of $X$ is given by
    $$
    f = x + y + z + xy + \frac{1}{xyz}.
    $$
\end{statement}

\begin{statement}[No. 2.34]
    In this case $X \simeq \mathbb{P}^2 \times \mathbb{P}^1$. The associated Sarkisov link is
    $$
    
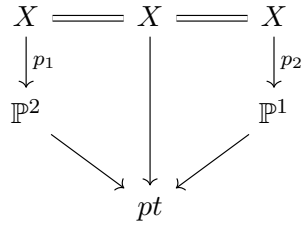
\begin{tikzcd}
    X \arrow[d,"p_{1}"] & X \arrow[equal,l] \arrow[equal,r] \arrow[dd] & X \arrow[d,"p_{2}"] \\
     \mathbb{P}^2 \arrow[dr] & & \mathbb{P}^1 \arrow[dl] \\
     & pt &
    \end{tikzcd}
    $$
    This is a link of type IV. A toric LG model of $X$ is given by
    $$
    x + y + z + \frac{1}{x} + \frac{1}{yz}.
    $$
\end{statement}

\begin{statement}[No. 2.35]
    In this case $X \simeq \mathrm{Bl}_{p}\mathbb{P}^3$, where $p$ is a point on $\mathbb{P}^3$. The associated Sarkisov link is
    $$
   
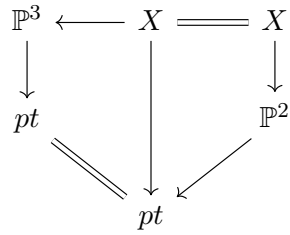
\begin{tikzcd}
     \mathbb{P}^3 \arrow[d] & X \arrow[l] \arrow[equal,r] \arrow[dd] & X \arrow[d] \\
     pt \arrow[equal,dr] & & \mathbb{P}^2 \arrow[dl] \\
     & pt &
   \end{tikzcd}
   $$
   This is a link of type I/III. A toric LG model of $X$ is given by
   $$
   x + y + z + \frac{1}{x} + \frac{1}{xyz}.
   $$
\end{statement}

\begin{statement}[No. 2.36]
    In this case, $X \simeq \mathbb{P}(\mathcal{O}_{\mathbb{P}^2} \oplus \mathcal{O}_{\mathbb{P}^2}(2))$. Notice that $X$ is a toric variety. Hence the associated Sarkisov link is
    $$
    
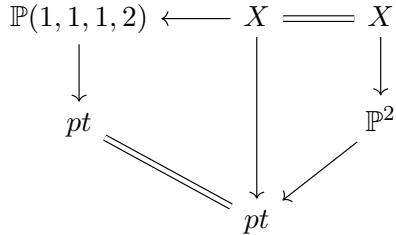
\begin{tikzcd}
    \mathbb{P}(1,1,1,2) \arrow[d] & X \arrow[l] \arrow[equal,r] \arrow[dd] & X \arrow[d] \\
     pt \arrow[equal,dr] & & \mathbb{P}^2 \arrow[dl] \\
     & pt &
    \end{tikzcd}
    $$
    This is a link of type I/III. A toric LG model of $X$ is given by
    $$
    f = x + y + z + \frac{1}{x} + \frac{1}{x^{2}yz}.
    $$
\end{statement}

\subsection{Picard number 3}
    Now we study smooth Fano threefolds of Picard number 3. The elementary syzygies associated to central models of rank 3 are elementary relations of Sarkisov links. Similar to the previous subsection, we compute them in multiple ways including the LG models method, while some of them can be directly derived from \cite{Matsuki1995}.

\begin{statement}[No. 3.1]
    In this case $X$ is a double cover of $\mathbb{P}^1 \times \mathbb{P}^1 \times \mathbb{P}^1$ branched over a divisor of degree $(2,2,2)$. The associated elementary relation is
    $$
    
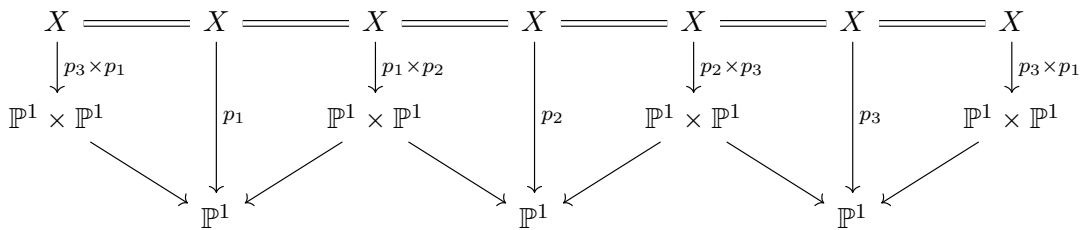
\begin{tikzcd}
    X \arrow[d,"p_{3} \times p_{1}"] \arrow[equal,r] & X \arrow[dd,"p_{1}"] \arrow[equal,r] & X \arrow[d,"p_{1} \times p_{2}"] \arrow[equal,r] & X \arrow[dd,"p_{2}"] \arrow[equal,r] & X \arrow[d,"p_{2} \times p_{3}"] \arrow[equal,r] & X \arrow[dd,"p_{3}"] \arrow[equal,r] & X \arrow[d,"p_{3} \times p_{1}"] \\    
    \mathbb{P}^1 \times \mathbb{P}^1 \arrow[dr] & & \mathbb{P}^1 \times \mathbb{P}^1 \arrow[dl] \arrow[dr] & & \mathbb{P}^1 \times \mathbb{P}^1 \arrow[dl] \arrow[dr] & & \mathbb{P}^1 \times \mathbb{P}^1 \arrow[dl] \\    
    & \mathbb{P}^1 & & \mathbb{P}^1 & & \mathbb{P}^1 &
    \end{tikzcd}
    $$
\end{statement}

\begin{statement}[No. 3.2]\label{Case 3.2}
    In this case $X$ is a divisor on the $\mathbb{P}^{2}$ bundle $W = \mathbb{P}(\mathcal{O}_{\mathbb{P}^1 \times \mathbb{P}^1} \oplus (\mathcal{O}_{\mathbb{P}^1 \times \mathbb{P}^1}(-1,-1))^{\oplus 2})$ over $\mathbb{P}^1 \times \mathbb{P}^1$ from the linear system $|L^{\otimes 2} \otimes \mathcal{O}_{\mathbb{P}^1 \times \mathbb{P}^1}(2,3)|$, such that $X \cap Y$ is irreducible. Here $L$ is the tautological line bundle of $W$ and $Y$ is a
    member of $|L|$. Let $p_{1},p_{2}$ be the first and second projection, and $F_{1},F_{2}$ be the class of a general fibre of $p_{1},p_{2}$. Then the effective cone $\mathrm{Eff}_{\mathbb{R}}(X)$ is generated by $L,F_{1},F_{2}$. A toric LG model of $X$ is given by
    \begin{align*}
        f &= x + y + z + \frac{x^{2}}{yz} + \frac{3x}{z} + \frac{3y}{z} + \frac{y^{2}}{xz} + \frac{3x}{y} + \frac{3y}{x} + \frac{3z}{y} + \frac{3z}{x} + \frac{z^{2}}{xy} + \frac{1}{x} + \frac{1}{y} + \frac{z}{xy} \\
        & = x + y + z + \frac{(x+y+z)^3}{xyz} + \frac{x+y+z}{xy} - 6.
    \end{align*}
    A parametrized toric LG model of $X$ is given by
    \begin{multline*}
    \widetilde{f} = a_{1}(x + y + z) + \frac{1}{x} + \frac{1}{y} +  \frac{z}{xy} \\
    + \frac{a_{2}(x+y)^3 + (2a_{2}+a_{3})(x^{2}z + y^{2}z) +(a_{2}+2a_{3})(xz^{2} + yz^{2}) + a_{3}z^{3}}{xyz}.
    \end{multline*}
    The associated elementary relation is
    $$
    \begin{tikzcd}
        X \arrow[d] & X \arrow[dd,"p_{2}"] \arrow[equal,l] \arrow[r] & Y' \arrow[d] & Y \arrow[dd] \arrow[dashed,l] \arrow[dashed,r] & Y'' \arrow[d] & X \arrow[dd,"p_{1}"] \arrow[l] \arrow[r,equal] & X \arrow[d] \\
        \mathbb{P}^1 \times \mathbb{P}^1 \arrow[dr] & & \mathbb{P}^1 \arrow[equal,dl] \arrow[dr] & & \mathbb{P}^1 \arrow[dl] \arrow[equal,dr] & & \mathbb{P}^1 \times \mathbb{P}^1 \arrow[dl]\\
        & \mathbb{P}^1 & & pt & & \mathbb{P}^1
    \end{tikzcd}
    $$
    where $Y'$ and $Y''$ are smooth weak Fano varieties, and $Y$ is a terminal Fano variety with an ordinary double point. The birational map $Y' \dashrightarrow Y''$ is a flop over $Y$. A toric LG model of $Y$ is given by
    $$
    f_{Y} = x + y + z + \frac{1}{x} + \frac{1}{y} +  \frac{z}{xy} + \frac{(x+y)^3 + 2x^{2}z + 2y^{2}z + xz^{2} + yz^{2}}{xyz}.
    $$
    A general fibre of $Y' \rightarrow \mathbb{P}^1$ is $\mathbb{P}^1 \times \mathbb{P}^1$. A general fibre of $Y'' \rightarrow \mathbb{P}^1$ is a smooth del Pezzo surface of degree 4. The period of $f_{Y}$ coincides with the period of $V_{16}$. Hence $Y$ deforms to $V_{16}$. Here we give an explicit mutation from $f_{Y}$ to the toric LG model of $V_{16}$ in \cref{table: toric LG model picard number 1}. Indeed, we can first take a mutation given by weight $w_{1} = (0,0,1)$ and factor $a_{1} = x + y$, and after that we take a change of coordinates $y \mapsto xy$. The new toric LG model is
    $$
    f'_{Y} = \frac{(y+1)(z+1)(x^{2}yz + xyz + xy + xz + x + z)}{xyz}.
    $$
    Then we take 3 consecutive mutations given by
    \begin{align*}
        w_{2} = (0,0,-1)&,a_{2} = x + 1, \\
        w_{3} = (0,-1,0)&,a_{3} = z + 1, \\
        w_{4} = (0,0,1)&,a_{4} = y + 1.
    \end{align*}
    After a change of coordinates $x \mapsto xz$, the new toric LG model is
    $$
    f''_{Y} = \frac{(x+1)(y+1)(z+1)(xz+yz+z+1)}{xyz}.
    $$
    Finally, we make a mutation given by $w_{5} = (0,0,-1)$ and $a_{5} = x + y + 1$ and obtain the desired toric LG model.
\end{statement}

\begin{statement}[No. 3.3]
    In this case $X$ is a divisor on $\mathbb{P}^1 \times \mathbb{P}^1 \times \mathbb{P}^2$ of type $(1,1,2)$. The associated elementary relation is
    $$
    \adjustbox{scale=0.9,center}{%
    \begin{tikzcd}[cramped,column sep=small]
        \mathbb{P}^1 \times \mathbb{P}^2 \arrow[d] & \mathbb{P}^1 \times \mathbb{P}^2 \arrow[dd] \arrow[equal,l] \arrow[equal,r] & \mathbb{P}^1 \times \mathbb{P}^2 \arrow[d] & X \arrow[dd] \arrow[l] \arrow[r] & \mathbb{P}^1 \times \mathbb{P}^2 \arrow[d] & \mathbb{P}^1 \times \mathbb{P}^2 \arrow[dd] \arrow[equal,l] \arrow[equal,r] & \mathbb{P}^1 \times \mathbb{P}^2 \arrow[d] & X \arrow[dd] \arrow[l] \arrow[equal,r] & X \arrow[d] & X \arrow[dd] \arrow[equal,l] \arrow[r] & \mathbb{P}^1 \times \mathbb{P}^2 \arrow[d] \\
        \mathbb{P}^1 \arrow[dr] & & \mathbb{P}^2 \arrow[dl] \arrow[equal,dr] & & \mathbb{P}^2 \arrow[equal,dl] \arrow[dr] & & \mathbb{P}^1 \arrow[dl] \arrow[equal,dr] & & \mathbb{P}^1 \times \mathbb{P}^1 \arrow[dl] \arrow[dr] & & \mathbb{P}^1 \arrow[equal,dl] \\
        & pt & & \mathbb{P}^2 & & pt & & \mathbb{P}^1 & & \mathbb{P}^1
    \end{tikzcd}
    }
    $$
\end{statement}

\begin{statement}[No. 3.4]
    In this case $X$ is a blow-up of $X_{2.18}$ in a smooth fibre of the fibration $X_{2.18} \xrightarrow{\text{double cover}} \mathbb{P}^1 \times \mathbb{P}^2 \rightarrow \mathbb{P}^2$. The associated elementary relation is
    $$
    \begin{tikzcd}[cramped]
        X_{2.18} \arrow[d] & X_{2.18} \arrow[dd] \arrow[equal,l] \arrow[equal,r] & X_{2.18} \arrow[d] & X \arrow[dd] \arrow[l] \arrow[equal,r] & X \arrow[d] & X \arrow[dd] \arrow[equal,l] \arrow[equal,r] & X \arrow[d] & X \arrow[dd] \arrow[equal,l] \arrow[r] & X_{2.18} \arrow[d] \\
        \mathbb{P}^1 \arrow[dr] &  & \mathbb{P}^2 \arrow[dl] \arrow[equal,dr] & & \mathbb{F}_{1} \arrow[dl] \arrow[dr] & & \mathbb{P}^1 \times \mathbb{P}^1 \arrow[dl] \arrow[dr] & & \mathbb{P}^1 \arrow[equal,dl]\\
        & pt & & \mathbb{P}^2 & & \mathbb{P}^1 & & \mathbb{P}^1 &
    \end{tikzcd}
    $$
\end{statement}

\begin{statement}[No. 3.5]
    In this case $X$ is a blow-up of $\mathbb{P}^1 \times \mathbb{P}^2$ in a curve $C$ of degree $(5,2)$ such that the projection of $C$ to $\mathbb{P}^2$ is an embedding. A toric LG model of $X$ is given by
    \begin{align*}
    f &= \frac{x^2}{y} + 3x + 3y + \frac{y^{2}}{x} + z + 
    \frac{yz}{x} + \frac{x}{z} + \frac{y}{z} + \frac{2x}{y}  +  \frac{2y}{x} + \frac{1}{x} + \frac{1}{y} + \frac{1}{z} \\
     &= \frac{(x+y)(x+y+1)^2}{xy} + \frac{(x+y)z}{x} + \frac{x+y+1}{z}.
    \end{align*}
    Here the constant term is ignored. Taking a change of coordinates $y \mapsto \frac{1}{y}$, we can rewrite the Laurent polynomial as
    $$
    f = \frac{(xy+1)(xy+y+1)^2}{xy^{2}} + z + \frac{z}{xy} + \frac{x}{z} + \frac{1}{yz} + \frac{1}{z}.
    $$
    A parametrized toric LG model of $X$ is given by
    \begin{multline*}
    \widetilde{f} = \frac{xy^{3} + 2a_{3}x^{2}y^{3} + a_{3}^{2}x^{3}y^{3} + (a_{1}+2a_{3}^{3})x^{2}y^{2} + (2a_{1}a_{3}+a_{3}^{4})xy + a_{1}a_{3}^{2} + (a_{1}+a_{3}^{3})y + a_{3}y^{2}}{xy^{2}} \\
    + z + \frac{a_{3}z}{xy} + \frac{a_{2}a_{3}x}{z} + \frac{a_{2}a_{3}^{2}}{yz} + \frac{a_{2}}{z}.
    \end{multline*}
    Hence the associated elementary relation is
    $$
    \begin{tikzcd}[cramped]
        \mathbb{P}^1 \times \mathbb{P}^2 \arrow[d] & \mathbb{P}^1 \times \mathbb{P}^2 \arrow[dd] \arrow[equal,l] \arrow[equal,r] & \mathbb{P}^1 \times \mathbb{P}^2 \arrow[d] & X \arrow[dd] \arrow[l] \arrow[r] & Y' \arrow[d] & Y \arrow[dd] \arrow[dashed,l] \arrow[dashed,r] & Y'' \arrow[d] & X \arrow[dd] \arrow[l] \arrow[r] & \mathbb{P}^1 \times \mathbb{P}^2 \arrow[d] \\
        \mathbb{P}^2 \arrow[dr] & & \mathbb{P}^1 \arrow[dl] \arrow[equal,dr] & & \mathbb{P}^1 \arrow[equal,dl] \arrow[dr] & & \mathbb{P}^2 \arrow[dl] \arrow[equal,dr] & & \mathbb{P}^2 \arrow[equal,dl] \\ 
        & pt & & \mathbb{P}^1 & & pt & & \mathbb{P}^2
    \end{tikzcd}
    $$
    where $Y'$ and $Y''$ are smooth weak Fano varieties, and $Y$ is a terminal Fano variety with an ordinary double point. The birational map $Y' \dashrightarrow Y''$ is a flop over $Y$. A toric LG model of $Y$ is given by
    $$
    f_{Y} = \frac{xy^{3} + 2x^{2}y^{3} + x^{3}y^{3} + 2x^{2}y^{2} + xy + y + y^{2}}{xy^{2}} + z + \frac{z}{xy} + \frac{x}{z} + \frac{1}{yz} + \frac{1}{z}.
    $$
    A general fibre of $Y' \rightarrow \mathbb{P}^1$ is a smooth del Pezzo surface of degree 5. The birational map $Y' \dashrightarrow Y''$ is a flop over $Y$. The period of $f_{Y}$ is the quantum period of $V_{22}$. Hence $Y$ deforms to $V_{22}$.
\end{statement}

\begin{statement}[No. 3.6]
    In this case $X$ is a blow-up of $\mathbb{P}^3$ in the disjoint union of a line $L$ and an elliptic curve $C$ which is an intersection of two quadrics. The associated elementary relation is
    $$
    \begin{tikzcd}[cramped]
        \mathbb{P}^3 \arrow[d] & \mathrm{Bl}_{L}\mathbb{P}^3 \arrow[dd] \arrow[l] \arrow[equal,r] & \mathrm{Bl}_{L}\mathbb{P}^3 \arrow[d] & X \arrow[dd] \arrow[l] \arrow[equal,r] & X \arrow[d] & X \arrow[dd] \arrow[equal,l] \arrow[r] & \mathrm{Bl}_{C}\mathbb{P}^3 \arrow[d] & \mathrm{Bl}_{C}\mathbb{P}^3 \arrow[dd] \arrow[equal,l] \arrow[r] & \mathbb{P}^3 \arrow[d] \\
        pt \arrow[equal,dr] & & \mathbb{P}^1 \arrow[dl] \arrow[equal,dr] & & \mathbb{P}^1 \times \mathbb{P}^1 \arrow[dl] \arrow[dr] & & \mathbb{P}^1 \arrow[equal,dl] \arrow[dr] & & pt \arrow[equal,dl] \\
        & pt & & \mathbb{P}^{1} & & \mathbb{P}^1 & & pt &
    \end{tikzcd}
    $$
\end{statement}

\begin{statement}[No. 3.7]
    In this case $X$ is a blow-up of a variety $X_{2.32} \subseteq \mathbb{P}^2 \times \mathbb{P}^2$ in the family No. 2.32 in an elliptic curve $C$ which is the intersection of two divisors in $|-\frac{1}{2}K_{X_{2.32}}|$. Let $p_{1},p_{2}$ be the first and second projection respectively, and $H$ be the hyperplane class in $\mathbb{P}^2$. Then the effective cone $\mathrm{Eff}_{\mathbb{R}}(X)$ is generated by 4 extremal rays given by $3p_{1}^{*}H - E, 3p_{2}^{*}H - E, E, p_{1}^{*}H + p_{2}^{*}H - E$. 
    
    A toric LG model of $X$ is given by
    $$
    f = x + y + z + \frac{y}{z} + \frac{y}{x} + \frac{z}{y} + \frac{z}{x} + \frac{1}{z} + 
    \frac{y}{xz} + \frac{1}{y} + \frac{2}{x} + \frac{z}{xy} + \frac{1}{xz} + \frac{1}{xy}.
    $$
    Taking a change of coordinates $y \mapsto yz$, we can rewrite the Laurent polynomial as
    $$
    f = x + y + z + yz + \frac{yz}{x} + \frac{z}{x} + \frac{y}{x} + \frac{2}{x} + \frac{1}{y} + \frac{1}{z} + \frac{1}{xy} + \frac{1}{xz} + \frac{1}{yz} + \frac{1}{xyz}.
    $$
    A parametrized toric LG model of $X$ is given by
    $$
    \widetilde{f} = a_{4}x + a_{3}y + a_{3}z + a_{1}a_{3}yz + \frac{a_{1}^{2}yz}{x} + \frac{1+a_{1}^{3}}{x} + \frac{a_{1}a_{3}}{y} + \frac{a_{1}a_{3}}{z} + \frac{a_{1}y}{x} + \frac{a_{1}z}{x} + \frac{a_{1}^{2}}{xy} + \frac{a_{1}^{2}}{xz} + \frac{a_{3}}{yz} + \frac{a_{1}}{xyz}.
    $$
    with the relation $(a_{1},a_{2},a_{3},a_{4}) \sim (\lambda a_{1}, \lambda a_{2}, \lambda^{-1} a_{3}, \lambda^{-3} a_{4})$. On the other chart, a parametrized toric LG model of $X$ is given by 
    $$
    \widetilde{f} = a'_{4}x + a_{2}a'_{3}y + a_{2}a'_{3}z + a'_{3}yz + \frac{a_{2}yz}{x} + \frac{1+a_{2}^{3}}{x} + \frac{a'_{3}}{y} + \frac{a'_{3}}{z} + \frac{a_{2}^{2}y}{x} + \frac{a_{2}^{2}z}{x} + \frac{a_{2}}{xy} + \frac{a_{2}}{xz} + \frac{a_{2}a'_{3}}{yz} + \frac{a_{2}^{2}}{xyz}.
    $$
    Hence the associated elementary relation is
    $$
    \adjustbox{scale=0.9,center}{
    \begin{tikzcd}[cramped, column sep=small]
        X_{2.32} \arrow[d] & X_{2.32} \arrow[dd] \arrow[equal,l] \arrow[equal,r] & X_{2.32} \arrow[d] & X \arrow[dd] \arrow[l] \arrow[r] & \mathbb{P}^1 \times \mathbb{P}^2 \arrow[d] & \mathbb{P}^1 \times \mathbb{P}^2 \arrow[dd] \arrow[equal,l] \arrow[equal,r] & \mathbb{P}^1 \times \mathbb{P}^2 \arrow[d] & X \arrow[dd] \arrow[l] \arrow[r] & \mathbb{P}^1 \times \mathbb{P}^2 \arrow[d] & \mathbb{P}^1 \times \mathbb{P}^2 \arrow[dd] \arrow[equal,l] \arrow[equal,r] & \mathbb{P}^1 \times \mathbb{P}^2 \arrow[d] & X \arrow[dd] \arrow[l] \arrow[r] & X_{2.32} \arrow[d] \\
        \mathbb{P}^2 \arrow[dr] & & \mathbb{P}^2 \arrow[dl] \arrow[equal,dr]  & & \mathbb{P}^2 \arrow[equal,dl] \arrow[dr] & & \mathbb{P}^1 \arrow[dl] \arrow[equal,dr] & & \mathbb{P}^{1} \arrow[equal,dl] \arrow[dr] & & \mathbb{P}^2 \arrow[dl] \arrow[equal,dr] & & \mathbb{P}^2 \arrow[equal,dl] \\
        & pt & & \mathbb{P}^2 & & pt & & \mathbb{P}^1 & & pt & & \mathbb{P}^2
    \end{tikzcd}
    }
    $$
\end{statement}

\begin{statement}[No. 3.8]
    In this case $X$ is a divisor on $\mathbb{F}_{1} \times \mathbb{P}^2$ from the linear system $|p_{1}^{*}g^{*}\mathcal{O}_{\mathbb{P}^2}(1)\otimes p_{2}^{*}\mathcal{O}_{\mathbb{P}^2}(2)|$. Here $p_{1},p_{2},g$ are projections and blow-up morphisms as follows:
    $$
    \begin{tikzcd}
      &  \mathbb{F}_{1} \times \mathbb{P}^2 \arrow[dl,"p_{1}"] \arrow[dr,"p_{2}"] & \\
      \mathbb{F}_{1} \arrow[d,"g"] & & \mathbb{P}^2 \\
      \mathbb{P}^2 & &
    \end{tikzcd}
    $$
    The associated elementary relation is
    $$
    \begin{tikzcd}[cramped,column sep=small]
        X \arrow[d,"p_{1}"] & X \arrow[dd,"g \circ p_{1}"] \arrow[equal,l] \arrow[r] & X_{2.24} \arrow[d] & X_{2.24} \arrow[dd] \arrow[equal,l] \arrow[equal,r] & X_{2.24} \arrow[d] & X \arrow[dd,"p_{2}"] \arrow[l] \arrow[r] & \mathbb{P}^1 \times \mathbb{P}^2 \arrow[d] & \mathbb{P}^1 \times \mathbb{P}^2 \arrow[dd] \arrow[equal,l] \arrow[equal,r] & \mathbb{P}^1 \times \mathbb{P}^2 \arrow[d] & X \arrow[dd] \arrow[l] \arrow[equal,r] & X \arrow[d,"p_{1}"] \\
        \mathbb{F}_{1} \arrow[dr,"g"] & & \mathbb{P}^2  \arrow[equal,dl] \arrow[dr] & & \mathbb{P}^2 \arrow[dl] \arrow[equal,dr] & & \mathbb{P}^2 \arrow[equal,dl] \arrow[dr] & & \mathbb{P}^1 \arrow[dl] \arrow[equal,dr] & & \mathbb{F}_{1} \arrow[dl] \\
        & \mathbb{P}^2 & & pt & & \mathbb{P}^2 & & pt & & \mathbb{P}^1
    \end{tikzcd}
    $$
\end{statement}

\begin{statement}[No. 3.9]\label{case 3.9}
    In this case $X$ is a blow-up of the cone $W_{4} \subseteq \mathbb{P}^6$ over the Veronese surface $S_{4} \subseteq \mathbb{P}^5$ in the disjoint union of the vertex and a quartic $C$ in $S_{4} \simeq \mathbb{P}^2$. The effective cone $\mathrm{Eff}_{\mathbb{R}}(X)$ is generated by 4 extremal rays $2H-2E_{1}-E_{2},H-E_{2},E_{1},E_{2}$, where $E_{1}$ is the exceptional divisor centered at the vertex, $E_{2}$ is the exceptional divisor centered at $C$ and $H$ is the class of a hyperplane section. A toric LG model of $X$ is given by
    $$
    f = x + y + z + \frac{x^{2}}{yz} + \frac{y}{x} + \frac{z}{x} + \frac{2x}{yz} + \frac{1}{x} + \frac{1}{yz}.
    $$
    Taking a change of coordinates $y \mapsto xy, z \mapsto xz$, we can rewrite the Laurent polynomial as
    $$
    f = x + y + z + xy + xz + \frac{1}{x} + \frac{1}{yz} + \frac{2}{xyz} + \frac{1}{x^{2}yz}.
    $$
    A parametrized toric LG model of $X$ is given by
    $$
    \widetilde{f} = a_{2}x + y + z + a_{4}xy + a_{4}xz + \frac{a_{3}}{x} + \frac{(a_{4}x+1)^2}{x^{2}yz}
    $$
    with the relation $(a_{1},a_{2},a_{3},a_{4}) \sim (\lambda a_{1}, \lambda^{-2} a_{2}, \lambda^{2} a_{3}, \lambda^{-1} a_{4})$. On another chart, we can write the parametrized LG model as
    $$
    \widetilde{f} = a'_{2}x+a_{1}y+a_{1}z + a'_{4}xy + a'_{4}xz + \frac{1}{x} + \frac{(a'_{4}x+a_{1})^{2}}{x^{2}yz}.
    $$
    Hence the associated elementary relation is
    $$
    \begin{tikzcd}[cramped, column sep=small]
        X_{2.36} \arrow[d] & X \arrow[dd] \arrow[l] \arrow[r] & X'_{2.36} \arrow[d] & X'_{2.36} \arrow[dd] \arrow[equal,l] \arrow[r] & \mathbb{P}(1,1,1,2) \arrow[d] & Y_{1} \arrow[dd] \arrow[l] \arrow[r] & Y_{2} \arrow[d] & Y'_{1} \arrow[dd] \arrow[l] \arrow[r] & \mathbb{P}(1,1,1,2) \arrow[d] & X_{2.36} \arrow[dd] \arrow[l] \arrow[equal,r] & X \arrow[d] \\
        \mathbb{P}^2 \arrow[equal,dr] & & \mathbb{P}^2 \arrow[equal,dl] \arrow[dr] & & pt \arrow[equal,dl] \arrow[equal,dr] & & pt \arrow[equal,dl] \arrow[equal,dr] & & pt \arrow[equal,dl] \arrow[equal,dr] & & \mathbb{P}^2 \arrow[dl] \\
         & \mathbb{P}^2 & & pt & & pt & & pt & & pt
    \end{tikzcd}
    $$
    where $Y_{1}$ and $Y'_{1}$ have a quotient singularity of the form $\frac{1}{2}(1,1,1)$. A toric LG model of $Y_{1}$ and $Y'_{1}$ is given by
    $$
    f_{Y_{1}} = f_{Y'_{1}} = x + y + z + xy + xz + \frac{1}{yz} + \frac{2}{xyz} + \frac{1}{x^{2}yz}.
    $$
    The variety $Y_{2}$ has two quotient singularities of the form $\frac{1}{2}(1,1,1)$. A toric LG model of $Y_{2}$ is given by
    $$
    f_{Y_{2}} = y + z + xy + xz + \frac{1}{yz} + \frac{2}{xyz} + \frac{1}{x^{2}yz}.
    $$
\end{statement}

\begin{statement}[No. 3.10]
    In this case $X$ is a blow-up of $Q$ in the disjoint union of two conics $C_{1},C_{2}$. The associated elementary relation is
    $$
    \begin{tikzcd}[cramped]
        Q \arrow[d] & \mathrm{Bl}_{C_{1}}Q \arrow[dd] \arrow[l] \arrow[equal,r] & \mathrm{Bl}_{C_{1}}Q \arrow[d] & X \arrow[dd] \arrow[l] \arrow[equal,r] & X \arrow[d] & X \arrow[dd] \arrow[equal,l] \arrow[r] & \mathrm{Bl}_{C_{2}}Q \arrow[d] & \mathrm{Bl}_{C_{2}}Q \arrow[dd] \arrow[equal,l] \arrow[r] & Q \arrow[d] \\
        pt \arrow[equal,dr] & & \mathbb{P}^1 \arrow[dl] \arrow[equal,dr] & & \mathbb{P}^1 \times \mathbb{P}^1 \arrow[dl] \arrow[dr] & & \mathbb{P}^1 \arrow[equal,dl] \arrow[dr] & & pt \arrow[equal,dl] \\
        & pt & & \mathbb{P}^1 & & \mathbb{P}^1 & & pt &
    \end{tikzcd}
    $$
\end{statement}

\begin{statement}[No. 3.11]
    In this case $X$ is obtained by first blowing up $\mathbb{P}^3$ in a point $P$, then blowing up an elliptic curve $\widetilde{C}$ which is the intersection of two divisors in $|-\frac{1}{2}K_{\mathrm{Bl}_{P}\mathbb{P}^3}|$. The image $C$ of $\widetilde{C}$ in $\mathbb{P}^3$ is the intersection of two quadrics passing through $P$. The associated elementary relation is
    $$
    \begin{tikzcd}[cramped,column sep=small]
    \mathbb{P}^3 \arrow[d] & \mathrm{Bl}_{P} \mathbb{P}^3 \arrow[dd] \arrow[l] \arrow[equal,r] & \mathrm{Bl}_{P} \mathbb{P}^3 \arrow[d] & X \arrow[dd] \arrow[l] \arrow[r] & \mathbb{P}^1 \times \mathbb{P}^2 \arrow[d] & \mathbb{P}^1 \times \mathbb{P}^2 \arrow[dd] \arrow[equal,l] \arrow[equal,r] & \mathbb{P}^1 \times \mathbb{P}^2 \arrow[d] & X \arrow[dd] \arrow[l] \arrow[r] & \mathrm{Bl}_{C}\mathbb{P}^3 \arrow[d] & \mathrm{Bl}_{C}\mathbb{P}^3 \arrow[dd] \arrow[equal,l] \arrow[r] & \mathbb{P}^3 \arrow[d] \\
    pt \arrow[equal,dr] & & \mathbb{P}^2 \arrow[dl] \arrow[equal,dr] & & \mathbb{P}^2 \arrow[equal,dl] \arrow[dr] & & \mathbb{P}^1 \arrow[dl] \arrow[equal,dr] & & \mathbb{P}^1 \arrow[equal,dl] \arrow[dr] & & pt \arrow[equal,dl]\\
    & pt & & \mathbb{P}^2 & & pt & & \mathbb{P}^1 & & pt
    \end{tikzcd}
    $$
\end{statement}

\begin{statement}[No. 3.12]
    In this case $X$ is a blow-up of $\mathbb{P}^3$ in the disjoint union of a line $L$ and a twisted cubic $C$. The associated elementary relation is
    $$
    \begin{tikzcd}[cramped,column sep=small]
        \mathbb{P}^3 \arrow[d] & \mathrm{Bl}_{L}\mathbb{P}^3 \arrow[dd] \arrow[l] \arrow[equal,r] & \mathrm{Bl}_{L}\mathbb{P}^3 \arrow[d] & X \arrow[dd] \arrow[l] \arrow[r] & \mathbb{P}^1 \times \mathbb{P}^2 \arrow[d] & \mathbb{P}^1 \times \mathbb{P}^2 \arrow[dd] \arrow[equal,l] \arrow[equal,r] & \mathbb{P}^1 \times \mathbb{P}^2 \arrow[d] & X \arrow[dd] \arrow[l] \arrow[r] & \mathrm{Bl}_{C}\mathbb{P}^3 \arrow[d] & \mathrm{Bl}_{C}\mathbb{P}^3 \arrow[dd] \arrow[equal,l] \arrow[r] & \mathbb{P}^3 \arrow[d] \\
        pt \arrow[equal,dr] & & \mathbb{P}^1 \arrow[dl] \arrow[equal,dr] & & \mathbb{P}^1 \arrow[equal,dl] \arrow[dr] & & \mathbb{P}^2 \arrow[dl] \arrow[equal,dr] & & \mathbb{P}^2 \arrow[equal,dl] \arrow[dr] & & pt \arrow[equal,dl] \\
        & pt & & \mathbb{P}^1 & & pt & & \mathbb{P}^2 & & pt &
    \end{tikzcd}
    $$
\end{statement}

\begin{statement}[No. 3.13]
    In this case $X$ is a blow-up of a threefold $X_{2.32} \subseteq \mathbb{P}^2 \times \mathbb{P}^2$ from the family No. 2.32 in a general curve $C$ of bidegree $(2,2)$. A toric LG model of $X$ is given by
    $$
    f = x + y + z + yz + xy + \frac{1}{x} + \frac{1}{y} + \frac{1}{yz} + \frac{1}{xyz}.
    $$
    A parametrized LG model of $X$ is given by
    $$
    \widetilde{f} = a_{1}x + a_{2}y + a_{1}z + a_{3}yz + a_{3}xy + \frac{a_{1}}{x} + \frac{a_{2}}{y} + \frac{a_{3}}{yz} + \frac{a_{2}}{xyz}.
    $$
    
    Hence the associated elementary relation is
    $$
    \adjustbox{scale=0.9,center}{%
    \begin{tikzcd}[cramped, column sep=small]
        X_{2.32} \arrow[d] & X_{2.32} \arrow[dd] \arrow[equal,l] \arrow[equal,r] & X_{2.32} \arrow[d] & X \arrow[dd] \arrow[l] \arrow[r] & X'_{2.32} \arrow[d] & X'_{2.32} \arrow[dd] \arrow[equal,l] \arrow[equal,r] & X'_{2.32} \arrow[d] & X \arrow[dd] \arrow[l] \arrow[r] & X''_{2.32} \arrow[d] & X''_{2.32} \arrow[dd] \arrow[equal,l] \arrow[equal,r] & X''_{2.32} \arrow[d] & X \arrow[dd] \arrow[l] \arrow[r] & X_{2.32} \arrow[d] \\
        \mathbb{P}^2 \arrow[dr] & & \mathbb{P}^2 \arrow[dl] \arrow[equal,dr] & & \mathbb{P}^2 \arrow[equal,dl] \arrow[dr] & & \mathbb{P}^2 \arrow[dl] \arrow[equal,dr] & & \mathbb{P}^2 \arrow[equal,dl] \arrow[dr] & & \mathbb{P}^2 \arrow[dl] \arrow[equal,dr] & & \mathbb{P}^2 \arrow[equal,dl] \\
        & pt & &\mathbb{P}^2 & & pt & & \mathbb{P}^2 & & pt & & \mathbb{P}^2
    \end{tikzcd}
    }
    $$
\end{statement}

\begin{statement}[No. 3.14]\label{case 3.14}
    In this case $X$ is a blow-up of $\mathbb{P}^3$ in the disjoint union of a plane cubic curve $C$ and a point $P$ outside the plane containing $C$. A toric LG model of $X$ is given by
    $$
    f = x + y + z + xyz + \frac{1}{xy} + \frac{1}{yz} + \frac{1}{xz} + \frac{1}{xyz}.
    $$
    A parametrized LG model of $X$ is given by
    $$
    \widetilde{f} = a_{1}x + a_{1}y + a_{1}z + a_{3}xyz + \frac{a_{4}}{xy} + \frac{a_{4}}{yz} + \frac{a_{4}}{xz} + \frac{a_{2}}{xyz}
    $$
    with the relation $(a_{1},a_{2},a_{3},a_{4}) \sim (\lambda a_{1}, \lambda^{-3} a_{2}, \lambda^{3} a_{3}, \lambda^{-2} a_{4})$. Hence the associated elementary relation is
    $$
    \begin{tikzcd}[cramped, column sep=small]
        \mathbb{P}^3 \arrow[d] & \mathrm{Bl}_{P}\mathbb{P}^3 \arrow[dd] \arrow[l] \arrow[equal,r] & \mathrm{Bl}_{P}\mathbb{P}^3 \arrow[d] & X \arrow[dd] \arrow[l] \arrow[r] & X_{2.36} \arrow[d] & X_{2.36} \arrow[dd] \arrow[equal,l] \arrow[r] & \mathbb{P}(1,1,1,2) \arrow[d] & Y \arrow[dd] \arrow[l] \arrow[r] & Y' \arrow[d] & \mathrm{Bl}_{C}\mathbb{P}^3 \arrow[dd] \arrow[l] \arrow[r] & \mathbb{P}^3 \arrow[d] \\
        pt \arrow[equal,dr] & & \mathbb{P}^2 \arrow[dl] \arrow[equal,dr] & & \mathbb{P}^2 \arrow[equal,dl] \arrow[dr] & & pt \arrow[equal,dl] \arrow[equal,dr] & & pt \arrow[equal,dl] \arrow[equal,dr] &  & pt \arrow[equal,dl]  \\
        & pt & & \mathbb{P}^2 & & pt & & pt & & pt &
    \end{tikzcd}
    $$
    where $Y$ and $Y'$ have a quotient singularity of the form $\frac{1}{2}(1,1,1)$. A toric LG model of $Y$ is given by
    $$
    f_{Y} = x+y+z+xyz+\frac{1}{xy}+\frac{1}{yz}+\frac{1}{xz}
    $$
    and a toric LG model of $Y'$ is given by
    $$
    f_{Y'} = x+y+z+\frac{1}{xy}+\frac{1}{yz}+\frac{1}{xz}.
    $$
\end{statement}

\begin{statement}[No. 3.15]
    In this case $X$ is a blow-up of $Q$ in the disjoint union of a line $L$ and a conic $C$. The associated elementary relation is
    $$
    \begin{tikzcd}[cramped, column sep=small]
        Q \arrow[d] & \mathrm{Bl}_{L}Q \arrow[dd] \arrow[l] \arrow[equal,r] & \mathrm{Bl}_{L}Q \arrow[d] & X \arrow[dd] \arrow[l] \arrow[r] & \mathbb{P}^1 \times \mathbb{P}^2 \arrow[d] & \mathbb{P}^1 \times \mathbb{P}^2 \arrow[dd] \arrow[equal,l] \arrow[equal,r] & \mathbb{P}^1 \times \mathbb{P}^2 \arrow[d] & X \arrow[dd] \arrow[l] \arrow[r] & \mathrm{Bl}_{C}Q \arrow[d] & \mathrm{Bl}_{C}Q \arrow[dd] \arrow[equal,l] \arrow[r] & Q \arrow[d] \\
        pt \arrow[equal,dr] & & \mathbb{P}^2 \arrow[dl] \arrow[equal,dr] & & \mathbb{P}^2 \arrow[equal,dl] \arrow[dr] & & \mathbb{P}^1 \arrow[dl] \arrow[equal,dr] & & \mathbb{P}^1 \arrow[equal,dl] \arrow[dr] & & pt \arrow[equal,dl] \\
        & pt & & \mathbb{P}^2 & & pt & & \mathbb{P}^1 & & pt &
    \end{tikzcd}
    $$
\end{statement}

\begin{statement}[No. 3.16]
    In this case $X$ is obtained by first blowing up a point $P$ in $\mathbb{P}^3$, then blowing up the proper transformation $\widetilde{C}$ of a twisted cubic $C \subseteq \mathbb{P}^3$ containing $P$. The associated elementary relation is
    $$
    \begin{tikzcd}[cramped, column sep=small]
        \mathbb{P}^3 \arrow[d] & \mathrm{Bl}_{P} \mathbb{P}^3 \arrow[dd] \arrow[l] \arrow[equal,r] & \mathrm{Bl}_{P} \mathbb{P}^3 \arrow[d] & X \arrow[dd] \arrow[l] \arrow[r] & X_{2.32} \arrow[d] & X_{2.32} \arrow[dd] \arrow[equal,l] \arrow[equal,r] & X_{2.32} \arrow[d] & X \arrow[dd] \arrow[l] \arrow[r] & \mathrm{Bl}_{C}\mathbb{P}^3 \arrow[d] & \mathrm{Bl}_{C}\mathbb{P}^3 \arrow[dd] \arrow[equal,l] \arrow[r] & \mathbb{P}^3 \arrow[d] \\
        pt \arrow[equal,dr] & & \mathbb{P}^2 \arrow[dl] \arrow[equal,dr] & & \mathbb{P}^2 \arrow[equal,ld] \arrow[rd] & & \mathbb{P}^2 \arrow[ld] \arrow[equal,dr] & & \mathbb{P}^2 \arrow[equal,dl] \arrow[dr] & & pt \arrow[equal,dl] \\
        & pt & & \mathbb{P}^2 & & pt & & \mathbb{P}^2 & & pt &
    \end{tikzcd}
    $$  
\end{statement}

\begin{statement}[No. 3.17]
    In this case $X$ is a divisor on $\mathbb{P}^1 \times \mathbb{P}^1 \times \mathbb{P}^2$ of type $(1,1,1)$. The associated elementary relation is 

    $$
    \adjustbox{scale=0.9,center}{%
    \begin{tikzcd}[cramped,column sep=small]
        X \arrow[d] & X \arrow[equal,l] \arrow[dd] \arrow[r] & \mathbb{P}^1 \times \mathbb{P}^2 \arrow[d] & \mathbb{P}^1 \times \mathbb{P}^2 \arrow[dd] \arrow[equal,l] \arrow[equal,r] & \mathbb{P}^1 \times \mathbb{P}^2 \arrow[d] & X \arrow[dd] \arrow[l] \arrow[r] & \mathbb{P}^1 \times \mathbb{P}^2 \arrow[d] & \mathbb{P}^1 \times \mathbb{P}^2 \arrow[dd] \arrow[equal,l] \arrow[equal,r] & \mathbb{P}^1 \times \mathbb{P}^2 \arrow[d] & X \arrow[dd] \arrow[l] \arrow[equal,r] & X \arrow[d] \\
        \mathbb{P}^1 \times \mathbb{P}^1 \arrow[dr] & & \mathbb{P}^{1} \arrow[equal,dl] \arrow[dr] & & \mathbb{P}^2 \arrow[dl] \arrow[equal,dr] & & \mathbb{P}^2 \arrow[equal,dl] \arrow[dr] & & \mathbb{P}^1 \arrow[dl] \arrow[equal,dr] & & \mathbb{P}^1 \times \mathbb{P}^1 \arrow[dl] \\
        & \mathbb{P}^1 & & pt & & \mathbb{P}^2 & & pt & & \mathbb{P}^1
    \end{tikzcd}
    }
    $$
\end{statement}

\begin{statement}[No. 3.18]
    In this case $X$ is a blow-up of $\mathbb{P}^3$ in the disjoint union of a line $L$ and a conic $C$. The associated elementary relation is 
    $$
    \begin{tikzcd}
        \mathbb{P}^3 \arrow[d] & \mathrm{Bl}_{L}\mathbb{P}^3 \arrow[dd] \arrow[l] \arrow[equal,r] & \mathrm{Bl}_{L}\mathbb{P}^3 \arrow[d] & X \arrow[dd] \arrow[l] \arrow[r] & X_{2.29} \arrow[d] & X_{2.29} \arrow[dd] \arrow[equal,l] \arrow[r] & Q \arrow[d] & \mathrm{Bl}_{C}\mathbb{P}^3 \arrow[dd] \arrow[l] \arrow[r] & \mathbb{P}^3 \arrow[d] \\
        pt \arrow[equal,dr] & & \mathbb{P}^1 \arrow[dl] \arrow[equal,dr] & & \mathbb{P}^1 \arrow[equal,dl] \arrow[dr] 
        & & pt \arrow[equal,dl] \arrow[equal,dr] & & pt \arrow[equal,dl] \\
        & pt & & \mathbb{P}^1 & & pt & & pt &
    \end{tikzcd}
    $$
\end{statement}

\begin{statement}[No. 3.19]
    In this case $X$ is a blow-up of $Q$ in two non-collinear points $P_{1},P_{2}$. The associated elementary relation is
    $$
    \begin{tikzcd}[cramped,column sep=small]
        \mathbb{P}^3 \arrow[d] & \mathrm{Bl}_{P'_{1}}\mathbb{P}^3 \arrow[dd] \arrow[l] \arrow[equal,r] & \mathrm{Bl}_{P'_{1}}\mathbb{P}^3 \arrow[d] & X \arrow[dd] \arrow[l] \arrow[r] & \mathrm{Bl}_{P'_{2}}\mathbb{P}^3 \arrow[d] & \mathrm{Bl}_{P'_{2}}\mathbb{P}^3 \arrow[dd] \arrow[equal,l] \arrow[r] & \mathbb{P}^3 \arrow[d] & \mathrm{Bl}_{P_{1}} Q \arrow[dd] \arrow[l] \arrow[r] & Q \arrow[d] & \mathrm{Bl}_{P_{2}} Q \arrow[dd] \arrow[l] \arrow[r] & \mathbb{P}^3 \arrow[d] \\
        pt \arrow[equal,dr] & & \mathbb{P}^2 \arrow[dl] \arrow[equal,dr] & & \mathbb{P}^2 \arrow[equal,dl] \arrow[dr] & & pt \arrow[equal,dl] \arrow[equal,dr] & & pt \arrow[equal,dl] \arrow[equal,dr] & & pt \arrow[equal,dl] \\
         & pt & & \mathbb{P}^2 & & pt & & pt & & pt &
    \end{tikzcd}
    $$
\end{statement}

\begin{statement}[No. 3.20]
    In this case $X$ is a blow-up of $Q$ in the disjoint union of two lines $L_{1},L_{2}$. The associated elementary relation is
    $$
    \begin{tikzcd}[cramped, column sep=small]
        Q \arrow[d] & \mathrm{Bl}_{L_{1}}Q \arrow[dd] \arrow[l] \arrow[equal,r] & \mathrm{Bl}_{L_{1}}Q \arrow[d] & X \arrow[dd] \arrow[l] \arrow[r] & X_{2.32} \arrow[d] & X_{2.32} \arrow[dd] \arrow[equal,l] \arrow[equal,r] & X_{2.32} \arrow[d] & X \arrow[dd] \arrow[l] \arrow[r] & \mathrm{Bl}_{L_{2}}Q \arrow[d] & \mathrm{Bl}_{L_{2}}Q \arrow[dd] \arrow[equal,l] \arrow[r] & Q \arrow[d] \\
        pt \arrow[equal,dr] & & \mathbb{P}^2 \arrow[dl] \arrow[equal,dr] & & \mathbb{P}^2 \arrow[equal,dl] \arrow[dr] & & \mathbb{P}^2 \arrow[dl] \arrow[equal,dr] & & \mathbb{P}^2 \arrow[equal,dl] \arrow[dr] & & pt \arrow[equal,dl] \\
        & pt & & \mathbb{P}^2 & & pt & & \mathbb{P}^2 & & pt &
    \end{tikzcd}
    $$
\end{statement}

\begin{statement}[No. 3.21]
    In this case $X$ is a blow-up of $\mathbb{P}^1 \times \mathbb{P}^2$ in a curve $C$ of bidegree $(2,1)$. A toric LG model of $X$ is given by
    $$    
    f = x + y + z + \frac{1}{x} + \frac{1}{y} + \frac{1}{z} + \frac{1}{yz} + \frac{1}{xyz}.
    $$
    A parametrized toric LG model of $X$ is given by
    $$
    \widetilde{f} = a_{1}x + y + z + \frac{1}{x} + \frac{a_{3}}{y} + \frac{a_{3}}{z} + \frac{a_{2}}{yz} + \frac{a_{3}}{xyz}.
    $$
    The associated elementary relation is
    $$
    \begin{tikzcd}[cramped]
    \mathbb{P}^1 \times \mathbb{P}^2 \arrow[d] & \mathbb{P}^1 \times \mathbb{P}^2 \arrow[dd] \arrow[equal,l] \arrow[equal,r] & \mathbb{P}^1 \times \mathbb{P}^2 \arrow[d] & X \arrow[dd] \arrow[l] \arrow[r] & Y' \arrow[d] & Y \arrow[dd] \arrow[dashed,l] \arrow[dashed,r] & Y'' \arrow[d] & X \arrow[dd] \arrow[l] \arrow[r] & \mathbb{P}^1 \times \mathbb{P}^2 \arrow[d]  \\
    \mathbb{P}^1 \arrow[dr] & & \mathbb{P}^2 \arrow[dl] \arrow[equal,dr] & & \mathbb{P}^2 \arrow[equal,dl] \arrow[dr] & & \mathbb{P}^1 \arrow[dl] \arrow[equal,dr] & & \mathbb{P}^1 \arrow[equal,dl]  \\
    & pt & & \mathbb{P}^2 & & pt & & \mathbb{P}^1 &
    \end{tikzcd}
    $$
    where a general fibre of $Y'' \rightarrow \mathbb{P}^1$ is isomorphic to $\mathbb{P}^1 \times \mathbb{P}^1$. A toric LG model of $Y$ is given by
    $$
    f_{Y} = x + y + z + \frac{1}{x} + \frac{1}{y} + \frac{1}{z} + \frac{1}{xyz}.
    $$
    Hence $Y$ deforms to $B_{5}$. The birational map $Y' \dashrightarrow Y''$ is a flop over $Y$.
\end{statement}

\begin{statement}[No. 3.22]\label{case 3.22}
    In this case $X$ is a blow-up of $\mathbb{P}^1 \times \mathbb{P}^2$ in a conic $C_{x} \subseteq \{x\} \times \mathbb{P}^2$ for some $x \in \mathbb{P}^1$. A toric LG model of $X$ is given by
    $$
    f = x + y + z + \frac{z}{x} + \frac{1}{x} + \frac{1}{yz} + \frac{1}{xyz}.
    $$
    A parametrized LG model of $X$ is given by
    $$
    \widetilde{f} = x+y+a_{2}z+\frac{a_{3}z}{x} + \frac{a_{1}}{x}+\frac{a_{2}}{yz}+\frac{a_{3}}{xyz}.
    $$
    Hence the associated elementary relation is
    $$
    \begin{tikzcd}[cramped, column sep=small]
        \mathbb{P}^1 \times \mathbb{P}^2 \arrow[d] & X \arrow[dd] \arrow[l] \arrow[r] & X_{2.36} \arrow[d] & X_{2.36} \arrow[dd] \arrow[equal,l] \arrow[r] & \mathbb{P}(1,1,1,2) \arrow[d] & Y \arrow[dd] \arrow[l] \arrow[equal,r] & Y \arrow[d] & X \arrow[dd] \arrow[l] \arrow[r] & \mathbb{P}^1 \times \mathbb{P}^2 \arrow[d] & \mathbb{P}^1 \times \mathbb{P}^2 \arrow[dd] \arrow[equal,l] \arrow[equal,r] & \mathbb{P}^1 \times \mathbb{P}^2 \arrow[d] \\
        \mathbb{P}^2 \arrow[equal,dr] & & \mathbb{P}^2 \arrow[equal,dl] \arrow[dr] & & pt \arrow[equal,dl] \arrow[equal,dr] \arrow[equal,dr] & & \mathbb{P}^1 \arrow[dl] \arrow[equal,dr] & & \mathbb{P}^1 \arrow[equal,dl] \arrow[dr] & & \mathbb{P}^2 \arrow[dl] \\
         & \mathbb{P}^2 & & pt & & pt & & \mathbb{P}^1 & & pt &
    \end{tikzcd}
    $$
    where $Y$ has a singularity of type $\frac{1}{2}(1,1,1)$. A toric LG model of $Y$ is given by
    $$
    f_{Y} = x+y+z+\frac{z}{x}+\frac{1}{yz}+\frac{1}{xyz}.
    $$
\end{statement}

\begin{statement}[No. 3.23]
    In this case $X$ is obtained by first blowing up a point $P$ in $\mathbb{P}^3$, then blowing up the proper transformation $\widetilde{C}$ of a conic $C \subseteq \mathbb{P}^3$ containing $P$. The associated elementary relation is
    $$
    \begin{tikzcd}
        \mathbb{P}^3 \arrow[d] & \mathrm{Bl}_{P} \mathbb{P}^3 \arrow[dd] \arrow[l] \arrow[equal,r] & \mathrm{Bl}_{P} \mathbb{P}^3 \arrow[d] & X \arrow[dd] \arrow[l] \arrow[r] & X_{2.31} \arrow[d] & X_{2.31} \arrow[dd] \arrow[equal,l] \arrow[r] & Q \arrow[d] & X_{2.30} \arrow[dd] \arrow[l] \arrow[r] & \mathbb{P}^3 \arrow[d] \\
        pt \arrow[equal,dr] & & \mathbb{P}^2 \arrow[dl] \arrow[equal,dr] & & \mathbb{P}^2 \arrow[equal,dl] \arrow[dr] & & pt \arrow[equal,dl] \arrow[equal,dr] & & pt \arrow[equal,dl] \\
        & pt & & \mathbb{P}^2 & & pt & & pt &
    \end{tikzcd}
    $$
\end{statement}

\begin{statement}[No. 3.24]
    In this case $X$ is a fibre product
    $$
    \begin{tikzcd}
        X \arrow[d] \arrow[r] & X_{2.32} \arrow[d] \\
        \mathbb{F}_{1} \arrow[r] & \mathbb{P}^2
    \end{tikzcd}
    $$
    where $X_{2.32}$ is a Fano variety in family No. 2.32. The associated elementary relation is
    $$
    \begin{tikzcd}[cramped, column sep=small]
    X \arrow[d] & X \arrow[dd] \arrow[equal,l] \arrow[r] & X_{2.32} \arrow[d] & X_{2.32} \arrow[dd] \arrow[equal,l] \arrow[equal,r] & X_{2.32} \arrow[d] & X \arrow[dd] \arrow[l] \arrow[r] & \mathbb{P}^1 \times \mathbb{P}^2 \arrow[d] & \mathbb{P}^1 \times \mathbb{P}^2 \arrow[dd] \arrow[equal,l] \arrow[equal,r] & \mathbb{P}^1 \times \mathbb{P}^2 \arrow[d] & X \arrow[dd] \arrow[l] \arrow[equal,r] & X \arrow[d] \\
    \mathbb{F}_{1} \arrow[dr] & & \mathbb{P}^2 \arrow[equal,dl] \arrow[dr] & & \mathbb{P}^2 \arrow[dl] \arrow[equal,dr] & & \mathbb{P}^2 \arrow[equal,dl] \arrow[dr] & & \mathbb{P}^1 \arrow[dl] \arrow[equal,dr] & & \mathbb{F}_{1} \arrow[dl] \\
       & \mathbb{P}^2 & & pt & & \mathbb{P}^2 & & pt & & \mathbb{P}^1 &
    \end{tikzcd}
    $$
\end{statement}

\begin{statement}[No. 3.25]
    In this case $X$ is a blow-up of $\mathbb{P}^3$ in the disjoint union of two lines $L_{1}$ and $L_{2}$. The associated elementary relation is
    $$
    \begin{tikzcd}[cramped]
      \mathbb{P}^{3} \arrow[d] & \mathrm{Bl}_{L_{1}} \mathbb{P}^3 \arrow[dd] \arrow[l] \arrow[equal,r] & \mathrm{Bl}_{L_{1}} \mathbb{P}^3 \arrow[d] & X \arrow[dd] \arrow[l] \arrow[equal,r] & X \arrow[d] & X \arrow[dd] \arrow[equal,l] \arrow[r] & \mathrm{Bl}_{L_{2}} \mathbb{P}^3 \arrow[d] & \mathrm{Bl}_{L_{2}} \mathbb{P}^3 \arrow[dd] \arrow[equal,l] \arrow[r] & \mathbb{P}^3 \arrow[d] \\
      pt \arrow[equal,dr] & & \mathbb{P}^1 \arrow[dl] \arrow[equal,dr] & & \mathbb{P}^1 \times \mathbb{P}^1 \arrow[dl] \arrow[dr] & & \mathbb{P}^{1} \arrow[equal,dl] \arrow[dr] & & pt \arrow[equal,dl] \\
       & pt & & \mathbb{P}^1 & & \mathbb{P}^1 & & pt & \\
    \end{tikzcd}
    $$
\end{statement}

\begin{statement}[No. 3.26]
    In this case $X$ is a blow-up of $\mathbb{P}^3$ in the disjoint union of a point $P$ and a line $L$. A toric LG model of $X$ is given by
    $$
    f = x + y + z + xy + \frac{1}{x} + \frac{1}{xyz}.
    $$
    A parametrized toric LG model of $X$ is given by
    $$
    \widetilde{f} = x + a_{1}y + z + a_{2}xy + \frac{a_{3}}{x} + \frac{1}{xyz}.
    $$
    Hence the associated elementary relation is
    $$
    \begin{tikzcd}[cramped,column sep=small]
    \mathrm{Bl}_{P}\mathbb{P}^3 \arrow[d] & X \arrow[dd] \arrow[l] \arrow[r] & \mathbb{P}^2 \times \mathbb{P}^1 \arrow[d] & \mathbb{P}^2 \times \mathbb{P}^1 \arrow[dd] \arrow[equal,l] \arrow[equal,r] & \mathbb{P}^2 \times \mathbb{P}^1 \arrow[d] & X \arrow[dd] \arrow[l] \arrow[r] & \mathrm{Bl}_{L}\mathbb{P}^3 \arrow[d] \arrow[equal,r] & \mathrm{Bl}_{L}\mathbb{P}^3 \arrow[dd]\arrow[r] & \mathbb{P}^3 \arrow[d] & \mathrm{Bl}_{P}\mathbb{P}^3 \arrow[dd] \arrow[l] \arrow[equal,r] & \mathrm{Bl}_{P}\mathbb{P}^3 \arrow[d] \\
    \mathbb{P}^2 \arrow[equal,dr] & & \mathbb{P}^2 \arrow[equal,dl] \arrow[dr] & & \mathbb{P}^1 \arrow[dl] \arrow[equal,dr] & & \mathbb{P}^1 \arrow[equal,dl] \arrow[dr] & & pt \arrow[equal,dl] \arrow[equal,dr] & & \mathbb{P}^2 \arrow[dl] \\
     &\mathbb{P}^2 & & pt & & \mathbb{P}^1 & & pt & & pt &
    \end{tikzcd}    
    $$
\end{statement}

\begin{statement}[No. 3.27]
    In this case $X$ is isomorphic to $\mathbb{P}^1 \times \mathbb{P}^1 \times \mathbb{P}^1$. The associated elementary relation is
    $$
    \begin{tikzcd}
    X \arrow[d,"p_{3} \times p_{1}"] \arrow[equal,r] & X \arrow[dd,"p_{1}"] \arrow[equal,r] & X \arrow[d,"p_{1} \times p_{2}"] \arrow[equal,r] & X \arrow[dd,"p_{2}"] \arrow[equal,r] & X \arrow[d,"p_{2} \times p_{3}"] \arrow[equal,r] & X \arrow[dd,"p_{3}"] \arrow[equal,r] & X \arrow[d,"p_{3} \times p_{1}"] \\    
    \mathbb{P}^1 \times \mathbb{P}^1 \arrow[dr] & & \mathbb{P}^1 \times \mathbb{P}^1 \arrow[dl] \arrow[dr] & & \mathbb{P}^1 \times \mathbb{P}^1 \arrow[dl] \arrow[dr] & & \mathbb{P}^1 \times \mathbb{P}^1 \arrow[dl] \\    
    & \mathbb{P}^1 & & \mathbb{P}^1 & & \mathbb{P}^1 &
    \end{tikzcd}
    $$
\end{statement}

\begin{statement}[No. 3.28]
    In this case $X$ is isomorphic to $\mathrm{Bl}_{p}\mathbb{P}^2 \times \mathbb{P}^1$. The associated elementary relation is
    $$
    \begin{tikzcd}[cramped, column sep=small]
    \mathbb{P}^2 \times \mathbb{P}^1 \arrow[d,"p_{1}"] \arrow[equal,r] & \mathbb{P}^2 \times \mathbb{P}^1 \arrow[dd] \arrow[equal,r] & \mathbb{P}^2 \times \mathbb{P}^1 \arrow[d,"p_{2}"] & X \arrow[dd,"p_{2}"] \arrow[equal,r] \arrow[l] & X \arrow[d,"\pi \circ p_{1} \times p_{2}"] \arrow[equal,r] & X \arrow[dd,"\pi \circ p_{1}"] \arrow[equal,r] & X \arrow[d,"p_{1}"] \arrow[equal,r] & X \arrow[dd]\arrow[r] & \mathbb{P}^2 \times \mathbb{P}^1 \arrow[d,"p_{1}"] \\
    \mathbb{P}^2 \arrow[dr] & & \mathbb{P}^1 \arrow[dl] \arrow[dr] & & \mathbb{P}^1 \times \mathbb{P}^1 \arrow[dl] \arrow[dr] & & \mathrm{Bl}_{p}\mathbb{P}^2 \arrow[dl]\arrow[dr] & & \mathbb{P}^2 \arrow[dl] \\    
    & pt & & \mathbb{P}^1 & & \mathbb{P}^1 & & \mathbb{P}^2 &
    \end{tikzcd}
    $$
    where $\pi :\mathrm{Bl}_{p}\mathbb{P}^2 \rightarrow \mathbb{P}^1$ is the fibration of the Hirzebruch surface. A toric LG model of $X$ is given by
    $$
    f = x + y + z + \frac{x}{z} + \frac{1}{x} + \frac{1}{y}.
    $$
\end{statement}

\begin{statement}[No. 3.29]
    In this case $X$ is obtained by first blowing up a point $P$ in $\mathbb{P}^3$, then blowing up a line $L$ in the exceptional divisor $E \cong \mathbb{P}^2$. A toric LG model of $X$ is given by
    $$
    f = x + y + z + \frac{y}{x} + \frac{1}{x} + \frac{1}{xyz}.
    $$
    A parametrized toric LG model of $X$ is given by
    $$
    \widetilde{f} = x + a_{1}y + z + \frac{a_{3}y}{x} + \frac{a_{2}}{x} + \frac{1}{xyz}.
    $$
    Hence the associated elementary relation is
    $$
    \begin{tikzcd}[cramped]
        \mathbb{P}^3 \arrow[d] & \mathrm{Bl}_{P} \mathbb{P}^3 \arrow[dd] \arrow[l] \arrow[equal,r] & \mathrm{Bl}_{P} \mathbb{P}^3 \arrow[d] & X \arrow[dd] \arrow[l] \arrow[r] & X_{2.36} \arrow[d] & X_{2.36} \arrow[dd] \arrow[equal,l] \arrow[r] & \mathbb{P}(1,1,1,2) \arrow[d] & Y \arrow[dd] \arrow[l] \arrow[r] & \mathbb{P}^3 \arrow[d] \\
        pt \arrow[equal,dr] & & \mathbb{P}^2 \arrow[dl] \arrow[equal,dr] & & \mathbb{P}^2 \arrow[equal,dl] \arrow[dr] & & pt \arrow[equal,dl] \arrow[equal,dr] & & pt \arrow[equal,dl] \\
        & pt & & \mathbb{P}^2 & & pt & & pt &
    \end{tikzcd}
    $$
    where $Y \rightarrow \mathbb{P}(1,1,1,2)$ is a blow-up of a line outside the singular point.
\end{statement}

\begin{statement}[No. 3.30]
    In this case $X$ is obtained by first blowing up a point $P$ in $\mathbb{P}^3$, then blowing up the proper transformation $\widetilde{L}$ of a line $L \subseteq \mathbb{P}^3$ containing $P$. The associated elementary relation is
    $$
    \begin{tikzcd}
        \mathbb{P}^3 \arrow[d] & \mathrm{Bl}_{P} \mathbb{P}^3 \arrow[dd] \arrow[l] \arrow[equal,r] & \mathrm{Bl}_{P} \mathbb{P}^3 \arrow[d] & X \arrow[dd] \arrow[l] \arrow[equal,r] & X \arrow[d] & X \arrow[dd] \arrow[equal,l] \arrow[r] & \mathrm{Bl}_{L'}\mathbb{P}^{3} \arrow[d] & \mathrm{Bl}_{L'}\mathbb{P}^{3} \arrow[dd] \arrow[equal,l] \arrow[r] & \mathbb{P}^3 \arrow[d] \\
        pt \arrow[equal,dr] & & \mathbb{P}^2 \arrow[dl] \arrow[equal,dr] & & \mathbb{F}_{1} \arrow[dl] \arrow[dr] & & \mathbb{P}^1 \arrow[equal,dl] \arrow[dr] & & pt \arrow[equal,dl] \\
        & pt & & \mathbb{P}^2 & & \mathbb{P}^{1} & & pt &
    \end{tikzcd}
    $$
\end{statement}

\begin{statement}[No. 3.31]
    In this case $X$ is a blow-up of a cone $Y$ over smooth quadrics in $\mathbb{P}^3$ in the vertex. A toric LG model of $X$ is given by
    $$
    f = x + y + z + \frac{1}{x} + \frac{1}{xy} + \frac{1}{xz}.
    $$
    A parametrized toric LG model of $X$ is given by
    $$
    \widetilde{f} = x + y + z + \frac{a_{1}}{x} + \frac{a_{2}}{xy} + \frac{a_{3}}{xz}.
    $$
    The associated elementary relation is
    $$
    \begin{tikzcd}
        X \arrow[d] & X \arrow[dd]\arrow[equal,l] \arrow[r] & Y_{1} \arrow[d] & Y \arrow[dd] \arrow[dashed,l] \arrow[dashed,r] & Y_{2} \arrow[d] & X \arrow[dd] \arrow[l] \arrow[equal,r] & X \arrow[d] \\
        \mathbb{P}^1 \times \mathbb{P}^1 \arrow[dr] & & \mathbb{P}^1 \arrow[equal,dl] \arrow[dr] & & \mathbb{P}^1 \arrow[dl] \arrow[equal,dr] & & \mathbb{P}^1 \times \mathbb{P}^1 \arrow[dl] \\
        & \mathbb{P}^1 & & pt & & \mathbb{P}^1 &
    \end{tikzcd}
    $$
    where $Y_{1}$ and $Y_{2}$ are two small resolutions of $Y$, and the natural birational map $Y_{1} \dashrightarrow Y_{2}$ is a flop over $Y$. General fibres of $Y_{1} \rightarrow \mathbb{P}^1$ and $Y_{2} \rightarrow \mathbb{P}^1$ are isomorphic to $\mathbb{P}^2$. A toric LG model of $Y$ is given by
    $$
    f_{Y} = x + y + z + \frac{1}{xy} + \frac{1}{xz}.
    $$
    Hence $Y$ deforms to $Q$.
\end{statement}

\subsection{Picard number 4}
    Now we study smooth Fano threefolds of Picard number 4. The associated elementary syzygy is a polyhedral decomposition of the sphere $S^{2}$. In this subsection, we will present such polyhedral structure by projecting the sphere $S^{2}$ to the plane with infinity. We list some cases where we encounter a terminal but not smooth Fano variety.

\begin{example}[No. 4.2]\label{case 4.2}
    In this case, $X$ is a blow-up of the cone over a smooth quadric $S$ in $\mathbb{P}^3$ in the disjoint union of the vertex and an elliptic curve $C$ on $S$. Denote by $g:X \rightarrow X_{3.31}$ the blow-up of $C$ and $(p_{1},p_{2}): X_{3.31} \rightarrow \mathbb{P}^1 \times \mathbb{P}^1$ the natural projection. Let $F_{1}$ and $F_{2}$ be the fibre class of $p_{1}$ and $p_{2}$ respectively. The effective cone $\mathrm{Eff}_{\mathbb{R}}(X)$ is generated by 6 extremal rays:
    \begin{enumerate}
        \item The exceptional divisor $E_{1}$ centered at the vertex;
        \item The exceptional divisor $E_{2}$ centered at $C$;
        \item The strict transformation of $S \cong \mathbb{P}^1 \times \mathbb{P}^1$, whose divisor class is given by $g^{*}F_{1} + g^{*}F_{2} + E_{1} - E_{2}$;
        \item The strict transformation of the cone over $pt \times \mathbb{P}^1 \subseteq \mathbb{P}^1 \times \mathbb{P}^1 \cong S$, whose divisor class is given by $g^{*}F_{1}$;
        \item The strict transformation of the cone over $\mathbb{P}^1 \times pt \subseteq \mathbb{P}^1 \times \mathbb{P}^1 \cong S$, whose divisor class is given by $g^{*}F_{2}$;
        \item The strict transformation of the cone over $C$, whose divisor class is given by $2g^{*}F_{1} + 2g^{*}F_{2} - E_{2}$.
    \end{enumerate}
    Let $(v_{4},v_{1},v_{3},v_{5},v_{6},v_{2})$ be the corresponding prime vectors. A toric LG model of $X$ is given by
    $$
    f = x + y + \frac{x}{y} + \frac{y}{x} + \frac{z}{x} + \frac{z}{y} + \frac{2}{x} + \frac{2}{y} + \frac{1}{xz} + \frac{1}{yz}. 
    $$
    A parametrized toric LG model of $X$ is given by
    $$
    \widetilde{f} = a_{1}x + a_{2}y + \frac{a_{3}x}{y} + \frac{a_{4}y}{x} + \frac{a_{2}a_{5}z}{x} + \frac{a_{1}a_{5}z}{y} + \frac{a_{2}(a_{5}+a_{6})}{x} + \frac{a_{1}(a_{5}+a_{6})}{y} + \frac{a_{2}a_{6}}{xz} + \frac{a_{1}a_{6}}{yz}
    $$
    with the equivalence $(a_{1},a_{2},a_{3},a_{4},a_{5},a_{6}) \sim (\lambda a_{1},a_{2},\lambda a_{3},\lambda^{-1} a_{4},\lambda^{-1} a_{5},\lambda^{-1} a_{6})$ and $(a_{1},a_{2},a_{3},a_{4},a_{5},a_{6}) \sim ( a_{1},\lambda a_{2},\lambda^{-1} a_{3},\lambda a_{4},\lambda^{-1} a_{5},\lambda^{-1} a_{6})$. Hence the associated elementary syzygy is
    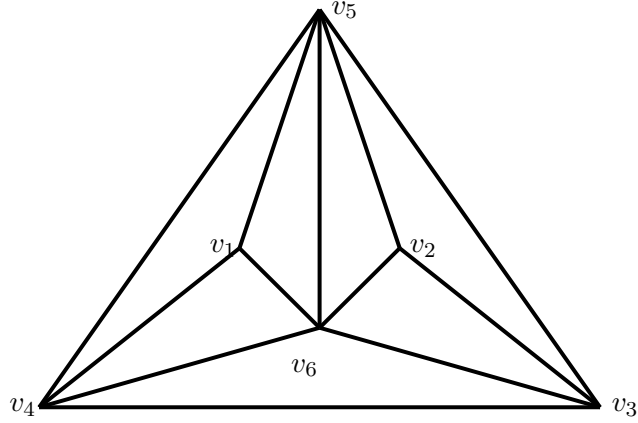
\begin{figure}[H]
        \centering
        \ifx\du\undefined
          \newlength{\du}
        \fi
        \setlength{\du}{15\unitlength}
        \begin{tikzpicture}
        \pgftransformxscale{1.000000}
        \pgftransformyscale{-1.000000}
        \definecolor{dialinecolor}{rgb}{0.000000, 0.000000, 0.000000}
        \pgfsetstrokecolor{dialinecolor}
        \definecolor{dialinecolor}{rgb}{1.000000, 1.000000, 1.000000}
        \pgfsetfillcolor{dialinecolor}
        \definecolor{dialinecolor}{rgb}{0.000000, 0.000000, 0.000000}
        \pgfsetstrokecolor{dialinecolor}
        \node[anchor=west] at (27.000000\du,14.000000\du){$v_{1}$};
        \definecolor{dialinecolor}{rgb}{0.000000, 0.000000, 0.000000}
        \pgfsetstrokecolor{dialinecolor}
        \node[anchor=west] at (32.000000\du,14.000000\du){$v_{2}$};
        \definecolor{dialinecolor}{rgb}{0.000000, 0.000000, 0.000000}
        \pgfsetstrokecolor{dialinecolor}
        \node[anchor=west] at (22.000000\du,18.000000\du){$v_{4}$};
        \pgfsetlinewidth{0.100000\du}
        \pgfsetdash{}{0pt}
        \pgfsetdash{}{0pt}
        \pgfsetbuttcap
        {
        \definecolor{dialinecolor}{rgb}{0.000000, 0.000000, 0.000000}
        \pgfsetfillcolor{dialinecolor}
        \definecolor{dialinecolor}{rgb}{0.000000, 0.000000, 0.000000}
        \pgfsetstrokecolor{dialinecolor}
        \draw (28.000000\du,14.000000\du)--(23.000000\du,18.000000\du);
        }
        \definecolor{dialinecolor}{rgb}{0.000000, 0.000000, 0.000000}
        \pgfsetstrokecolor{dialinecolor}
        \node[anchor=west] at (30.000000\du,8.000000\du){$v_{5}$};
        \definecolor{dialinecolor}{rgb}{0.000000, 0.000000, 0.000000}
        \pgfsetstrokecolor{dialinecolor}
        \node[anchor=west] at (37.000000\du,18.000000\du){$v_{3}$};
        \pgfsetlinewidth{0.100000\du}
        \pgfsetdash{}{0pt}
        \pgfsetdash{}{0pt}
        \pgfsetbuttcap
        {
        \definecolor{dialinecolor}{rgb}{0.000000, 0.000000, 0.000000}
        \pgfsetfillcolor{dialinecolor}
        \definecolor{dialinecolor}{rgb}{0.000000, 0.000000, 0.000000}
        \pgfsetstrokecolor{dialinecolor}
        \draw (32.000000\du,14.000000\du)--(37.000000\du,18.000000\du);
        }
        \pgfsetlinewidth{0.100000\du}
        \pgfsetdash{}{0pt}
        \pgfsetdash{}{0pt}
        \pgfsetbuttcap
        {
        \definecolor{dialinecolor}{rgb}{0.000000, 0.000000, 0.000000}
        \pgfsetfillcolor{dialinecolor}
        \definecolor{dialinecolor}{rgb}{0.000000, 0.000000, 0.000000}
        \pgfsetstrokecolor{dialinecolor}
        \draw (23.000000\du,18.000000\du)--(37.000000\du,18.000000\du);
        }
        \pgfsetlinewidth{0.100000\du}
        \pgfsetdash{}{0pt}
        \pgfsetdash{}{0pt}
        \pgfsetbuttcap
        {
        \definecolor{dialinecolor}{rgb}{0.000000, 0.000000, 0.000000}
        \pgfsetfillcolor{dialinecolor}
        \definecolor{dialinecolor}{rgb}{0.000000, 0.000000, 0.000000}
        \pgfsetstrokecolor{dialinecolor}
        \draw (28.000000\du,14.000000\du)--(30.000000\du,8.000000\du);
        }
        \pgfsetlinewidth{0.100000\du}
        \pgfsetdash{}{0pt}
        \pgfsetdash{}{0pt}
        \pgfsetbuttcap
        {
        \definecolor{dialinecolor}{rgb}{0.000000, 0.000000, 0.000000}
        \pgfsetfillcolor{dialinecolor}
        \definecolor{dialinecolor}{rgb}{0.000000, 0.000000, 0.000000}
        \pgfsetstrokecolor{dialinecolor}
        \draw (30.000000\du,8.000000\du)--(32.000000\du,14.000000\du);
        }
        \pgfsetlinewidth{0.100000\du}
        \pgfsetdash{}{0pt}
        \pgfsetdash{}{0pt}
        \pgfsetbuttcap
        {
        \definecolor{dialinecolor}{rgb}{0.000000, 0.000000, 0.000000}
        \pgfsetfillcolor{dialinecolor}
        \definecolor{dialinecolor}{rgb}{0.000000, 0.000000, 0.000000}
        \pgfsetstrokecolor{dialinecolor}
        \draw (30.000000\du,16.000000\du)--(28.000000\du,14.000000\du);
        }
        \pgfsetlinewidth{0.100000\du}
        \pgfsetdash{}{0pt}
        \pgfsetdash{}{0pt}
        \pgfsetbuttcap
        {
        \definecolor{dialinecolor}{rgb}{0.000000, 0.000000, 0.000000}
        \pgfsetfillcolor{dialinecolor}
        \definecolor{dialinecolor}{rgb}{0.000000, 0.000000, 0.000000}
        \pgfsetstrokecolor{dialinecolor}
        \draw (30.000000\du,16.000000\du)--(32.000000\du,14.000000\du);
        }
        \definecolor{dialinecolor}{rgb}{0.000000, 0.000000, 0.000000}
        \pgfsetstrokecolor{dialinecolor}
        \node[anchor=west] at (29.000000\du,17.000000\du){$v_{6}$};
        \pgfsetlinewidth{0.100000\du}
        \pgfsetdash{}{0pt}
        \pgfsetdash{}{0pt}
        \pgfsetbuttcap
        {
        \definecolor{dialinecolor}{rgb}{0.000000, 0.000000, 0.000000}
        \pgfsetfillcolor{dialinecolor}
        \definecolor{dialinecolor}{rgb}{0.000000, 0.000000, 0.000000}
        \pgfsetstrokecolor{dialinecolor}
        \draw (30.000000\du,16.000000\du)--(23.000000\du,18.000000\du);
        }
        \pgfsetlinewidth{0.100000\du}
        \pgfsetdash{}{0pt}
        \pgfsetdash{}{0pt}
        \pgfsetbuttcap
        {
        \definecolor{dialinecolor}{rgb}{0.000000, 0.000000, 0.000000}
        \pgfsetfillcolor{dialinecolor}
        \definecolor{dialinecolor}{rgb}{0.000000, 0.000000, 0.000000}
        \pgfsetstrokecolor{dialinecolor}
        \draw (37.000000\du,18.000000\du)--(30.000000\du,16.000000\du);
        }
        \pgfsetlinewidth{0.100000\du}
        \pgfsetdash{}{0pt}
        \pgfsetdash{}{0pt}
        \pgfsetbuttcap
        {
        \definecolor{dialinecolor}{rgb}{0.000000, 0.000000, 0.000000}
        \pgfsetfillcolor{dialinecolor}
        \definecolor{dialinecolor}{rgb}{0.000000, 0.000000, 0.000000}
        \pgfsetstrokecolor{dialinecolor}
        \draw (30.000000\du,8.000000\du)--(23.000000\du,18.000000\du);
        }
        \pgfsetlinewidth{0.100000\du}
        \pgfsetdash{}{0pt}
        \pgfsetdash{}{0pt}
        \pgfsetbuttcap
        {
        \definecolor{dialinecolor}{rgb}{0.000000, 0.000000, 0.000000}
        \pgfsetfillcolor{dialinecolor}
        \definecolor{dialinecolor}{rgb}{0.000000, 0.000000, 0.000000}
        \pgfsetstrokecolor{dialinecolor}
        \draw (30.000000\du,8.000000\du)--(37.000000\du,18.000000\du);
        }
        \pgfsetlinewidth{0.100000\du}
        \pgfsetdash{}{0pt}
        \pgfsetdash{}{0pt}
        \pgfsetbuttcap
        {
        \definecolor{dialinecolor}{rgb}{0.000000, 0.000000, 0.000000}
        \pgfsetfillcolor{dialinecolor}
        \definecolor{dialinecolor}{rgb}{0.000000, 0.000000, 0.000000}
        \pgfsetstrokecolor{dialinecolor}
        \draw (30.000000\du,8.000000\du)--(30.000000\du,16.000000\du);
        }
        \end{tikzpicture}
        \caption{The elementary syzygy of Mori-Mukai No. 4.2}
        \label{figure:4.2}
    \end{figure}    
    
    The correspondence of vertices and lines is given by:
    \begin{itemize}
        \item $v_{1}$: The smooth Fano variety $X_{3.31}$;
        \item $v_{2}$: The other smooth Fano variety $X'_{3.31}$;
        \item $v_{3}$: A Fano variety $Y_{1}$ with an ordinary double point, which deforms to $X_{2.23}$;
        \item $v_{4}$: A Fano variety $Y_{2}$ with an ordinary double point, which deforms to $X_{2.23}$;
        \item $v_{5}$: A Fano fibration $X \rightarrow \mathbb{P}^1$ whose general fibres are smooth del Pezzo surfaces of degree 6;
        \item $v_{6}$: A Fano fibration $X \rightarrow \mathbb{P}^1$ whose general fibres are smooth del Pezzo surfaces of degree 6;
        \item $\langle v_{1},v_{4} \rangle$: A Fano variety $Y_{3}$ with an ordinary double point, which deforms to $Q$;
        \item $\langle v_{1},v_{5} \rangle$: A Fano fibration $X_{3.31} \rightarrow \mathbb{P}^1$ whose general fibres are isomorphic to $\mathbb{F}_{1}$;
        \item $\langle v_{1},v_{6} \rangle$: Another Fano fibration $X_{3.31} \rightarrow \mathbb{P}^1$ whose general fibres are isomorphic to $\mathbb{F}_{1}$;
        \item $\langle v_{2},v_{3} \rangle$: A Fano variety $Y_{4}$ with an ordinary double point, which deforms to $Q$;
        \item $\langle v_{2},v_{5} \rangle$: A Fano fibration $X'_{3.31} \rightarrow \mathbb{P}^1$ whose general fibres are isomorphic to $\mathbb{F}_{1}$;
        \item $\langle v_{2},v_{6} \rangle$: Another Fano fibration $X'_{3.31} \rightarrow \mathbb{P}^1$ whose general fibres are isomorphic to $\mathbb{F}_{1}$;
        \item $\langle v_{3},v_{4} \rangle$: A Fano variety $Y_{5}$ with two ordinary double points, which deforms to $B_{4}$;
        \item $\langle v_{3},v_{5} \rangle$: A Fano fibration $\widetilde{Y_{1}} \rightarrow \mathbb{P}^1$ whose general fibres are del Pezzo surfaces of degree 7, where $\widetilde{Y_{1}}$ is a $\mathbb{Q}$-factorization of $Y_{1}$;
        \item $\langle v_{3},v_{6} \rangle$: A Fano fibration $\widetilde{Y_{1}}' \rightarrow \mathbb{P}^1$ whose general fibres are del Pezzo surfaces of degree 7, where $\widetilde{Y_{1}}'$ is a $\mathbb{Q}$-factorization of $Y_{1}$;
        \item $\langle v_{4},v_{5} \rangle$: A Fano fibration $\widetilde{Y_{2}} \rightarrow \mathbb{P}^1$ whose general fibres are del Pezzo surfaces of degree 7, where $\widetilde{Y_{2}}$ is a $\mathbb{Q}$-factorization of $Y_{2}$;
        \item $\langle v_{4},v_{6} \rangle$: A Fano fibration $\widetilde{Y_{2}}' \rightarrow \mathbb{P}^1$ whose general fibres are del Pezzo surfaces of degree 7, where $\widetilde{Y_{2}}'$ is a $\mathbb{Q}$-factorization of $Y_{2}$;
        \item $\langle v_{5},v_{6} \rangle$: The Fano fibration $X \rightarrow \mathbb{P}^1 \times \mathbb{P}^1$.
    \end{itemize}
    The correspondence of faces follows from the above correspondence and the elementary syzygies of smooth Fano threefolds of Picard rank $\leq 3$.
\end{example}

\begin{example}[No. 4.5]
    In this case, $X$ is a blow-up of $\mathbb{P}^1 \times \mathbb{P}^2$ in the disjoint union of a curve $C_{1}$ of degree $(2,1)$ and a curve $C_{2}$ of degree $(1,0)$. Denote by $g: X \rightarrow \mathbb{P}^1 \times \mathbb{P}^2$ the blow-up morphism, and denote by $F_{1}$ and $F_{2}$ the divisor class of the fibre of $p_{1}$ and $p_{2}$ respectively. The effective cone $\mathrm{Eff}_{\mathbb{R}}(X)$ is generated by 5 extremal rays:
    \begin{enumerate}
        \item The exceptional divisor $E_{1}$ centered at $C_{1}$;
        \item The exceptional divisor $E_{2}$ centered at $C_{2}$;
        \item The strict transformation of a general fibre of $p_{1}$, whose divisor class is given by $g^{*}F_{1}$;
        \item The strict transformation of the divisor given by connecting $F \cap C_{1}$ and $F  \cap C_{2}$ on each fibre $F$ of $p_{1}$, whose divisor class is given by $g^{*}F_{1} + 2g^{*}F_{2}-E_{1}-2E_{2}$;
        \item The strict transformation of a plane $\mathbb{P}^1 \times H$ which contains $C_{1}$, whose divisor class is given by $g^{*}F_{2} - E_{1}$.
    \end{enumerate}
    A toric LG model of $X$ is given by
    $$
    f = x + y + z + \frac{y}{x} + \frac{y}{z} + \frac{z}{y} + \frac{1}{x} + \frac{1}{y} + \frac{1}{z}.
    $$
    Taking a change of variables $y \mapsto \frac{1}{yz}, z \mapsto \frac{1}{z}$, we can rewrite the Laurent polynomial as
    $$
    f = x + y + z + yz + \frac{1}{x} + \frac{1}{y} + \frac{1}{z} + \frac{1}{yz} + \frac{1}{xyz}.
    $$
    A parametrized toric LG model of $X$ is given by
    $$
    \widetilde{f} = a_{2}x + a_{3}y + a_{3}z + a_{1}yz + \frac{a_{3}}{x} + \frac{a_{5}}{y} + \frac{a_{5}}{z} + \frac{a_{4}}{yz} + \frac{a_{5}}{xyz}
    $$
    with the equivalence $(a_{1},a_{2},a_{3},a_{4},a_{5}) \sim (\lambda^{2} a_{1}, \lambda^{-1} a_{2}, \lambda a_{3}, \lambda^{-2} a_{4}, \lambda^{-1} a_{5})$.
    Hence the associated elementary syzygy is
    \begin{figure}[H]
        \centering
        \ifx\du\undefined
          \newlength{\du}
        \fi
        \setlength{\du}{15\unitlength}
        \begin{tikzpicture}
        \pgftransformxscale{1.000000}
        \pgftransformyscale{-1.000000}
        \definecolor{dialinecolor}{rgb}{0.000000, 0.000000, 0.000000}
        \pgfsetstrokecolor{dialinecolor}
        \definecolor{dialinecolor}{rgb}{1.000000, 1.000000, 1.000000}
        \pgfsetfillcolor{dialinecolor}
        \definecolor{dialinecolor}{rgb}{0.000000, 0.000000, 0.000000}
        \pgfsetstrokecolor{dialinecolor}
        \node[anchor=west] at (25.000000\du,10.000000\du){$v_{1}$};
        \definecolor{dialinecolor}{rgb}{0.000000, 0.000000, 0.000000}
        \pgfsetstrokecolor{dialinecolor}
        \node[anchor=west] at (35.000000\du,20.000000\du){$v_{2}$};
        \definecolor{dialinecolor}{rgb}{0.000000, 0.000000, 0.000000}
        \pgfsetstrokecolor{dialinecolor}
        \node[anchor=west] at (19.000000\du,20.000000\du){$v_{4}$};
        \pgfsetlinewidth{0.100000\du}
        \pgfsetdash{}{0pt}
        \pgfsetdash{}{0pt}
        \pgfsetbuttcap
        {
        \definecolor{dialinecolor}{rgb}{0.000000, 0.000000, 0.000000}
        \pgfsetfillcolor{dialinecolor}
        \definecolor{dialinecolor}{rgb}{0.000000, 0.000000, 0.000000}
        \pgfsetstrokecolor{dialinecolor}
        \draw (25.050000\du,10.050000\du)--(20.000000\du,20.000000\du);
        }
        \pgfsetlinewidth{0.100000\du}
        \pgfsetdash{}{0pt}
        \pgfsetdash{}{0pt}
        \pgfsetbuttcap
        {
        \definecolor{dialinecolor}{rgb}{0.000000, 0.000000, 0.000000}
        \pgfsetfillcolor{dialinecolor}
        \definecolor{dialinecolor}{rgb}{0.000000, 0.000000, 0.000000}
        \pgfsetstrokecolor{dialinecolor}
        \draw (25.000000\du,10.000000\du)--(35.000000\du,20.000000\du);
        }
        \pgfsetlinewidth{0.100000\du}
        \pgfsetdash{}{0pt}
        \pgfsetdash{}{0pt}
        \pgfsetbuttcap
        {
        \definecolor{dialinecolor}{rgb}{0.000000, 0.000000, 0.000000}
        \pgfsetfillcolor{dialinecolor}
        \definecolor{dialinecolor}{rgb}{0.000000, 0.000000, 0.000000}
        \pgfsetstrokecolor{dialinecolor}
        \draw (20.000000\du,20.000000\du)--(35.000000\du,20.000000\du);
        }
        \pgfsetlinewidth{0.100000\du}
        \pgfsetdash{}{0pt}
        \pgfsetdash{}{0pt}
        \pgfsetbuttcap
        {
        \definecolor{dialinecolor}{rgb}{0.000000, 0.000000, 0.000000}
        \pgfsetfillcolor{dialinecolor}
        \definecolor{dialinecolor}{rgb}{0.000000, 0.000000, 0.000000}
        \pgfsetstrokecolor{dialinecolor}
        \draw (25.000000\du,10.000000\du)--(25.000000\du,16.000000\du);
        }
        \pgfsetlinewidth{0.100000\du}
        \pgfsetdash{}{0pt}
        \pgfsetdash{}{0pt}
        \pgfsetbuttcap
        {
        \definecolor{dialinecolor}{rgb}{0.000000, 0.000000, 0.000000}
        \pgfsetfillcolor{dialinecolor}
        \definecolor{dialinecolor}{rgb}{0.000000, 0.000000, 0.000000}
        \pgfsetstrokecolor{dialinecolor}
        \draw (25.000000\du,16.000000\du)--(20.000000\du,20.000000\du);
        }
        \pgfsetlinewidth{0.100000\du}
        \pgfsetdash{}{0pt}
        \pgfsetdash{}{0pt}
        \pgfsetbuttcap
        {
        \definecolor{dialinecolor}{rgb}{0.000000, 0.000000, 0.000000}
        \pgfsetfillcolor{dialinecolor}
        \definecolor{dialinecolor}{rgb}{0.000000, 0.000000, 0.000000}
        \pgfsetstrokecolor{dialinecolor}
        \draw (25.000000\du,16.000000\du)--(26.000000\du,18.000000\du);
        }
        \pgfsetlinewidth{0.100000\du}
        \pgfsetdash{}{0pt}
        \pgfsetdash{}{0pt}
        \pgfsetbuttcap
        {
        \definecolor{dialinecolor}{rgb}{0.000000, 0.000000, 0.000000}
        \pgfsetfillcolor{dialinecolor}
        \definecolor{dialinecolor}{rgb}{0.000000, 0.000000, 0.000000}
        \pgfsetstrokecolor{dialinecolor}
        \draw (26.000000\du,18.000000\du)--(35.000000\du,20.000000\du);
        }
        \pgfsetlinewidth{0.100000\du}
        \pgfsetdash{}{0pt}
        \pgfsetdash{}{0pt}
        \pgfsetbuttcap
        {
        \definecolor{dialinecolor}{rgb}{0.000000, 0.000000, 0.000000}
        \pgfsetfillcolor{dialinecolor}
        \definecolor{dialinecolor}{rgb}{0.000000, 0.000000, 0.000000}
        \pgfsetstrokecolor{dialinecolor}
        \draw (25.000000\du,16.000000\du)--(35.000000\du,20.000000\du);
        }
        \definecolor{dialinecolor}{rgb}{0.000000, 0.000000, 0.000000}
        \pgfsetstrokecolor{dialinecolor}
        \node[anchor=west] at (25.000000\du,15.800000\du){$v_{5}$};
        \definecolor{dialinecolor}{rgb}{0.000000, 0.000000, 0.000000}
        \pgfsetstrokecolor{dialinecolor}
        \node[anchor=west] at (26.000000\du,17.800000\du){$v_{3}$};
        \pgfsetlinewidth{0.100000\du}
        \pgfsetdash{}{0pt}
        \pgfsetdash{}{0pt}
        \pgfsetbuttcap
        {
        \definecolor{dialinecolor}{rgb}{0.000000, 0.000000, 0.000000}
        \pgfsetfillcolor{dialinecolor}
        \definecolor{dialinecolor}{rgb}{0.000000, 0.000000, 0.000000}
        \pgfsetstrokecolor{dialinecolor}
        \draw (26.000000\du,18.000000\du)--(20.000000\du,20.000000\du);
        }
        \end{tikzpicture}
    \caption{The elementary syzygy of Mori-Mukai No. 4.5}
    \label{figure:4.5}
    \end{figure}
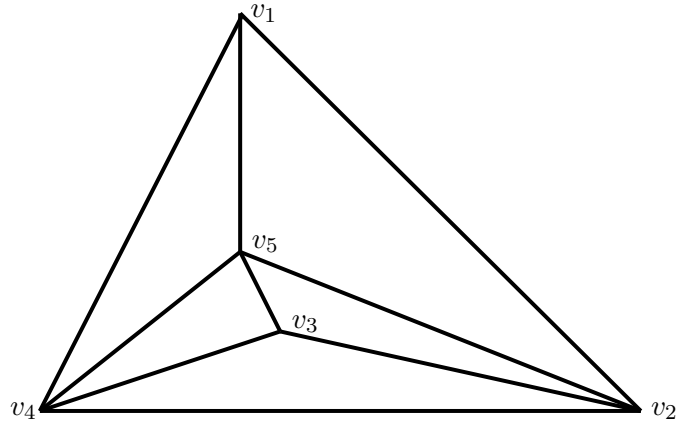
    The correspondence of vertices and lines is given as:
    \begin{itemize}
    \item $v_{1}$: The smooth Fano variety $X_{3.21}$.
    \item $v_{2}$: The composition $X \rightarrow \mathbb{P}^1 \times \mathbb{F}_{1} \xrightarrow{p_{1}} \mathbb{P}^1$.
    \item $v_{3}$: The smooth Fano variety $X_{3.31}$.
    \item $v_{4}$: A Fano variety $Y_{1}$ with an ordinary double point, which deforms to $X_{2.26}$.
    \item $v_{5}$: The smooth Fano variety $\mathbb{P}^1 \times \mathbb{F}_{1}$.
    \item $\langle v_{1},v_{2} \rangle$: The composition $X_{3.21} \rightarrow \mathbb{P}^1 \times \mathbb{P}^2 \rightarrow \mathbb{P}^1$.
    \item $\langle v_{1},v_{4} \rangle$: A Fano variety $Y_{2}$ with an ordinary double point, which deforms to $B_{5}$.
    \item $\langle v_{1},v_{5} \rangle$: The smooth Fano variety $\mathbb{P}^{1} \times \mathbb{P}^2$.
    \item $\langle v_{2},v_{3} \rangle$: The composition $X_{3.31} \rightarrow \mathbb{P}^1 \times \mathbb{P}^1 \rightarrow \mathbb{P}^1$.
    \item $\langle v_{2},v_{4} \rangle$: The other composition $X_{3.31} \rightarrow \mathbb{P}^1 \times \mathbb{P}^1 \rightarrow \mathbb{P}^1$.
    \item $\langle v_{2},v_{5} \rangle$: The first projection $\mathbb{P}^1 \times \mathbb{F}_{1} \rightarrow \mathbb{P}^1$.
    \item $\langle v_{3},v_{4} \rangle$: A Fano variety $Y_{3}$ with an ordinary double point, which deforms to $Q$.
    \item $\langle v_{3},v_{5} \rangle$: The composition $\mathbb{P}^1 \times \mathbb{F}_{1} \rightarrow \mathbb{F}_{1} \rightarrow \mathbb{P}^1$.
    \item $\langle v_{4},v_{5} \rangle$: The composition $\mathbb{P}^1 \times \mathbb{F}_{1} \rightarrow \mathbb{F}_{1} \rightarrow \mathbb{P}^2$.
    \end{itemize}
    The correspondence of faces follows from the above correspondence and the elementary syzygies of smooth Fano threefolds of Picard rank $\leq 3$.
\end{example}

\printbibliography

\end{document}